\documentclass[reqno,11pt]{amsart}
\usepackage{amsfonts,amsmath,amsthm, url}
\usepackage{amssymb}
\usepackage{graphics}
\usepackage{graphicx}
\usepackage{epsfig,color}
\usepackage{longtable}
\usepackage{stmaryrd}
\usepackage{tikz}
\usetikzlibrary{calc}
\usetikzlibrary{decorations.pathmorphing}
\usepackage{mathrsfs}

   \newtheorem{thm}{Theorem}
   \newtheorem{prop}{Proposition}
   \newtheorem{lem}{Lemma}[section]
   \newtheorem{cor}[lem]{Corollary}
   \newtheorem{rem}{Remark}[section]

\newcommand{\N}{\mathbb{N}}
\newcommand{\Z}{\mathbb{Z}}

\newcommand{\R}{\mathbb{R}}
\newcommand{\C}{\mathbb{C}}
\newcommand{\E}{\mathbb{E}}
\newcommand{\prob}{\mathbb{P}}

\newcommand{\re}{\mathop{\mathrm{Re}}\nolimits}
\newcommand{\im}{\mathop{\mathrm{Im}}\nolimits}

\newcommand{\supp}{\mathop{\mathrm{supp}}\nolimits}

\newcommand{\dd}{\mathrm{d}}

\newcommand{\bracknabla}{\langle\nabla\rangle}

\newcommand{\brackxi}
\makeatletter
    
    \@addtoreset{equation}{section}
\makeatother
\makeatletter
  \newcommand{\subsubsubsection}{\@startsection{paragraph}{4}{\z@}%
    {1.0\Cvs \@plus.5\Cdp \@minus.2\Cdp}%
    {.1\Cvs \@plus.3\Cdp}%
    {\reset@font\sffamily\normalsize}
  }
  \makeatother
\allowdisplaybreaks[4]

\makeatletter
  
  \@addtoreset{equation}{section}
\makeatother

\makeatletter
\let\@cleartopmattertags\relax
\newcommand\articleend
 {\enddoc@text
  \let\authors\@empty
  \let\contribs\@empty
  \let\xcontribs\@empty
  \let\toccontribs\@empty
  \let\addresses\@empty
  \let\thankses\@empty
  \newpage
 }
\let\@wraptoccontribs\wraptoccontribs
\makeatother

\begin{document}

\title[Non relativistic and ultra relativistic limits]{Non relativistic and ultra relativistic limits in 2d stochastic nonlinear damped Klein-Gordon equation}
\author[Reika Fukuizumi]{Reika Fukuizumi$^{\scriptsize 1}$}
\author[Masato Hoshino]{Masato Hoshino$^{\scriptsize 2}$}
\author[Takahisa Inui]{Takahisa Inui$^{\scriptsize 3}$} 



\maketitle

\begin{center} 
$^1$ Research Center for Pure and Applied Mathematics, \\
Graduate School of Information Sciences, Tohoku University,\\
Sendai 980-8579, Japan; \\
\email{fukuizumi@math.is.tohoku.ac.jp}
\end{center}

\begin{center} 
$^2$ Graduate School of Engineering Science, Osaka University,\\
Toyonaka, Osaka 560-8531,Japan;\\
\email{hoshino@sigmath.es.osaka-u.ac.jp}
\end{center}

\begin{center}
$^3$ Department of Mathematics, Graduate School of Science, \\
Osaka University, Toyonaka, Osaka 560-0043, Japan \& \\
Department of Mathematics, University of British Columbia, \\
1984 Mathematics Rd., Vancouver, Canada V6T1Z2;\\
\email{inui@math.sci.osaka-u.ac.jp}
\end{center}

\vskip 0.1 in
\noindent
{\bf Abstract}. We study the non relativistic and ultra relativistic limits in the two-dimensional nonlinear damped Klein-Gordon equation 
driven by a space-time white noise on the torus. 
In order to take the limits, it is crucial to clarify the parameter dependence in the estimates of solution. 
In this paper we present two methods to confirm this parameter dependence. One is the classical, simple 
energy method. Another is the method via Strichartz estimates.   

\vskip 0.1 in

%%%%%%%%%%%%%%%%%%%%%%% Introduction %%%%%%%%%%%%%%%%%%%%%%%%%%%%%%%%%%%
\section{Introduction} 

A stochastic force combined with a dissipation is used to model 
a temperature effect in the dynamics of  differential equations. 
Such so-called Langevin/Over damped Langevin dynamics are versatile in biology, chemistry, engneering, physics, and computer sciences. Traditionally, one assumes that the forces are given as the gradient of a potential, and that a fluctuation-dissipation relation holds between stochastic and dissipative forces. We are interested in the model proposed by the paper \cite{CK}, where the authors, adding noise and dissipation, 
study the $U(1)$-invariant relativistic complex field model in three dimensions,  
which serves to describe, in various limits, 
properties at finite temperatures of superfluid systems, superconductors of type II, nematic liquid crystals, 
as well as relativistic bosons at finite chemical potential. The paper \cite{CK} begins with 
a revisit on the well-known expansion and numerical methods to observe the equilibrium behavior  
of the three-dimensional $U(1)$ complex field, identifying the above mentioned various dynamic regimes 
in the thermal equilibrium. Namely, the authors found a statistical universality among those models of different physical backgrounds. 
The paper is then, as more important and main subject,  
devoted to the study of the vortex tangle in and out of  (but near) equilibrium. 
In this paper we are motivated to justify rigorously the numerical results obtained in the former part of the paper \cite{CK}. 
Remark that the restriction to the dimension three is essential only for the latter part by the reason that a phase transition 
should occur in the dimension larger than two.

More precisely, the under damped Langevin equation derived from the Lagrangian density for relativistic bosons 
with finite chemical potential in \cite{CK} reads the following damped nonlinear wave equation 
on $\mathbb{T}^d:=(\mathbb{R}/2\pi\mathbb{Z})^d$ driven by a complex-valued space-time white noise $\xi$. 
\begin{eqnarray} \label{eq:phys}
\frac{1}{c^2} \partial_t^2 \psi +(\mu-i\gamma) \partial_t \psi  -\Delta \psi  -(\nu- |\psi|^2) \psi  =\sqrt{\mu \mathcal{T}}\xi, \\ \nonumber 
\langle \xi(t,x) \rangle =0, \quad \langle \xi^{\ast}(t,x)\xi(s,y) \rangle =2\delta(t-s) \delta(x-y),
\end{eqnarray}
with $\mu, \gamma>0$, $\nu\ge 0$ and $c>0$. Here, $\mathcal{T}>0$ is the temperature.
In the ultra relativistic and non relativistic limits, i.e. $\gamma\to 0$ and $c \to +\infty$ respectively, 
this Langevin equation approaches to 

\begin{equation*}%\label{eq:ultra}
\frac{1}{c^2} \partial_t^2 \psi + \mu \partial_t \psi  -\Delta \psi  -(\nu- |\psi|^2) \psi  =\sqrt{\mu \mathcal{T} }\xi
\end{equation*}
and 
\begin{equation*} %\label{eq:nonrela}
(\mu-i\gamma) \partial_t \psi  -\Delta \psi  -(\nu- |\psi|^2) \psi  =\sqrt{\mu \mathcal{T}}\xi,
\end{equation*}
which are
%respectively. Eq.(\ref{eq:ultra}) and Eq.(\ref{eq:nonrela}) are 
known as the Goldstone and the Gross-Pitaevskii 
models, respectively.
The latter describes the dynamics of gaseous Bose-Einstein condensates (\cite{griffin}). 
On the other hand, the former describes the dynamics of the Mott insulator phase with integer fillings (\cite{aa}). 
 
By the numerical simulations, it is observed that 
all these models have the same statistical quantities around the equilibrium.
%and the relaxation rate to the equilibrium as $t \to \infty$ seems similar between the models (\ref{eq:phys}) and (\ref{eq:nonrela}). 
%We will clarify rigorously in this paper the reasons for those simulation results in the case of $d=2$. 
%More pricisely, we will see that the equilibrium does not depend on the parameters, neither on the systems above except $\nu$, 
%and we will prove the existence of solutions to those models, 
%also the convergence of solutions in the non relativistic and in the ultrarelativistic limit. 
What they call in \cite{CK} the equilibrium, can be described 
by the Gibbs measure $\rho$, which is written formally as: 
\begin{align}\label{eq:GibbsdependsonT}
\rho(d\psi d\phi) = \Gamma^{-1} e^{-\frac{H(\psi, \phi)}{ \mathcal{T} } } d\psi d\phi,
\end{align}
with
$$H(\psi, \phi)=\frac{1}{2 c^2} \int |\phi|^2 dx + V(\psi),$$
and
$$V(\psi) =\frac{1}{2} \int |\nabla \psi|^2 dx - \frac{\nu}{2} \int |\psi|^2 dx + \frac{1}{4} \int |\psi|^{4} dx,$$
where $\Gamma$ is a normalizing constant. This $\rho$ depends on the parameter $c$, but in fact we will see in Proposition \ref{mainprop:GWP} that 
this dependence can be removed by the change of variables $\phi\mapsto c\phi$,
%the equilibrium does not depend on the parameters except $\nu$, neither on the equations above, 
which, we believe, may explain the result of the numerical simulation.   

We consider this equilibrium in the case of $d=2$, setting $\mathcal{T}=2$, in this paper. 
Unfortunately, we have to modify (\ref{eq:phys}) and replace the term $-\Delta$ by $-\Delta+1$ 
and set $\nu=0$ in order to keep the positive definiteness of the linear part to define the corresponding Gaussian measure. 
If not, we have to remove the Fourier zero mode to take into account the degenerate direction, which makes the arguments more complex.
Moreover, the $\rho$ is a priori not well-defined, since $L^4$ is not in the support of $\rho$, 
but at this point we may give a sense using a renormalization technique as has been widely used by now. 

Existence of solutions of all those models and the construction of the Gibbs measure have been 
established in \cite{GKO,GKOT, ORT,H2,HIN,Tr,M}. 
In the present paper, our interest is mainly in justifying the both limits $\gamma \to 0$ and $c \to +\infty$. 
The convergence as $c \to +\infty$ and the independence of Gaussian measure on $c$ were already justified in \cite{FC} 
under the name of Smoluchowski-Kramers approximation in the case of $d=1$ with a space-time white noise, 
and for $d \ge 2$ with a colored noise under the Dirichlet boundary condition. See for example \cite{CX} and references therein 
for generalizations of this Smoluchowski-Kramers approximation issue.
In order to show the convergence, 
it is required to know the parameter dependence in the energy or Strichartz estimates. 
In this paper we present both methods to confirm this parameter dependence. One is the classical, simple energy method. Another is the method via Strichartz estimates. In particular, 
to our best knowledge, uniform Strichartz estimates with respect to the non relativistic limit parameter on the torus have not been known.
%not known. 
We thus believe that its proof  itself is of interest.
%\textcolor{red}{The convergence of the systems to the equilibrium will be also proved,  
%since Poincar\'{e} inequality which gives in general an optimal convergence rate, is not known in 2d up to our knowledge, 
%we only use the convergence result in \cite{TW} making use of the limit.}
Finally we remark that the uniqueness of the invariant measures was investigated 
in \cite{T} for the case of $d=1$ with space-time white noise, 
and in \cite{FT} in the case $d=2$ for a slightly more regular noise than space-time white noise.   
The case $d=2$ with space-time white noise is open, but we expect the uniqueness of the Gibbs measure. 
%The convergence of the law to the Gibbs measure: to be announced, too.}  

%%%%%%%%%%%%%%%%%%%%%%%%%%%%%%%%%%%%%%%%%%%%%%%%%%%
\section{Main results} 
%%%%%%%%%%%%%%%%%%%%%%%%%%%%%%%%%%%%%%%%%%%%%%%%%%%%%%%%%%%%

In this section, we precisely mention our mathematical results on 
the equation explained in the previous section, setting $\varepsilon=\frac{1}{c} \in (0,1]$ in (\ref{eq:phys}). 
As a damping coefficient we consider more generally $\alpha\in\mathbb{C}$ such that $\re(\alpha)>0$, $\im(\alpha) \ne 0$. 
For each $\varepsilon\in(0,1]$, we consider the damped Klein-Gordon equation with an exterior force $f(t,x)$:
\begin{align}\label{eq:dampedwave_f}
\left\{
\begin{aligned}
&\varepsilon^2\partial_t^2u_{\varepsilon,\alpha}+2\alpha\partial_tu_{\varepsilon,\alpha} +(1-\Delta) u_{\varepsilon,\alpha}=f, &t>0,~ x\in\mathbb{T}^d,\\
&(u_{\varepsilon,\alpha},\varepsilon\partial_tu_{\varepsilon,\alpha})|_{t=0}=(\phi_0,\phi_1), &x\in\mathbb{T}^d.
\end{aligned}
\right.
\end{align}
The spatial dimension $d$ can be any $d\ge 1$ for the moment.   
\vspace{5mm}

We begin with the notation.
\begin{itemize} 
\setlength{\parskip}{1mm}
\setlength{\itemsep}{1mm}
\item To clarify the dependence on the parameter $\alpha$ in the estimates we will encounter, we use the following notation:
$$
\mathbb{C}_+=\{\alpha\in\mathbb{C}\, ;\, \re(\alpha)>0\},\qquad
\mathbb{C}_+\setminus(0,\infty)=\{\alpha\in\mathbb{C}_+\, ;\, \im(\alpha)\neq0\}.
$$
\item We denote by $L^2(\mathbb{T}^d)$ the complex-valued Lebesgue space on the torus 
with the inner product $\langle f,g\rangle:=\int_{\mathbb{T}^d}f(x)\overline{g(x)}dx$.
$C^\infty(\mathbb{T}^d)$ is the space of complex-valued smooth functions and denote by $\mathcal{S}'(\mathbb{T}^d)$ its dual.
We extend the inner product $\langle\cdot,\cdot\rangle$ to the paring of $\mathcal{S}'(\mathbb{T}^d)$ and $C^\infty(\mathbb{T}^d)$.
\item For $k\in\mathbb{Z}^d$ and $x\in\mathbb{T}^d$, we define $e_k(x)=\frac1{(2\pi)^{d/2}}e^{ik\cdot x}$.
Fourier transform $\mathcal{F}$ of $f\in\mathcal{S}'(\mathbb{T}^d)$, $\mathcal{F}f(k)$, which is sometimes denoted by $\hat{f}(k)$, 
and the inverse transform of a function $g:\mathbb{Z}^d\to\mathbb{C}$ are defined by
$$
\mathcal{F}f(k)=\hat{f}(k):=\langle f,e_k \rangle,\qquad
\mathcal{F}^{-1}g(x):=\sum_{k\in\mathbb{Z}^d}g(k)e_k(x).
$$
For some appropriate function $\varphi:\mathbb{Z}^d\to\mathbb{C}$, the Fourier multiplier is defined by
$$
\varphi(\nabla)f(x):=\mathcal{F}^{-1}(\varphi\mathcal{F}f)(x)=\sum_{k\in \Z^d} \varphi(k) \hat{f}(k) e_k(x).
$$
Thus, denoting  by ${\bf 1}_A$ the indicator function on $A$, 
${\bf 1}_{\{|\nabla| \le c\}}$ means that 
$${\bf 1}_{ \{|\nabla| \le c\}} f(x) = \sum_{k\in \Z^d} {\bf 1}_{\{ |k| \le c \}} \hat{f} (k) e_k(x) =\sum_{k\in \Z^d: |k| \le c} \hat{f} (k) e_k(x).$$
We frequently use the projector $\Pi_N  ={\bf 1}_{\{|\nabla|\le N\}}  : L^2 \to L^2$ in this paper. 
\item $\{\chi_j\}_{j\in\mathbb{Z}}$ is an inhomogeneous dyadic decomposition of unity, that is,
\begin{itemize}
\setlength{\parskip}{1mm}
\setlength{\itemsep}{1mm}
\item $\chi_{-1}$ and $\chi_0$ are radial and smooth functions on $\mathbb{R}^d$. $\chi_{-1}$ is supported in a ball, and $\chi_0$ is supported in an annulus.
\item $\chi_j=\chi_0(2^{-j}\cdot)$ for $j\ge0$. $\chi_i$ and $\chi_j$ are disjointly supported if $|i-j|\ge2$.
\item $\sum_{j\ge-1}\chi_j=1$.
\end{itemize}
Define $\Delta_j:=\chi_j(\nabla)$.
\item For any $s\in\mathbb{R}$ and $q,r \in [1, \infty]$, we denote the inhomogeneous Sobolev norms by 
\begin{align*}
\|f\|_{W^{s,q}}:=\big\| (1-\Delta)^sf \big\|_{L^q}
\end{align*}
and the inhomogeneous Besov norms by 
\begin{align*}
\|f\|_{B_{q,r}^s}:=\Big\|\Big(2^{js}\|\Delta_jf\|_{L^q}\Big)_{j\ge-1}\Big\|_{\ell^r}. 
\end{align*}
We write $W^{s,2}=H^s$ and define $\mathcal{H}^{s}=H^{s} \times H^{s-1}$. 

\item If $E$ is a Banach space with the norm $|\cdot|_{E}$, then for $T>0$, and $p\in[1,\infty]$, 
$L^p(0,T;E)$ is the space of strongly Lebesgue measurable functions $v$ from $[0,T]$ into $E$ such that 
$t \mapsto |v(t)|_{E}$ is in $L^p(0,T)$. We sometimes abbreviate this space as $L_T^p E$. 
Similarly, we define the space $C([0,T];E)$ and abbreviate as $C_TE$.
When $E$ is a function space on $\mathbb{T}^d$ such as Sobolev or Besov spaces, we sometimes emphasize the spatial variable $x$ and write $L_T^pE_x$ or $C_TE_x$.
\item Finally, %$\Pi_N {\color{blue} ={\bf 1}_{\{|\nabla|\le N\}} } : L^2 \to L^2$ denotes the standard projector. 
the conjugate exponent of $p>1$ is denoted by $p'$, i.e. $\frac{1}{p'}+\frac{1}{p}=1$, and 
$\langle x\rangle:=\sqrt{1+|x|^2}$ for $x\in\mathbb{R}^d$ or $\mathbb{Z}^d$.
\end{itemize}
\vspace{5mm}

We will mainly focus on the non relativistic limit problem presenting the detailed analysis, 
and then we will mention our results briefly on the ultra relativistic limit which are similarly obtained.        
Hence, the first step is to derive the uniform energy (more generally Strichartz) estimate in $\varepsilon$ for (\ref{eq:dampedwave_f}). 
In our case, the uniform energy estimate can be proved simply  
using Fourier series expression of the solution and it is enough for our aim (see Theorem \ref{mainthm_energy}).   
However, as a by-product during this study we also obtained an $\varepsilon$-uniform Strichartz estimate on $\mathbb{T}^d$, which 
generalizes the energy estimates in
Theorem \ref{mainthm_energy}, thus we will address the statement of this uniform Strichartz estimate below (see Theorem \ref{mainthm}) 
and a brief proof of it in Section \ref{section7}.   

Note that once the Strichartz estimate is available for nonlinear wave equations on $\R^d$, 
the same estimate on $\mathbb{T}^{d}$ follows from the finite propagation property.
%it follows from the finite propagation property the same estimate on $\mathbb{T}^{d}$. 
We may then prove easily the local well posedness for (\ref{eq:dampedwave_f}) as in \cite{GKO}. However, our purpose is 
to show the $\varepsilon$-limit in the equation, thus it requires to make precise the $\varepsilon$ dependence of all the constants, in particular of the time interval. 
  
We write the solution of  (\ref{eq:dampedwave_f}) in the mild form.
\begin{align} \label{mildform}
u_{\varepsilon, \alpha}(t)=e^{t\lambda_{\varepsilon, \alpha}^+(\nabla)}\phi_
{\varepsilon, \alpha}^+
+e^{t\lambda_{\varepsilon, \alpha}^-(\nabla)}\phi_{\varepsilon, \alpha}^-
+\int_0^t\left(e^{(t-t')\lambda_{\varepsilon, \alpha}^+(\nabla)}f_{\varepsilon, \alpha}^+(t')
+e^{(t-t')\lambda_{\varepsilon, \alpha}^-(\nabla)}f_{\varepsilon, \alpha}^-(t')\right)dt',
\end{align}
where
%where,
%The quadratic equation $\varepsilon^2\lambda^2+2\alpha\lambda+|\xi|^2=0$ has exactly two solutions
$$
\lambda_{\varepsilon, \alpha}^\pm(\nabla)=\frac{-\alpha\pm\sqrt{\alpha^2-\varepsilon^2\langle\nabla\rangle^2}}{\varepsilon^2},
$$
with the square root $\sqrt{\cdot}$ defined by
$$
\sqrt{e^{i\theta}}=e^{i\frac\theta2},\quad\theta\in(-\pi,\pi],
$$
 and 
$$
\phi_{\varepsilon, \alpha}^\pm
=\frac{\mp\varepsilon^2\lambda_{\varepsilon, \alpha}^\mp(\nabla)\phi_0 \pm\varepsilon\phi_1}{2\sqrt{\alpha^2-\varepsilon^2 \langle\nabla\rangle^2}},\quad
f_{\varepsilon, \alpha}^\pm(t)=\pm\frac1{2\sqrt{\alpha^2-\varepsilon^2 \langle\nabla\rangle^2}}f(t).
$$  
\vspace{3mm}  

\begin{thm}\label{mainthm_energy} Let $d\ge 1$. For any $\sigma\in\mathbb{R}$,
\begin{align*}
\|(u_{\varepsilon, \alpha}, \varepsilon \partial_t u_{\varepsilon, \alpha})\|_{L_T^\infty \mathcal{H}_x^\sigma(\mathbb{T}^d)}
\lesssim_{\alpha} \|\phi_0\|_{H^\sigma(\mathbb{T}^d)}+\|\phi_1\|_{H^{\sigma-1}(\mathbb{T}^d)}+\|f\|_{L_T^2 H_x^{\sigma-1}(\mathbb{T}^d)},
\end{align*}
where the implicit proportional constants are locally bounded function of $\alpha\in\mathbb{C}_+\setminus(0,\infty)$.
\end{thm}
\vspace{3mm}

The proof of Theorem \ref{mainthm_energy} can be found in Section \ref{section:LWP}. The non relativistic limit $\varepsilon \to 0$ in $\R^d$ can be proved for $d=1,2$ by energy arguments as in Theorem \ref{mainthm_energy} (see \cite{Tsu, N, N1} for the results in the case of $d=1,2$) and, however, to extend these results to the case $d \ge 3$, it takes some ingenuity since the nonlinearity may not be treated simply by Sobolev embeddings (see \cite{N1, MACHI} for $d=3$ under restrictions on the nonlinearity). 
In \cite{MNO, MNO1} the authors solve this problem under no restriction on the nonlinearity by use of the Strichartz estimate and decomposing the solution into the low frequency part and the high frequency part  (see also \cite{IM}). More precisely it should be noticed
that by the kernel expression above the damped wave equation behaves as a heat equation for the low frequency part and a wave equation for the high frequency part (\cite{I}).  

%\textcolor{blue}{The proof of Theorem \ref{mainthm_energy} can be found in Section \ref{section:LWP}.} \textcolor{blue}{Histrically, 
%the non relativistic limit $\varepsilon \to 0$ in $\R^d$ was considered for $d=1,2$ by \cite{N, N1, MACHI} using $L^2$ based arguments 
%as in Theorem \ref{mainthm_energy}. However to extend these {\color{blue}results} to the case $d \ge 3$ it takes some ingenuity since the nonlinearity 
%may not be treated simply by Sobolev embeddings. In \cite{MNO, MNO1} the authors firstly made use of the Strichartz estimate to get over 
%this problem (see also \cite{IM}) decomposing the solution into the low frequency part and the high frequency part. More precisely 
%it should be noticed } that by the kernel expression above the damped wave equation behaves 
%as a heat equation for the low frequency part and a wave equation for the high frequency part (\cite{I}).  
We thus decompose the solution $u_{\varepsilon, \alpha}$ into low frequency part and high frequency part:
fix a radial smooth function $I:\mathbb{R}^d\to[0,1]$ such that
\begin{equation} \label{def:cutoff}
I(\xi)=
\begin{cases}
1&|\xi|\le1,\\
0&|\xi|\ge2,
\end{cases}
\end{equation}
and we decompose $u_{\varepsilon, \alpha}=u_{\varepsilon, \alpha}^<+u_{\varepsilon, \alpha}^>$, where
$$
u_{\varepsilon, \alpha}^<=I(\varepsilon\nabla)u_{\varepsilon, \alpha},\quad
u_{\varepsilon, \alpha}^>=(1-I(\varepsilon\nabla))u_{\varepsilon, \alpha}.
$$
\vspace{3mm}

Remark that we need this smooth cut-off $I$ to derive the Strichartz estimate since we are required  
to treat some $L^p$ ($p \ge 1$) based estimates, while for the energy estimates only the $L^2$ based estimates are used, 
thus the non smooth cut-off ${\bf 1}$ is enough. 
Moreover, we can choose the cut-off of the form $I(\varepsilon\nabla)$, not $I(\varepsilon\langle\nabla\rangle)$, by the transform of equation as in the beginning of Section \ref{section7}.

 The problem here is that the scaling argument frequently used in $\R^d$ (e.g. \cite{N, MNO, IM}) to have such $\varepsilon$-uniform Strichartz estimates does not work immediately  
if we consider $\mathbb{T}^d$ since after such scaling the period also exhibits an $\varepsilon$-dependence.    
Therefore, we are obliged to repeat the proof of classical Strichartz estimates to verify the dependence on $\varepsilon$ 
regarding a function in $C^\infty(\mathbb{T}^d)$ as a periodic function on $\R^d$. 
\vspace{3mm}

The following result is the first $\varepsilon$ uniform Strichartz estimates on $\mathbb{T}^d$ as far as we know.  
In this paper only the case $d\le 2$ of this estimate is useful since we are interested in the stochatic eqution (\ref{eq:dampedwave_W}) below which has a meaning for the moment only in $d \le 2$, but we expect that the following Theorem 2 will be helpful to consider 
the non relativistic limit in the high dimensional case on $\mathbb{T}^d$ in the future when  (\ref{eq:dampedwave_W}) will be made sense. 

\begin{thm}\label{mainthm} Let $d\ge 1$. 
For any $r\in[1,\infty]$, $q_1\ge q_2\in[1,\infty]$, $\sigma\in\mathbb{R}$, and $s\in[0,2(1+\frac1{q_1}-\frac1{q_2})]\cap[0,2)$,
\begin{align}\label{thm:u<}
\|u_{\varepsilon, \alpha}^<\|_{L_T^{q_1}B_{r,2}^{\sigma+s}(\mathbb{T}^d)}&
\lesssim_{\alpha} \|\phi_0\|_{B_{r,2}^{\sigma+s}(\mathbb{T}^d)}+\|\phi_1\|_{B_{r,2}^{\sigma+s-1}(\mathbb{T}^d)}
+\|f\|_{L_T^{q_2}B_{r,2}^{\sigma}(\mathbb{T}^d)}.
\end{align}
Assume that $(q_k,r_k)\in[2,\infty]^2$ ($k=1,2$) satisfies
\begin{align*}%\label{admissible}
\frac{1}{m q_k}=\frac12-\frac{1}{r_k}
\end{align*}
for $ \frac{d-1}{2}< m \leq \infty$ 
and define $s_k=\frac{d+1}2\left(\frac12-\frac1{r_k}\right)$.
For any $\sigma \in\mathbb{R}$ and $ s \in[0,1]$, it holds that
\begin{align}\label{thm:u>}
\begin{aligned}
\|u_{\varepsilon,\alpha}^>\|_{L_T^{q_1}B_{r_1,2}^{\sigma} (\mathbb{T}^d)}
&
\lesssim_{\alpha,T} \varepsilon^{\delta_1}\|\phi_0\|_{H^ {\sigma +s_1}(\mathbb{T}^d)}
+\varepsilon^{\delta_1}\|\phi_1\|_{H^{\sigma +s_1-1}(\mathbb{T}^d)}\\
&\quad+\varepsilon^{\delta_1+\delta_2- s }\|f\|_{L_T^{q_2'}B_{r_2',2}^{ \sigma +s_1+s_2-  s }(\mathbb{T}^d)},
\end{aligned}
\end{align}
where $\delta_k =\delta(q_k,r_k)=\frac2{q_k}-\frac{d-1}2\left(\frac12-\frac1{r_k}\right)$.
\end{thm}
\vspace{3mm}

Theorem \ref{mainthm_energy} is a particular case of Theorem \ref{mainthm}: $(q_1,q_2,r,r_1,r_2,s)=(\infty,2,2,2,2,1)$.
This estimate can be extended more generally to, for example, the endpoint case or 
the case of $ m< \frac{d-1}{2}$, but this is not our objective and we omit it  (refer to \cite{I, IM}).  The proof of Theorem \ref{mainthm} can be found in Section \ref{section7}.
\vspace{3mm}

From now we restrict ourselves to the two dimensional case, and we pay attention to the equation in purpose, namely 
\begin{align}\label{eq:dampedwave_W}
\left\{
\begin{aligned}
&\varepsilon^2\partial_t^2 \Psi_{\varepsilon,\alpha}+ 2\alpha \partial_t \Psi_{\varepsilon,\alpha}+(1-\Delta) \Psi_{\varepsilon,\alpha}
+|\Psi_{\varepsilon,\alpha}|^{2n} \Psi_{\varepsilon,\alpha} =2\sqrt{\re(\alpha)} \partial_t W, &t>0,~ x\in \mathbb{T}^2,\\
&(\Psi_{\varepsilon,\alpha},\varepsilon\partial_t \Psi_{\varepsilon,\alpha})|_{t=0}=(\psi,\phi), &x\in\mathbb{T}^2,
\end{aligned}
\right.
\end{align}
where $n\in\mathbb{N}$ and $W$ is the cylindrical Wiener process as  
\begin{equation*} %\label{WienerProcess}
W (t, x)= \sum_{k\in \Z^2} (\beta_{k,R} (t) + i\beta_{k,I}(t)) e_k (x).
\end{equation*}
Here, $(\beta_{k,R} (t))_{t \ge 0}$ and $(\beta_{k,I} (t))_{t \ge 0}$ are sequences of independent real-valued 
Brownian motions on the stochastic basis $(\Omega, \mathcal{F}, \prob, (\mathcal{F}_t)_{t\ge 0}).$ 
The coefficient balance between the dissipation term $2\alpha \partial_t \Psi_{\varepsilon,\alpha}$ and 
the noise term  $2\sqrt{\re(\alpha)} \partial_t W$ makes the corresponding Gibbs measure be independent of $\alpha$. 
If we consider general coefficients $2\sqrt{\re(\alpha)\mathcal{T}}$ with $\mathcal{T}>0$, then the corresponding Gibbs measure depends on $\mathcal{T}$ as in \eqref{eq:GibbsdependsonT}.
%the ratio of the real part of the dissipation coefficient and the half of the squared coefficient of the noise.
In all what follows, the notation $\E$ stands for the expectation with respect to $\prob$. 
For a probability measure $\mu$ defined on $H^{s}$, integration with respect to $\mu \otimes \mathbb{P}$ denoted by $\E_{\mu}$.
\vspace{3mm}

As was already proved in \cite{GKO, ORT}, due to the space-time white noise, 
the solution of (\ref{eq:dampedwave_W}) have negative space regularity, 
and thus the nonlinear term $|\Psi|^{2n} \Psi$ is ill-defined.  
In order to make sense of this term, 
we will proceed as in \cite{DPD, GKO, ORT}, use the Wick product and renormalize the nonlinear term. 
%\textcolor{red}{Note that a phenomenon of triviality is known to hold for (\ref{eq:dampedwave_W}) without renormalization (\cite{AHR,OOR}). }
\begin{align}\label{eq:dampedwave_Wick}
\left\{
\begin{aligned}
&\varepsilon^2\partial_t^2 \Psi_{\varepsilon,\alpha}+2\alpha \partial_t \Psi_{\varepsilon,\alpha}+(1-\Delta) \Psi_{\varepsilon,\alpha}
+:\Psi_{\varepsilon,\alpha}^{n+1}\overline{\Psi_{\varepsilon,\alpha}}^n: \ =2\sqrt{\re(\alpha)} \partial_t W, &t>0,~ x\in \mathbb{T}^2,\\
&(\Psi_{\varepsilon,\alpha},\varepsilon\partial_t \Psi_{\varepsilon,\alpha})|_{t=0}=(\psi,\phi), &x\in\mathbb{T}^2. 
\end{aligned}
\right.
\end{align}
%\begin{align}\label{eq:dampedwave_Wick}
%\left\{
%\begin{aligned}
%&\varepsilon^2\partial_t^2 \Psi_\varepsilon+\alpha \partial_t \Psi_\varepsilon-\Delta \Psi_\varepsilon
%+:|\Psi|^{2n} \Psi: =\sqrt{2\re(\alpha)} \partial_t W, &t>0,~ x\in \mathbb{T}^2,\\
%&(\Psi_\varepsilon,\partial_t \Psi_\varepsilon)|_{t=0}=(\psi,\phi), &x\in\mathbb{T}^2,
%\end{aligned}
%\right.
%\end{align}
The notation 
$:u^{n+1}\bar{u}^n:$ is the complex Wick product %of the $n$-th power nonlinearity 
defined below (see Appendix \ref{app:Wick}). 

Writing the solution $\Psi_{\varepsilon,\alpha}=U_{\varepsilon,\alpha}+Z_{\varepsilon,\alpha}$ with the stationary solution $Z_{\varepsilon,\alpha}$
%\begin{equation} \label{eq:Zinfini}
%Z_{\varepsilon}(t)=\sqrt{2\re(\alpha)}\int_{-\infty}^t  
%\left( \frac{e^{(t-t')\lambda_\varepsilon^+(\nabla)} }{ {2\sqrt{\alpha^2-\varepsilon^2|\nabla|^2}}} 
%-\frac{e^{(t-t')\lambda_\varepsilon^-(\nabla)}}{  2\sqrt{\alpha^2-\varepsilon^2|\nabla|^2}}  \right) dW(t'),  
%\end{equation}
%which is the stationary solution 
for the linear stochastic equation
\begin{equation} \label{eq:Z}
\varepsilon^2\partial_t^2 Z_{\varepsilon,\alpha}+ 2\alpha \partial_t Z_{\varepsilon,\alpha}+(1-\Delta) Z_{\varepsilon,\alpha}
=2\sqrt{\re(\alpha)} \partial_t W,  
\end{equation}
we find out the following random partial differential equation for $U_{\varepsilon,\alpha}$:
\begin{equation}\label{eq:u}
\left\{
\begin{aligned}
&\varepsilon^2 \partial_t^2 U_{\varepsilon,\alpha} +2 \alpha \partial_t U_{\varepsilon,\alpha} +(1-\Delta) U_{\varepsilon,\alpha}
+ :(U_{\varepsilon,\alpha}+Z_{\varepsilon,\alpha})^{n+1}(\overline{U_{\varepsilon,\alpha}+Z_{\varepsilon,\alpha}})^n:\ =0,\\
%~  t>0,~ x\in \mathbb{T}^2,\\
&(U_{\varepsilon,\alpha},\varepsilon\partial_tU_{\varepsilon,\alpha})|_{t=0}=(u,v)
\end{aligned}
\right.
\end{equation}
with $(u,v)=(\psi,\phi)- (Z_{\varepsilon,\alpha},\varepsilon \partial_t Z_{\varepsilon,\alpha})|_{t=0}$.
\vspace{3mm}

Applying Theorem \ref{mainthm_energy} to (\ref{eq:u}), 
we have the local existence of solution. 
The proof of the following result can be found in Section \ref{section:LWP}.

\begin{cor}\label{cor2 of mainthm}
Fix any $T>0$. Let $\sigma<1$ be sufficiently close to $1$ and let $\delta=1-\sigma$.
For any compact subset $K$ of $\mathbb{C}_+\setminus(0,\infty)$, the equations (\ref{eq:u}) parametrized 
by $\varepsilon>0$ and $\alpha\in K$ are uniformly well-posed:
 there exist a random time $T^*(\omega) >0$, and a unique solution   
$$
U_{\varepsilon,\alpha} \in  C([0, T^*); H^\sigma) \cap C^1([0, T^*); H^{\sigma-1}), \quad \mbox{a.s.}
$$
Here, $T^*$ depends only on $\|(u,v)\|_{\mathcal{H}^{\sigma}}$ and $\sum_{k+\ell\le 2n+1}
\|:Z_{\varepsilon, \alpha}^k \overline{Z_{\varepsilon, \alpha}}^\ell:\|_{L^{\infty}_T W_x^{-\delta, \infty}}$. 
\end{cor}

\begin{rem} \label{rem:cor21} As we will see just below, we will consider the limit $\varepsilon\to 0$ (non relativistic limit) 
or $\mathrm{Im} (\alpha) \to 0$ (ultra relativistic limit). Therefore, 
the statement of  Corollary \ref{cor2 of mainthm} is precisely as follows: 
for each $\varepsilon>0$  and  $\alpha \in K$, there exist a random maximal existence time 
$T^*_{\varepsilon, \alpha}(\omega) >0$, and a unique solution   
$$
U_{\varepsilon,\alpha} \in  C([0, T^*_{\varepsilon, \alpha}); H^\sigma) \cap C^1([0, T^*_{\varepsilon, \alpha}); H^{\sigma-1}), \quad \mbox{a.s.},
$$
and if we set  $T^*:=\mathrm{liminf}_{\varepsilon \to 0} T^*_{\varepsilon, \alpha},$ 
(or $T^{**}:=\mathrm{liminf}_{\mathrm{Im}(\alpha) \to 0} T^*_{\varepsilon, \alpha},$ )
then $T^*(\omega)>0$ ($T^{**}(\omega)>0$) a.s. and  
$$
U_{\varepsilon,\alpha} \in  C([0, T^* (T^{**})); H^\sigma) \cap C^1([0, T^* (T^{**})); H^{\sigma-1}), \quad \mbox{a.s.}
$$
\end{rem}

The Gibbs measure has been constructed in \cite{ORT}, and can be made sense, again with the help of the renormalization:
$$ \rho_{2n+2}(d\psi d\phi) = \Gamma^{-1} e^{-H(\psi,\phi)} d\psi d\phi$$  
where
%where,  
$$H(\psi, \phi)=\frac{1}{2} \int |\phi|^2 dx + V(\psi),$$
and
$$V(\psi) =\frac{1}{2} \int |\nabla \psi|^2 dx + \frac12 \int |\psi|^2 dx + \frac{1}{2n+2} \int :|\psi|^{2n+2} : dx,$$
and $\Gamma$ is the normalizing constant. 
Using the Gaussian measures $\mu_0 = \mathcal{N}(0, (1-\Delta)^{-1})$ and $\mu_1=\mathcal{N}(0,I)$ we may write 
\begin{eqnarray*}
&& \rho_{2n+2}(d\psi d\phi) = \Gamma^{-1} e^{-H(\psi,\phi)} d\psi d\phi
= \Gamma^{-1} e^{-\frac{1}{2n+2} \int :|\psi|^{2n+2} : dx }  \mu(d\psi d\phi) \\
&& \hspace{5cm}= \Gamma^{-1} e^{-\frac{1}{2n+2} \int :|\psi|^{2n+2} : dx }  \mu_0(d\psi)\otimes\mu_1(d\phi).
\end{eqnarray*}
It is known that 
%$\mu_0 = \mathcal{L}(Z_0)$ 
$\mu_0$ is a stationary measure of $Z_{0,\alpha}$
and $\supp\rho_{2n+2}=\supp\mu\subset\mathcal{H}^{-\delta}$ for any $\delta>0$. Remark that the Gibbs measure does not depend on 
$\varepsilon>0$, nor on $\alpha \in \C$. 
\vspace{3mm}

We make use of the Gibbs measure to globalize the solution obtained above. 

\begin{prop}\label{mainprop:GWP}
There exists a measurable set $\mathcal{O}_{\varepsilon, \alpha} \subset \mathcal{H}^{-\delta}$ such that $\rho_{2n+2}(\mathcal{O}_{\varepsilon, \alpha})=1$ 
and for $(\psi, \phi) \in \mathcal{O}_{\varepsilon, \alpha}$ the solution of \eqref{eq:dampedwave_Wick} exists globally a.s..
Moreover, the measure $\rho_{2n+2}$ is invariant for this solution.
\end{prop}
\vspace{3mm}

The proof of Proposition \ref{mainprop:GWP} can be found in Section \ref{section:GE}.
Finally we justify the non-relativistic limit. We may formally expect that when $\varepsilon$ goes to $0$, $\Psi_{\varepsilon,\alpha}$ converges 
to $\Psi_\alpha$, which is the solution of the stochastic complex Ginzburg-Landau equation: 
\begin{equation} \label{eq:CGL}
\left\{
\begin{aligned}
&2\alpha \partial_t \Psi_\alpha +(1-\Delta)\Psi_\alpha + : \Psi_\alpha^{n+1}\overline{\Psi_\alpha}^n:\ = 2\sqrt{\re(\alpha)} \partial_tW, \\ 
&\Psi_\alpha |_{t=0}=\psi.
\end{aligned}
\right.
\end{equation}

\begin{thm}\label{mainthm:NRL} Let $\delta>0$. 
Consider the solutions $\{\Psi_{\varepsilon(j),\alpha}\}_{j\in\mathbb{N}}$ of \eqref{eq:dampedwave_Wick} 
according to the sequence $\varepsilon(j)=j^{-1}$. 
There exists a measurable set $\mathcal{A}\subset\mathcal{H}^{-\delta}$ such that $\mu(\mathcal{A})=1$ 
and for $(\psi,\phi)\in\mathcal{A}$, $\{\Psi_{\varepsilon(j),\alpha}\}_{j\in\mathbb{N}}$ converges to the solution of \eqref{eq:CGL} 
in $L^{\infty}_{\tau} H^{-\delta}(\mathbb{T}^2)$ for any $\tau<T_{\mathrm{hmax}}$, almost surely, 
where $T_{\mathrm{hmax}}$ is the maximal existence time of the solution of (\ref{eq:CGL}). Moreover, 
this convergence holds globally in time for any $(\psi,\phi)\in\mathcal{A}\cap\bigcap_{j\in\mathbb{N}}\mathcal{O}_{\varepsilon(j),\alpha}$. 
\end{thm}

The proof of Theorem \ref{mainthm:NRL} can be found in Section \ref{section:NRL}.

\begin{rem}
In \cite{Tr}, it was shown that $T_{\mathrm{hmax}}=\infty$ a.s. for any $\psi\in B_{\infty,\infty}^s$ with $s>-\frac2{2n+1}$.
The same result was shown in \cite{M} when $n=1$.
\end{rem}

\begin{rem}
More generally, as a sequence $\varepsilon(j)$ in Theorem \ref{mainthm:NRL}, we can take $\varepsilon(j)$ such that 
$\{\varepsilon(j)\}_j \in \ell^p$ for some large $p<\infty$ (see the proof of Proposition 8 where we use the Borel-Cantelli lemma.) The same remark is applied for $\alpha_2(j)$ of Corollary \ref{maincor:URL} below.
\end{rem}

\begin{rem}
This limit is justified by proving that
$U_{\varepsilon, \alpha}\to U_{0,\alpha}$ and
$Z_{\varepsilon,\alpha}\to Z_{0,\alpha}$ as $\varepsilon \to 0$.
The convergence of the former deterministic solutions can be obtained
once the $\varepsilon$-uniform estimates
is derived, using the fact that the eigenvalue of the damped wave
operator
$\lambda_{\varepsilon,\alpha}^+(\nabla)$ converges to the eigenvalue of the heat operator $-\frac{\langle \nabla\rangle^2}{2\alpha}$
as $\varepsilon \to 0$, while $\lambda_{\varepsilon,\alpha}^-(\nabla)$ diverges to $-\infty$ (see the proof of Theorem \ref{mainthm:simple} in Section \ref{section:NRL}).
The convergence of the latter stochastic solutions is shown, similarly
to the proof
that the finite dimensional Wick products converge, and here the
important point is
that the Gaussian measure $\mu$ is independent of $\varepsilon.$
The $\varepsilon$-independence of the Gibbs measure $\rho_{2n+2}$,
and the fact $\supp \mu = \supp \rho_{2n+2}$ (typically in the case $d\le 2$)
are quite important and they allow us to prove the global-in-time
convergence
for $\rho$-almost all initial data.
\end{rem}

In a similar way, we may prove the ultra relativistic limit. 
We fix $\varepsilon=1$ and $\re(\alpha)=\alpha_1>0$ and let $\im(\alpha)$ go to $0$.
For the sake of simplicity, we write $\Psi_{1,\alpha_1+\alpha_2i}=\Psi_{\alpha_2}$.
%For the sake of simplicity, we set $\varepsilon=1$. Our results are described as follows, 
%writing $\alpha_1=\re(\alpha), \alpha_2=\im(\alpha)$.    
\begin{cor}\label{maincor:URL} 
Consider the solutions $\{\Psi_{{\alpha_2}(j)}\}_{j\in\mathbb{N}}$ of \eqref{eq:dampedwave_Wick} according to the sequence ${\alpha_2}(j)=j^{-1}$.
There exists a measurable set $\mathcal{B}\subset\mathcal{H}^{-\delta}$ such that $\mu(\mathcal{B})=1$ 
and for $(\psi,\phi)\in\mathcal{B}$, $\{\Psi_{{\alpha_2}(j)}\}_{j\in\mathbb{N}}$ converges to the solution of \eqref{eq:dampedwave_W} with $\alpha$ replaced by $\alpha_1$
in $L^{\infty}_{\tau} H^{-\delta}(\mathbb{T}^2)$ for any $\tau< T_{\mathrm{wmax}} $, almost surely, 
where $T_{\mathrm{wmax}}$ is the maximal existence time 
of the solution of  (\ref{eq:dampedwave_W}) with $\alpha$ replaced by $\alpha_1$. Moreover,
this convergence holds globally in time 
for any $(\psi,\phi)\in\mathcal{B}\cap\bigcap_{j\in\mathbb{N}}\mathcal{O}_{{\alpha_2}(j)}$.
\end{cor}

The proof of Corollary \ref{maincor:URL} can be found in Section \ref{section:URL}.

\begin{rem} Recall that we have modified the equation (\ref{eq:phys}) by a replacement of $-\Delta$ to $-\Delta+1$.  
Local-in-time arguments can work for the regular initial data as in \cite{GKO}, thus the statements 
until Corollary  \ref{cor2 of mainthm} can be shown without such modification since we do not need to use 
the Gibbs invariant measure. By the same reason, we can prove the non relativistic and ultra relativistic convergence  
in (\ref{eq:phys}) if,  only local-in-time convergence and regular initial data are considered.
\end{rem}
\vspace{3mm}

We remark that in \cite {OOR} the authors consider the damped nonlinear wave equation  
with a regularized noise (without renormalization), and study possible limiting behavior of solutions as they remove the regularization. 
Such a triviality result is known for stochastic nonlinear heat and wave equations (see \cite{AHR,HRW,OOR}).
\vspace{3mm}

This paper is organized as follows. Section \ref{section3} is devoted to study the linear stochastic equation, including Proposition \ref{prop:base} 
which gives useful technical estimates for all over the paper.  
We prove in Section \ref{section:LWP} the local existence of solution, showing $\varepsilon$-uniform energy estimates.  
We globalize the local-in-time solution obtained in Section \ref{section:LWP} using the Gibbs measure in Section \ref{section:GE}. We give a proof of non relativistic limit in Section \ref{section:NRL}, 
and of ultra relativistic limit briefly in Section \ref{section:URL}. 
A uniform Strichartz estimate is shown in Section \ref{section7} . 
One can check the whole, direct proof of Strichartz estimates in case of torus in Appendices \ref{app:A} and \ref{app:W}, and a review on the complex Wick products in Appendix \ref{app:Wick}.  

%%%%%%%%%%%%%%%%%%%%%%%%%%%%%%%%%%%%%%%%%%%%%%%%%%%%%%
\section{Linear equation}\label{section3}
%%%%%%%%%%%%%%%%%%%%%%%%%%%%%%%%%%%%%%%%%%%%%%%%%%%%%%

The following technical estimates are used throughout the paper.

\begin{prop} \label{prop:base} Let $\alpha\in\mathbb{C}_+\setminus(0,\infty)$.%$\alpha \in \C$ such that $\re(\alpha)>0$ and $\im(\alpha) \ne 0$. 
\begin{enumerate}
\item\label{base:item1} $0 < \re\sqrt{\alpha^2-s}\le\re(\alpha)$ for any $s\ge0$, and the function $[0,\infty)\ni s\mapsto\re\sqrt{\alpha^2-s}$ is strictly decreasing.
\item\label{base:item2} $|\sqrt{\alpha^2-s}| \ge \sqrt{2 |\re(\alpha) \im(\alpha)|} $ for any $s\ge 0$.
%{\color{blue}
%\item\label{base:item2''} If $\re(\alpha)\le 1$ and $(\im(\alpha))^2<1$, then $|\sqrt{\alpha^2-s}| \ge 1$ for any $s\ge1$.
\item\label{base:item2'} 
%\begin{itemize}
%\item[(i)]
For any $s\ge0$,
$$
\left|\frac{\sqrt{s}}{\sqrt{\alpha^2-s}}\right|\le\sqrt{1+\frac{|\alpha|^2}{2|\re(\alpha)\im(\alpha)|}}.
$$
%Moreover, 
%\item[(ii)]
%If $\re(\alpha)\le 1$ and $(\im(\alpha))^2<1$, then for any $s\ge1$,
%$$
%\left|\frac{\sqrt{s}}{\sqrt{\alpha^2-s}}\right|\le\sqrt{1+|\alpha|^2}.
%$$
%\end{itemize}
%}
\item\label{base:item3} 
For any $s_0>0$, there exists a constant $C_{\alpha,s_0}>0$ such that 
$$
\re(-\alpha +\sqrt{\alpha^2-s}) \le -C_{\alpha,s_0} (s\wedge s_0)
$$
holds for any $s\ge0$. 
Moreover, $C_{\alpha,s_0}$ is locally bounded from above and below in the region $(\alpha,s_0)\in\mathbb{C}_+\times(0,\infty)$.
%$\mathbb{C}_+:=\{\alpha\in\mathbb{C}\, ;\, \re(\alpha)>0\}$.
\item\label{base:item4} For any $s\ge0$ we have
\begin{align*}
\left\{
\begin{aligned}
\sqrt{\alpha^2-s}= &\alpha-\frac{s}{2\alpha}+h(s,\alpha), \quad |h(s,\alpha)| \le 8 s^2/|\alpha|^3 
 &~  \mathrm{if}~ s < |\alpha|^2/2,\\
\sqrt{\alpha^2-s}= &i\sqrt{s}+ g(s, \alpha), \quad |g(s, \alpha)| \le 6 |\alpha|^2/\sqrt{s}, &~ \mathrm{if}~ s > 2|\alpha|^2.
\end{aligned}
\right.
\end{align*}
%with constants $C$ that do not depend on $\alpha$ {\color{blue}or?} $s$.
%
\end{enumerate}
\end{prop}

\begin{proof} The former statement of \eqref{base:item1} is trivial according to the definition of $\sqrt{\cdot}$ mentioned in the introduction. The latter 
may be seen from the computation 
$$ \frac{d}{ds} \left(\re{\sqrt{\alpha^2-s}} \right) =-\frac{1}{2} \frac{\re{\sqrt{\alpha^2-s}}}{|\sqrt{\alpha^2-s}|^2} <0$$ 
for any $s \ge 0$. The item \eqref{base:item2} follows from   
$$ |\sqrt{\alpha^2-s}|^4=\left\{ (\re \alpha)^2 -(\im \alpha)^2 -s \right\}^2 +4 (\re \alpha)^2 (\im \alpha)^2
\ge 4 (\re \alpha)^2 (\im \alpha)^2.$$
%If $\re(\alpha) \le 1$ and $(\im(\alpha))^2<1$, then we have
%\begin{align*}
%\inf_{s\ge1}|\sqrt{\alpha^2-s}|^4
%&\ge\inf_{s\ge1}\left\{ -(1-(\re \alpha)^2) -(\im \alpha)^2 - (s-1) \right\}^2\\
%& > \inf_{s\ge1}\left\{ -(1-(\re \alpha)^2) -1 - (s-1) \right\}^2 \\
%&=\inf_{s \ge 1} \left\{ (1-(\re \alpha)^2) + s)\right\}^2 
%\ge\inf_{s \ge 1} \left\{ s \right\}^2 \ge 1.
%\end{align*}
%Hence the item \eqref{base:item2''} holds. 
The item \eqref{base:item2'} follows from
\begin{align*}
\left|\frac{\sqrt{s}}{\sqrt{\alpha^2-s}}\right|^2
&=\frac{s}{|\alpha^2-s|}
\le1+\frac{|\alpha^2|}{|\alpha^2-s|}, 
\end{align*}
and \eqref{base:item2}. 
For \eqref{base:item3}, since $\re(-\alpha +\sqrt{\alpha^2-s})$ is decreasing with respect to $s$,
%$\re(-\alpha +\sqrt{\alpha^2-s}) \le\re(-\alpha+\sqrt{\alpha^2-s_0})$ for any $s\ge1$, 
it is sufficient to consider the case $s\le s_0$.
Writing 
 $\displaystyle{\re(-\alpha +\sqrt{\alpha^2-s}) =-\re \frac{s}{\alpha+\sqrt{\alpha^2-s}}}$,  we see the bound since by \eqref{base:item1} we have 
 \begin{align*}
 \re \left(\frac{1}{\alpha+\sqrt{\alpha^2-s}} \right) 
 &=\frac{\re(\alpha)+\re(\sqrt{\alpha^2-s})}{|\alpha+\sqrt{\alpha^2-s}|^2}\\
 &\ge\frac{\re(\alpha)}{|\alpha|^2+|\alpha^2-s|}
 \ge \frac{\re \alpha}{2|\alpha|^2+s_0^2}:=C_{\alpha,s_0}.
 \end{align*}
Finally we show \eqref{base:item4}. First we consider the complex function $f(z)=(1+z)^{1/2}=e^{\frac{1}{2} \mathrm{Log}_{\C}(1+z)}$ with $z \ne -1$.
%$z\in \C \setminus \{-1\}$.  
Fix any $N \in \N$. For $|z| < 1/2$, this function can be described by Cauchy's integral formula as 
\begin{eqnarray*}
f(z)
%=\frac{1}{2\pi i} \int_{|\xi|=1-\delta} \frac{f(\xi)}{\xi-z} d\xi 
%= \sum_{n=0}^{N-1} \left[ \frac{1}{2\pi i} \int_{|\xi|=1-\delta} \frac{f(\xi)}{\xi^{n+1}} d\xi\right] z^n +R_N(z)
= \sum_{n=0}^{N-1} \frac{f^{(n)}(0)}{n !} z^n + R_N(z),~\mbox{with}~ R_N(z)=\left[\frac{1}{2\pi i} \oint_{|\xi|=\frac{3}{4}} \frac{f(\xi)}{\xi^{N}(\xi-z)} d\xi \right] z^N,
\end{eqnarray*} 
where 
$|R_N(z)| \le 3\sqrt{2} (\frac{4}{3})^N |z|^N$. 
Now, we write $(\alpha^2-s)^{\frac12} = \alpha(1-\frac{s}{\alpha^2})^{\frac12}$, we apply the above formula to $z=-\frac{s}{\alpha^2}$ and we obtain 
$$ (\alpha^2-s)^{\frac12} =\alpha \left(1-\frac{1}{2}\frac{s}{\alpha^2}\right) +\alpha R_2(\alpha,s), \quad |R_2(\alpha,s)| \le 8 \left( \frac{s}{|\alpha|^2}\right)^2,$$
if $s<\frac{1}{2} |\alpha|^2$, which implies the first case. It then suffices to remark $(\alpha^2-s)^{\frac12} = i\sqrt{s}(1-\frac{\alpha^2}{s})^{\frac12}$ for the second case.   
\end{proof}

\subsection{Stationary solution}

We consider the linear equation \eqref{eq:Z}.
Since the results in this section are independent of $\varepsilon$ and $\alpha$, we write $Z=Z_{\varepsilon,\alpha}$ for simplicity.
Setting $Y=\varepsilon\partial_tZ$, we have the system
\begin{align}\label{eq:YZ}
\left\{
\begin{aligned}
dZ(t)&=\varepsilon^{-1}Y(t)dt,\\
dY(t)&=\varepsilon^{-1}\left\{-2\alpha \varepsilon^{-1}Y(t)-(1-\Delta)Z(t)\right\}dt 
+2\varepsilon^{-1}\sqrt{\re(\alpha)}dW(t).
\end{aligned}
\right.
\end{align}
We define the Gaussian measure on $\mathcal{S}'(\mathbb{T}^2)^2$ by
$$
\mu(dzdy)=\frac1{\Gamma_0}e^{-V_0(z,y)}dzdy,
$$
with $V_0(z,y)=\frac12\int_{\mathbb{T}^2}|z(x)|^2dx+\frac12\int_{\mathbb{T}^2}|\nabla z(x)|^2dx+\frac12\int_{\mathbb{T}^2}|y(x)|^2dx$
and with a normalizing constant $\Gamma_0$. Precisely, by identifying an element $\xi\in\mathcal{S}'(\mathbb{T}^2)$ with the sequence $(\hat{\xi}(k))_{k\in\mathbb{Z}^2}$, the measure $\mu$ is defined by the product
$$
\mu=\bigotimes_{k\in\mathbb{Z}^2}\mathcal{N}_c(0,2(1+|k|^2)^{-1})\otimes\bigotimes_{ \ell\in\mathbb{Z}^2 }\mathcal{N}_c(0,2),
$$
where $\mathcal{N}_c(0,r)$ denotes the complex normal distribution with mean zero and covariance $r>0$, see Appendix \ref{app:Wick}.
It is straightforward to see that $\mu$ is supported in $\mathcal{H}^{-\delta}$ for any $\delta>0$.

\begin{prop}\label{prop:invariance of YZ}
The Gaussian measure $\mu$ is invariant under the Markov process $(Z,Y)$.
\end{prop}

\begin{proof}
Denote by $\mathcal{D}$ the set of all functionals $\mathcal{H}^{-\delta}\to\mathbb{R}$ of the form
$$
F(z,y)=f\big((\hat{z}(k))_{|k|\le N},(\hat{y}(\ell))_{|\ell|\le N}\big)
$$
for some $N\in\mathbb{N}$ and some $C_b^2$ function $f$.
Moreover, for any semimartingale $X:[0,\infty)\to\mathcal{S}'(\mathbb{T}^2)$, we denote $\hat{X}(t;k):=\mathcal{F}(X(t))(k)$.
For any $F\in\mathcal{D}$, we have by the complex version of It\^o formula (see Proposition \ref{app:complexIto})
\begin{align*}
&dF(Z(t),Y(t))\\
&=\sum_{k}\partial_{\hat{z}(k)}F(Z(t),Y(t))d\hat{Z}(t;k)
+\sum_{k}\partial_{\overline{\hat{z}(k)}}F(Z(t),Y(t))d\overline{\hat{Z}(t;k)}\\
&\quad
+\sum_{k}\partial_{\hat{y}(k)}F(Z(t),Y(t))d\hat{Y}(t;k)
+\sum_{k}\partial_{\overline{\hat{y}(k)}}F(Z(t),Y(t))d\overline{\hat{Y}(t;k)}\\
&\quad
+4\varepsilon^{-2}\re(\alpha)\sum_{k}\partial_{\hat{y}(k)\overline{\hat{y}(k)}}F(Z(t),Y(t))d\hat{W}(t;k)d\overline{\hat{W}(t;k)}\\
&=:\varepsilon^{-1}\mathcal{L}^{(1)}F(Z(t),Y(t))dt
+\varepsilon^{-2}\mathcal{L}^{(2)}F(Z(t),Y(t))dt+(\text{martingale}),
\end{align*}
where
\begin{align*}
\mathcal{L}^{(1)}F(z,y)
&=2\re\left\{\sum_k\partial_{\hat{z}(k)}F(z,y)\hat{y}(k)
-\sum_k\partial_{\hat{y}(k)}F(z,y)(1+|k|^2)\hat{z}(k)\right\}
\end{align*}
and
\begin{align*}
\mathcal{L}^{(2)}F(z,y)
&=-4\re\left(\alpha\sum_k\partial_{\hat{y}(k)}F(z,y)\hat{y}(k)\right)
%-2\bar{\alpha}\sum_k\partial_{\overline{\hat{y}(k)}}F(z,y)\overline{\hat{y}(k)}\\
+8\re(\alpha)\sum_k\partial_{\hat{y}(k)\overline{\hat{y}(k)}}F(z,y).
\end{align*}
Then by using the elementary formula $\partial_{z}e^{-\frac12a|z|^2}=-\frac{a}2\bar{z} e^{-\frac12a|z|^2} $ ($a\in\mathbb{R}$),
it is straightforward to see that
\begin{align*}
\int_{\mathcal{H}^{-\delta}}\mathcal{L}^{(i)}F(z,y)\mu(dzdy)=0,\qquad i=1,2.
\end{align*}
This implies $\int P_tFd\mu=\int Fd\mu$ for any $t\ge0$ and $F\in\mathcal{D}$, where $P_t$ is the Markov semigroup associated with the process $(Z,Y)$.
By an approximation argument, the same equality holds for all $F\in C_b(\mathcal{H}^{-\delta})$. Hence $\mu$ is an invariant measure of the process $(Z,Y)$.
\end{proof}
\vspace{3mm}

\subsection{Wick polynomials of the stationary solution}

Let $(Z,Y)$ be the stationary solution of \eqref{eq:YZ} with the initial law $(Z,Y)|_{t=0}\sim\mu$.
We consider the Wick polynomials of $Z$.
See Appendix \ref{app:Wick} for the complex Wick polynomials. Note that
by the stationarity,
$$
C_N:=\mathbb{E}[|\Pi_NZ(t,x)|^2]
=\int |\Pi_Nz|^2\mu(dzdy)
=\sum_{|k|\le N}\frac2{1+|k|^2}.
$$

\begin{prop} \label{prop:wick_N}
Let $m,n\in\mathbb{N}$, $T>0$, $p\in[1,\infty)$, and $\delta>0$.
The sequence $\{H_{m,n}(\Pi_NZ;C_N)\}_{N\in\mathbb{N}}$ is Cauchy in $L^p(\Omega; C([0,T];W^{-\delta,\infty}))$
%C([0,T];W^{-\delta,\infty}))$,
and converges $\mathbb{P}_\mu$-almost surely. 
\end{prop}

\begin{proof}
Some modifications of \cite[Proposition 2.1]{GKO} implies the result, since
\begin{align*}
\mathbb{E}[\Pi_NZ(t,x)\overline{\Pi_NZ(t,y)}]
=\sum_{|k|\le N} e^{ik\cdot(x-y) } \frac2{1+|k|^2}.
\end{align*}
We will in fact present very similar computations below in Section \ref{subsection6.2}, we thus omit the proof.  
The latter part (almost sure convergence) is obtained by a similar way to \cite[Proposition 3.2]{OPT}, or the proof of Proposition \ref{prop:probability part of NRL}. 
\end{proof}

%%%%%%%%%%%%%%%%%%%%%%%%%%%%%%%%%%%%%%%%%%%%%%%%%%%%%%
\section{local existence}\label{section:LWP}
%%%%%%%%%%%%%%%%%%%%%%%%%%%%%%%%%%%%%%%%%%%%%%%%%%%%%%

We solve (\ref{eq:u})  
in the mild form, in the space 
$$ C([0,T], H^\sigma(\mathbb{T}^2)) \cap C^1([0,T],H^{\sigma-1}(\mathbb{T}^2)).$$
The proof of the local existence is quite similar to \cite{ORT}, but thanks to Theorem \ref{mainthm_energy} 
we have the local existence uniformly in 
$$
\varepsilon \in (0,1],\qquad \alpha\in K,
$$
for any fixed compact $K\subset\mathbb{C}_+\setminus(0,\infty)$.
First, we show Theorem \ref{mainthm_energy}. 

\begin{prop}\label{prop L2EST}
For any $\sigma\in\mathbb{R}$, we have
\begin{align}
\label{L2EST1}
&\|e^{t\lambda_{\varepsilon,\alpha}^\pm(\nabla)}u\|_{L_T^{\infty}H_x^\sigma}
\le \|u\|_{H_x^\sigma},\\
\label{L2EST2}
&\left\|\int_0^te^{(t-s)\lambda_{\varepsilon,\alpha}^-(\nabla)}f(s)ds\right\|_{L_T^{\infty}H_x^\sigma}
+\left\|\int_0^te^{(t-s)\lambda_{\varepsilon,\alpha}^+(\nabla)} {\bf 1}_{\{\varepsilon\langle\nabla\rangle>\frac{|\alpha|}{\sqrt{2}}\}}f(s)ds\right\|_{L_T^{\infty}H_x^\sigma}
\lesssim_\alpha\varepsilon\|f\|_{L_T^2H_x^\sigma},\\
\label{L2EST3}
&\left\|\int_0^te^{(t-s)\lambda_{\varepsilon,\alpha}^+(\nabla)} {\bf 1}_{\{\varepsilon\langle\nabla\rangle\le\frac{|\alpha|}{\sqrt{2}}\}}f(s)ds\right\|_{L_T^{\infty}H_x^\sigma}
\lesssim_\alpha\|f\|_{L_T^2H_x^{\sigma-1}}.
\end{align}
In the second and third inequalities, the implicit proportional constants are locally bounded functions of $\alpha\in\mathbb{C}_+$.
\end{prop}

\begin{proof}
\eqref{L2EST1} follows from $|e^{t\lambda_{\varepsilon,\alpha}^\pm(k)}|^2=e^{2t\re\lambda_{\varepsilon,\alpha}^\pm(k)}\le1$, since $\re\lambda_{\varepsilon,\alpha}^\pm(k)\le0$.
We show \eqref{L2EST2} for negative sign. Since $\re\lambda_{\varepsilon,\alpha}^-(k)<-\re(\alpha)/\varepsilon^2$,
\begin{align*}
\left\|\int_0^te^{(t-s)\lambda_{\varepsilon,\alpha}^-(\nabla)}f(s)ds\right\|_{H^\sigma}^2
&=\sum_{k\in\mathbb{Z}^d}\langle k\rangle^{2\sigma}\left|\int_0^te^{(t-s)\lambda_{\varepsilon,\alpha}^-(k)}\hat{f}(s;k)ds\right|^2\\
&\le\sum_{k\in\mathbb{Z}^d}\langle k\rangle^{2\sigma} \int_0^t|e^{(t-s)\lambda_{\varepsilon,\alpha}^-(k)}|^2ds \int_0^t|\hat{f}(s;k)|^2ds\\
&\le\sum_{k\in\mathbb{Z}^d}\langle k\rangle^{2\sigma} \frac1{2|\re\lambda_{\varepsilon,\alpha}^-(k)|} \int_0^t|\hat{f}(s;k)|^2ds\\
&\le\frac{\varepsilon^2}{2\re(\alpha)}\|f\|_{L_T^2H_x^\sigma}^2.
\end{align*}
We have \eqref{L2EST2} for positive sign by a similar argument, since $\re\lambda_{\varepsilon,\alpha}^+(k)<-C_\alpha/\varepsilon^2$ if $\varepsilon\langle k\rangle>|\alpha|/\sqrt2$ by Proposition \ref{prop:base}-\eqref{base:item3}.
%we decompose the integral into two parts
%$$
%{\bf1}_{\{\varepsilon\langle\nabla\rangle\le1\}}\int_0^te^{(t-s)\lambda_{\varepsilon,\alpha}^+(\nabla)}f(s)ds,\qquad
%{\bf1}_{\{\varepsilon\langle\nabla\rangle>1\}}\int_0^te^{(t-s)\lambda_{\varepsilon,\alpha}^+(\nabla)}f(s)ds.
%$$
%For the latter part, since $\re\lambda_{\varepsilon,\alpha}^+(k)<-C_\alpha/\varepsilon^2$ if $\varepsilon\langle k\rangle>1$ by Proposition \ref{prop:base}-\eqref{base:item3}, we have the required estimate similarly to the above argument for negative sign.
To show \eqref{L2EST3}, note from Proposition \ref{prop:base}-\eqref{base:item3} that if $\varepsilon \langle k\rangle\le|\alpha|/\sqrt2$ then
$$
\int_0^t|e^{(t-s)\lambda_{\varepsilon,\alpha}^+(k)}|^2ds\le \int_0^te^{-2C_\alpha(t-s)\langle k\rangle^2}ds\lesssim
\frac1{C_\alpha\langle k\rangle^2}.
$$
%for some $c>0$. 
Thus similarly to the above argument we have
\begin{align*}
\left\| \int_0^te^{(t-s)\lambda_{\varepsilon,\alpha}^+(\nabla)}{\bf 1}_{\{\varepsilon\langle\nabla\rangle\le\frac{|\alpha|}{\sqrt{2}}\}}f(s)ds\right\|_{H^\sigma}^2
\lesssim\frac1{C_\alpha}\|f\|_{L_T^2H_x^{\sigma-1}}^2.
\end{align*}
\end{proof}

Combining Proposition \ref{prop L2EST} with the following lemma, we can prove Theorem \ref{mainthm_energy}.

\begin{lem}\label{lem phipm and fpm}
For any $\sigma\in\mathbb{R}$,
\begin{align}
\label{eq1:lem phipm and fpm}\|\phi_{\varepsilon,\alpha}^\pm\|_{H^\sigma}&\lesssim_\alpha\|\phi_0\|_{H^\sigma}+\|\phi_1\|_{H^{\sigma-1}},\\
\label{eq2:lem phipm and fpm}\|\varepsilon\lambda_{\varepsilon,\alpha}^\pm(\nabla)\phi_{\varepsilon,\alpha}^\pm\|_{H^{\sigma-1}}&\lesssim_\alpha\|\phi_0\|_{H^\sigma}+\|\phi_1\|_{H^{\sigma-1}},\\
\label{eq3:lem phipm and fpm}\|f_{\varepsilon,\alpha}^\pm(t)\|_{H^{\sigma}}&\lesssim_\alpha\|f(t)\|_{H^{\sigma}}\wedge(\varepsilon^{-1}\|f(t)\|_{H^{\sigma-1}}),\\
\label{eq4:lem phipm and fpm}\|\varepsilon^2\lambda_{\varepsilon,\alpha}^\pm(\nabla) f_{\varepsilon,\alpha}^\pm(t)\|_{H^{\sigma}}&\lesssim_\alpha\|f(t)\|_{H^\sigma},\\
\label{eq5:lem phipm and fpm}\|\varepsilon\lambda_{\varepsilon,\alpha}^+(\nabla){\bf 1}_{\{\varepsilon\langle\nabla\rangle\le\frac{|\alpha|}{\sqrt2}\}}f_{\varepsilon,\alpha}^+(t)\|_{H^{\sigma}}&\lesssim_\alpha\|f(t)\|_{H^{\sigma+1}},
\end{align}
where the implicit proportional constants are locally bounded function of $\alpha\in\mathbb{C}_+\setminus(0,\infty)$.
\end{lem}

\begin{proof}
By Proposition \ref{prop:base}-\eqref{base:item2}, the operators $\frac1{\sqrt{\alpha^2-\varepsilon^2\langle\nabla\rangle^2}}$ and $\frac{\varepsilon^2\lambda^\pm_{\varepsilon,\alpha}(\nabla)}{\sqrt{\alpha^2-\varepsilon^2\langle\nabla\rangle^2}}=-\frac\alpha{\sqrt{\alpha^2-\varepsilon^2\langle\nabla\rangle^2}}\pm1$ are uniformly bounded in $H^\sigma$.
By Proposition \ref{prop:base}-\eqref{base:item2'}, $\frac1{\sqrt{\alpha^2-\varepsilon^2\langle\nabla\rangle^2}}$ is bounded by $\frac1{\varepsilon\langle\nabla\rangle}$.
They show \eqref{eq1:lem phipm and fpm}, \eqref{eq3:lem phipm and fpm}, and \eqref{eq4:lem phipm and fpm}.
Since
\begin{align*}
\varepsilon\lambda_{\varepsilon,\alpha}^\pm(\nabla)\phi_{\varepsilon,\alpha}^\pm
=\frac{\mp\varepsilon^3\lambda_{\varepsilon,\alpha}^\pm(\nabla)\lambda_{\varepsilon,\alpha}^\mp(\nabla)\phi_0\pm\varepsilon^2\lambda_{\varepsilon,\alpha}^\pm(\nabla)\phi_1}{2\sqrt{\alpha^2-\varepsilon^2\langle\nabla\rangle^2}}
=\frac{\mp\varepsilon\langle\nabla\rangle^2\phi_0\pm\varepsilon^2\lambda_{\varepsilon,\alpha}^\pm(\nabla)\phi_1}{2\sqrt{\alpha^2-\varepsilon^2\langle\nabla\rangle^2}},
\end{align*}
a similar argument to above shows \eqref{eq2:lem phipm and fpm}.
To show \eqref{eq5:lem phipm and fpm}, we have only to compute that $\lambda_{\varepsilon,\alpha}^+(k)=\varepsilon^{-2}O(\varepsilon^2\langle k\rangle^2)=\varepsilon^{-1}O(\langle k\rangle)$ in the region $\varepsilon\langle k\rangle\le\frac{|\alpha|}{\sqrt2}$, by using Proposition \ref{prop:base}-\eqref{base:item4}. 
%We can replace $\sqrt{\alpha^2-\varepsilon^2\langle k\rangle^2}$ by $1$ or $\varepsilon\langle k\rangle$, 
%from Proposition \ref{prop:base}-\eqref{base:item2} and \eqref{base:item2'} respectively, modulo proportional constants depending only on $\re(\alpha)$ and $|\im(\alpha)|$.
\end{proof}

\begin{proof}[{\bfseries Proof of Theorem \ref{mainthm_energy}}]
Combining Proposition \ref{prop L2EST} and Lemma \ref{lem phipm and fpm}-\eqref{eq1:lem phipm and fpm} and \eqref{eq3:lem phipm and fpm},
\begin{align*}
\|u_{\varepsilon,\alpha}\|_{L_T^\infty H_x^\sigma}
&\lesssim_\alpha \sum_{\circ=+,-}\|\phi_{\varepsilon,\alpha}^\circ\|_{H^\sigma}
+\varepsilon\sum_{\circ=+,-}\|f_{\varepsilon,\alpha}^\circ\|_{L_T^2H_x^{\sigma}}
+\|f_{\varepsilon,\alpha}^+\|_{L_T^2H_x^{\sigma-1}}\\
&\lesssim_\alpha \|\phi_0\|_{H^\sigma}+\|\phi_1\|_{H^{\sigma-1}}+\|f\|_{L_T^2H_x^{\sigma-1}}.
\end{align*}
For the bound on $\varepsilon \partial_t u_{\varepsilon, \alpha}$,  we see by (\ref{mildform}), 
\begin{eqnarray*} \nonumber  
\partial_t u_{\varepsilon, \alpha}(t)&=& \lambda_{\varepsilon,\alpha}^{+}(\nabla) 
e^{t\lambda_{\varepsilon, \alpha}^+(\nabla)}\phi_{\varepsilon, \alpha}^+
+\lambda_{\varepsilon,\alpha}^{-} (\nabla) e^{t\lambda_{\varepsilon, \alpha}^-(\nabla)}\phi_{\varepsilon, \alpha}^- \\ \nonumber 
&&+ \int_0^t\left(\lambda_{\varepsilon,\alpha}^{+}(\nabla) e^{(t-t')\lambda_{\varepsilon, \alpha}^+(\nabla)}f_{\varepsilon, \alpha}^+(t')
+\lambda_{\varepsilon,\alpha}^{-}(\nabla) e^{(t-t')\lambda_{\varepsilon, \alpha}^-(\nabla)}f_{\varepsilon, \alpha}^-(t')\right)dt'\\ %\label{deriv_mild}
&&+f^{+}_{\varepsilon, \alpha}(t) +f^{-}_{\varepsilon, \alpha}(t).
\end{eqnarray*}
Here $f_{\varepsilon,\alpha}^++f_{\varepsilon,\alpha}^-=0$ by definition.
Applying Proposition \ref{prop L2EST} to $u=\lambda_{\varepsilon,\alpha}^\pm(\nabla)\phi_{\varepsilon,\alpha}^\pm$ and $f=\lambda_{\varepsilon,\alpha}^\pm(\nabla)f_{\varepsilon,\alpha}^\pm$, we have
\begin{align*}
\|\partial_t u_{\varepsilon, \alpha}\|_{L_T^\infty H_x^{\sigma-1}}
&\lesssim_\alpha \sum_{\circ=+,-}\|\lambda_{\varepsilon,\alpha}^\circ(\nabla)\phi_{\varepsilon,\alpha}^\circ\|_{H^{\sigma-1}}
+\varepsilon\sum_{\circ=+,-}\|\lambda_{\varepsilon,\alpha}^\circ(\nabla)f_{\varepsilon,\alpha}^\circ\|_{L_T^2H_x^{\sigma-1}}\\
&\quad+\|{\bf 1}_{\{\varepsilon\langle\nabla\rangle\le\frac{|\alpha|}{\sqrt2}\}}\lambda_{\varepsilon,\alpha}^+(\nabla)f_{\varepsilon,\alpha}^+\|_{L_T^2H_x^{\sigma-2}}.
\end{align*}
Then by using Lemma \ref{lem phipm and fpm}-\eqref{eq2:lem phipm and fpm}, \eqref{eq3:lem phipm and fpm}, \eqref{eq4:lem phipm and fpm}, and \eqref{eq5:lem phipm and fpm}, we can conclude that
\begin{align*}
\|\partial_t u_{\varepsilon, \alpha}\|_{L_T^\infty H_x^{\sigma-1}}
&\lesssim_\alpha \varepsilon^{-1}(\|\phi_0\|_{H^\sigma}+\|\phi_1\|_{H^{\sigma-1}}+\|f\|_{L_T^2H_x^{\sigma-1}}).
\end{align*}
\end{proof}

%\begin{lem}\label{cor1 of mainthm}
%For any $\beta\in(0,1)$ and $\sigma\in\mathbb{R}$, we have
%\begin{align}\label{thm:u<+u>}
%\|(u_\varepsilon, \partial_t u_{\varepsilon})\|_{L^{\infty}_T(H^{\sigma} \times  H^{\sigma-1})}
%\lesssim\|\phi_0\|_{H^\sigma}+\|\phi_1\|_{H^{\sigma-1}}
%+\|f\|_{L_T^2H^{\sigma-\beta}}.
%\end{align}
%\end{lem}

We can get Corollary \ref{cor2 of mainthm} by the standard PDE argument, similarly to \cite[Proposition 3.5]{GKO} and \cite[Proposition 4.1]{ORT}.

\begin{proof}[{\bfseries Proof of Corollary \ref{cor2 of mainthm}}]
By Proposition \ref{app eq:Hermite} in Appendix \ref{app:Wick}, we can expand
\begin{align*}
:(U+Z)^{n+1}(\overline{U}+\overline{Z})^n:\
=\sum_{k=0}^{n+1}\sum_{\ell=0}^n\binom{n+1}{k}\binom{n}{\ell}
U^{n+1-k}\overline{U}^{n-\ell}:Z^k\overline{Z}^\ell:.
\end{align*}
Hence we can write the right hand side by the form
$$
F_n(U)=\sum_{k\le 2n+1} P_k(U)\Xi_k,
$$
where $P_k(U)$ is a $k$-th homogeneous polynomial of $U$ and $\overline{U}$, and $\Xi_k$ is an element of $L_T^\infty W_{\color{blue}x}^{-\delta,\infty}$,
depends on the degree $k$.

Fix any $T>0$. 
Denote the solution map for \eqref{eq:dampedwave_f} by $S_{\varepsilon,\alpha}(\phi_0,\phi_1;f)$.
For (\ref{eq:u}) it is sufficient to show that the nonlinear operator
$$
\Upsilon_{\varepsilon,\alpha}(U):=S_{\varepsilon,\alpha}(u,v;-F_n(U)), 
$$
%with $F_k(u,Z) = u^k+\sum_{\ell=0}^{k-1}:Z^{k-l} : u^\ell$
is a contraction map on the space $C_{T_0} H_x^\sigma$ for small $0<T_0 \le T \wedge 1$.

Fix $\sigma<1$ and set $\delta:=1-\sigma$.
%Fix $\beta\in(0,1)$ and $\sigma<\beta$ such that $\varepsilon_0:=\beta-\sigma$ is sufficiently small.
By Theorem \ref{mainthm_energy}, 
\begin{align*}
\|\Upsilon_{\varepsilon,\alpha}(U)\|_{L^{\infty}_{T_0} H_x^{\sigma}}
\lesssim\|u\|_{H^\sigma}+\|v\|_{H^{\sigma-1}}+\|F_n(U)\|_{L_{T_0}^2H_x^{-\delta}}.
\end{align*}
Consider the each term of $F_n(U)$, for any $k \le 2n+1$,
%For $u^k$, by the Sobolev embedding,
%\begin{align*}
%\|u^k\|_{L_{T_0}^2H^{-\varepsilon_0}}
%\lesssim {T_0}^{\frac12}\|u^k\|_{L_{T_0}^\infty L^2}
%\lesssim {T_0}^{\frac12}\|u\|_{L_{T_0}^\infty L^{2k}}^k
%\lesssim {T_0}^{\frac12}\|u\|_{L_{T_0}^\infty H^{(1-1/k)+}}^k.
%\end{align*}
%Hence
%\begin{align*}
%\|u^k\|_{L_{T_0}^2H^{-\varepsilon_0}}\lesssim
%{T_0}^{\frac12}\|u\|_{L_{T_0}^\infty H^\sigma}^k
%\end{align*}
%if $\beta$ is sufficiently closed to $1$.
%For the other terms $l=0,...,k-1$,
\begin{align*}
\|P_k(U)\Xi_k\|_{L_{T_0}^2 H_x^{-\delta}}
\lesssim \|\Xi_k\|_{L_{T_0}^2W_x^{-\delta,\infty}}\|P_k(U)\|_{L_{T_0}^\infty H_x^{2\delta}}
\lesssim {T_0}^{\frac12}\|\Xi_k\|_{L_{T_0}^\infty W_x^{-\delta,\infty}}
\|u\|_{L_{T_0}^\infty B_{2k,2k}^{2\delta}}^k.
\end{align*}
Hence by the Besov embedding $H^{2\delta+(1-1/k)}\subset B_{2k,2k}^{2\delta}$,
\begin{align*}
\|P_k(U)\Xi_k\|_{L_{T_0}^2H_x^{-\delta}}
\lesssim {T_0}^{\frac12}\|\Xi_k\|_{L_T^\infty W_x^{-\delta,\infty}}
\|U\|_{L_{T_0}^\infty H_x^\sigma}^k
\end{align*}
if $\sigma$ is sufficiently close to $1$.
Therefore we see that 
$\Upsilon_{\varepsilon,\alpha}(U)$ is a contraction in the closed ball 
$\{ \|U\|_{L_{T_0}^\infty H_x^\sigma} \le R \}$ with 
$$ \|(u,v)\|_{\mathcal{H}^{\sigma}} + \sup_{k \le 2n+1}\|\Xi_k\|_{L^\infty_T W_x^{-\delta,\infty}} =R/2, $$
if  $T_0^{\frac12} R^{2n+1} \le \frac12.$  
Note that $R$ is random, but finite a.s.  
The difference $\|\Upsilon_{\varepsilon,\alpha}(U_1)-\Upsilon_{\varepsilon,\alpha}(U_2)\|_{L^{\infty}_{T_0} H_x^{\sigma}}$ 
may be estimated similarly. 
The continuity in time follows from the continuity of $e^{t\lambda_{\varepsilon,\alpha}^\pm(\nabla)}$. 
The bound on $\|\varepsilon \partial_t U\|_{L^{\infty}_{T_0} H_x^{\sigma-1}}$ follows from Theorem \ref{mainthm_energy}. 
The uniformity of the maximal existence time mentioned in Remark \ref{rem:cor21} follows from the similar arguments 
to \cite[Theorem 1.1]{IM}
%Theorem 1.1 of \cite{IM} 
using the uniform bound of Theorem \ref{mainthm_energy}.
\end{proof}

%%%%%%%%%%%%%%%%%%%%%%%%%%%%%%%%%%%%%%%%%%%%%%%%%%%%%%
\section{global existence a.e. $\rho_{2n+2}$}\label{section:GE}
%%%%%%%%%%%%%%%%%%%%%%%%%%%%%%%%%%%%%%%%%%%%%%%%%%%%%%

We show Proposition \ref{mainprop:GWP} in this section.
For simplicity, we write $\Psi=\Psi_{\varepsilon,\alpha}$, $\rho=\rho_{2n+2}$ and $T^{\ast}=T^{\ast}_{\varepsilon, \alpha}$ which is given by Remark \ref{rem:cor21}.
We consider a finite dimensional approximation
\begin{align}\label{eq:dampedwave_approx}
\left\{
\begin{aligned}
&\varepsilon^2\partial_t^2 \Psi^N+2\alpha \partial_t \Psi^N+(1-\Delta) \Psi^N
+\Pi_NH_{n+1,n}(\Pi_N\Psi^N;C_N)\\
& =2\sqrt{\re(\alpha)} \partial_t W, &t>0,~ x\in \mathbb{T}^2,\\
&(\Psi^N,\varepsilon\partial_t \Psi^N)|_{t=0}=(\psi,\phi), &x\in\mathbb{T}^2.
\end{aligned}
\right.
\end{align}
Setting $\Phi^N=\varepsilon\partial_t\Psi^N$, we have the system
\begin{align*}
\left\{
\begin{aligned}
d\Psi^N(t)&=\varepsilon^{-1} \Phi^N(t)dt,\\
d\Phi^N(t)&=\varepsilon^{-1}\left\{-2\alpha \varepsilon^{-1}\Phi^N(t)-(1-\Delta)\Psi^N(t)-\Pi_NH_{n+1,n}(\Pi_N\Psi^N;C_N)\right\}dt\\
&\quad+2\varepsilon^{-1}\sqrt{\re(\alpha)}dW(t).
\end{aligned}
\right.
\end{align*}
For any $N\in\mathbb{N}$, we define the truncated measure
$$ \rho_N(d\psi d\phi) = \Gamma_N^{-1} e^{-H_N(\psi,\phi)} d\psi d\phi= \Gamma_N^{-1} e^{-\frac{1}{2n+2} \int_{\mathbb{T}^2} H_{n+1,n+1}(\Pi_N\psi;C_N) dx }  
\mu(d\psi d\phi).
%d\mu_0  \otimes \mu_1.
$$  
where,  
$$H_N(\psi, \phi)=\frac{1}{2} \int_{\mathbb{T}^2} |\phi|^2 dx + V_N(\psi),$$
and
$$V_N(\psi) =\frac{1}{2} \int_{\mathbb{T}^2} |\nabla \psi|^2 dx + \frac12 \int_{\mathbb{T}^2} |\psi|^2 dx + \frac{1}{2n+2} \int_{\mathbb{T}^2} H_{n+1,n+1}(\Pi_N\psi;C_N) dx.$$

We can apply Corollary \ref{cor2 of mainthm} to Eq. (\ref{eq:dampedwave_approx}) and prove the following result.
The global well-posedness of $(\Psi^N,\Phi^N)$ follows from the energy estimate as in \cite[Proposition 5.1]{ORT}.

\begin{prop} \label{prop:approx}
Let $\delta>0$ and $T>0$.  
Let $(\psi,\phi) \in \mathcal{H}^{-\delta}$. Then, there exists a unique global 
solution in $C([0,T];\mathcal{H}^{-\delta})$ denoted  by $(\Psi^N,\Phi^N)$ of (\ref{eq:dampedwave_approx}).  Moreover,  
there exists a measurable set $\mathcal{O}\subset\mathcal{H}^{-\delta}$ such that $\mu(\mathcal{O})=1$ 
with the following properties: for any $(\psi,\phi)\in\mathcal{O}$, $(\Psi^N,\Phi^N) \to (\Psi,\Phi)$ in $C([0,T]; \mathcal{H}^{-\delta})$ in probability
as $N \to \infty$ for any $T<T^{\ast}$, where $(\Psi,\Phi) = (U, \varepsilon\partial_t U) +(Z, \varepsilon \partial_t Z)$,  
$U$ being the solution of  
(\ref{eq:u}) with zero initial values, given by Corollary \ref{cor2 of mainthm} and $T^{\ast}$ being its maximal existence time 
given by Corollary \ref{cor2 of mainthm} and Remark \ref{rem:cor21}.  
\end{prop}

\begin{proof}
Proposition \ref{prop:invariance of YZ} and Proposition \ref{prop:wick_N} imply the existence of 
the measurable set $\mathcal{O}\subset\mathcal{H}^{-\delta}$ such that, $\mu(\mathcal{O})=1$ and, for any $(\psi,\phi)\in\mathcal{O}$ 
the equation \eqref{eq:Z} has a unique solution $Z$ with $(Z,\varepsilon\partial_tZ)|_{t=0}=(\psi,\phi)$, 
and the sequence $\{H_{m,n}(\Pi_NZ;C_N)\}_{N\in\mathbb{N}}$ converges in probability in $C([0,T];W^{-\delta,\infty})$.  
Then we can decompose $(\Psi^N,\Phi^N)=(U^N,\varepsilon\partial_tU^N)+(Z,\varepsilon\partial_tZ)$, where
$$
\left\{
\begin{aligned}
&\varepsilon^2\partial_t^2U^N+2\alpha\partial_tU^N+(1-\Delta)U^N+\Pi_NH_{n+1,n}(\Pi_NU^N+\Pi_NZ;C_N)=0,\\
&(U^N,\varepsilon\partial_tU^N)|_{t=0}=(0,0).
\end{aligned}
\right.
$$
Combining Proposition \ref{prop:wick_N} and Corollary \ref{cor2 of mainthm}, we have the convergence of $(U^N,\varepsilon\partial_tU^N)$ up to the maximal existence time.
\end{proof}

$(\Psi^N,\Phi^N)$ is a Markov process for each fixed $N\in\mathbb{N}$.
We define the Feller transition semigroup $P_t^N f(\psi, \phi)=\E (f(\Psi^N, \Phi^N)(t, (\psi, \phi)))$ for $(\psi, \phi) \in \mathcal{H}^{-\delta}$.

\begin{prop}
The measure $\rho_N$ is invariant for $(P_t^N)_{t\ge 0}$. 
\end{prop}

\begin{proof}
First we use \eqref{eq:Hermite2} to have
\begin{align*}
\partial_{\overline{\hat\psi(k)}}V_N(\psi)
&=\partial_{\overline{\hat\psi(k)}}
\Big\{\frac12\sum_k(1+|k|^2)|\hat\psi(k)|^2
+\frac1{2n+2}\int_{\mathbb{T}^2}H_{n+1,n+1}(\Pi_N\psi;C_N)dx
\Big\}\\
&=\frac12\Big\{
\sum_k(1+|k|^2)\hat\psi(k)
+\sum_{|k|\le N}\int_{\mathbb{T}^2}H_{n+1,n}(\Pi_N\psi;C_N)e_{-k}(x)dx
\Big\}\\
&=\frac12\Big\{
\sum_k(1+|k|^2)\hat\psi(k)
+\langle \Pi_NH_{n+1,n}(\Pi_N\psi;C_N), e_k\rangle
\Big\},
\end{align*}
which is compared with the drift term of \eqref{eq:dampedwave_approx}.

For any $F\in\mathcal{D}$ (see the proof of Proposition \ref{prop:invariance of YZ}), we have by the complex version of It\^o formula,
\begin{align*}
&dF(\Psi^N(t),\Phi^N(t))\\
&=\sum_{k}\partial_{\hat\psi(k)}F(\Psi^N(t),\Phi^N(t))d\hat\Psi^N(t;k)
+\sum_{k}\partial_{\overline{\hat\psi(k)}}F(\Psi^N(t),\Phi^N(t))d\overline{\hat\Psi^N(t;k)}\\
&\quad
+\sum_{k}\partial_{\hat\phi(k)}F(\Psi^N(t),\Phi^N(t))d\hat\Phi^N(t;k)
+\sum_{k}\partial_{\overline{\hat\phi(k)}}F(\Psi^N(t),\Phi^N(t))d\overline{\hat\Phi^N(t;k)}\\
&\quad
+4\varepsilon^{-2}\re(\alpha)\sum_{k}\partial_{\hat{\phi}(k)\overline{\hat\phi(k)}}F(\Psi^N(t),\Phi^N(t))d\hat{W}(t;k)d\overline{\hat{W}(t;k)}\\
&=:\varepsilon^{-1}\mathcal{L}_V^{(1)}F(\Phi^N(t),\Psi^N(t))dt
+\varepsilon^{-2}\mathcal{L}_V^{(2)}F(\Phi^N(t),\Psi^N(t))dt+(\text{martingale}),
\end{align*}
where
\begin{align*}
\mathcal{L}_V^{(1)}F(\psi,\phi)
&=2\re\Big\{\sum_k\partial_{\hat\psi(k)}F(\psi,\phi)\hat\phi(k)
-2\sum_k\partial_{\hat\phi(k)}F(\psi,\phi)\partial_{\overline{\hat\psi(k)}}V_N(\psi)\Big\}
\end{align*}
and $\mathcal{L}_V^{(2)}$ is the same as $\mathcal{L}^{(2)}$ which appears in the generator of \eqref{eq:YZ}.
For $\mathcal{L}_V^{(1)}$, since
\begin{align*}
&\int_{\mathcal{H}^{-\delta}}
\Big\{\partial_{\hat\psi(k)}F(\psi,\phi)\hat\phi(k)-2\partial_{\hat\phi(k)}F(\psi,\phi)\partial_{\overline{\hat\psi(k)}}V_N(\psi)\hat\psi(k)\Big\}
\rho_N(d\psi d\phi)\\
&=\int_{\mathcal{H}^{-\delta}}
F(\psi,\phi)
\Big\{\hat\phi(k)\partial_{\hat\psi(k)}V_N(\psi)-\overline{\hat\phi(k)}\partial_{\overline{\hat\psi(k)}}V_N(\psi)\Big\}
\rho_N(d\psi d\phi)\\
&=\int_{\mathcal{H}^{-\delta}}
F(\psi,\phi)
\Big\{\hat\phi(k)\partial_{\hat\psi(k)}V_N(\psi)-\overline{\hat\phi(k)\partial_{\hat\psi(k)}V_N(\psi)}\Big\}
\rho_N(d\psi d\phi)
\end{align*}
is a pure imaginary number, we have
$$
\int_{\mathcal{H}^{-\delta}}\mathcal{L}_V^{(1)}F(\psi,\phi)\rho_N(d\psi d\phi)=0.
$$
Since $\mathcal{L}^{(2)}$ does not contain $\partial_{\hat\psi(k)}$ and $\partial_{\overline{\hat\psi(k)}}$, similarly to Proposition \ref{prop:invariance of YZ} we have
$$
\int_{\mathcal{H}^{-\delta}}\mathcal{L}_V^{(2)}F(\psi,\phi)\rho_N(d\psi d\phi)=0.
$$
\end{proof}

The following lemma reflects some important properties of $\rho_N$, which are remarked in \cite[Lemma 3.2]{ORT}.

\begin{lem} \label{lem:G_N}
\begin{itemize}
\item[(1)] The normalizing constant $\Gamma_N$ is bounded below independent of $N$.
%The sequence of measures $\{\rho^N\}_{N\in\mathbb{N}}$ weakly converges.
\item[(2)] 
The sequence 
$$\left\{e^{-\frac{1}{2n+2}\int_{\mathbb{T}^2} H_{n+1,n+1}(\Pi_N\psi;C_N) dx} \right\}_{N\in \N}$$ 
is  Cauchy in $L^p(d\mu_0)$ for $p \ge 1$, 
and we denote the limit by 
$$ e^{-\frac{1}{2n+2} \int_{\mathbb{T}^2} : |\psi|^{2n+2}:(x) dx}.$$
\end{itemize}
\end{lem}

The invariance of $\rho_N$ yields the following lemma.

\begin{lem}
Let any $T>0$ be fixed. There exists a constant $C_T$ independent of $N$ such that 
\begin{align*} 
\sup_{N\in\mathbb{N}}\int_{\mathcal{H}^{-\delta}}
\mathbb{E}\left[\sup_{0\le t\le T}\|(\Psi^N(t),\Phi^N(t))\|_{\mathcal{H}^{-\delta}}\right]\rho_N(d\psi d\phi) \le C_T.
\end{align*}
\end{lem}

\begin{proof}
Let $R>0$. According to the proof of Corollary \ref{cor2 of mainthm}, 
there are constants $C>0$ and $\eta>0$ independent of $N$ and $R$, such that 
if $T_0=CR^{-\eta}$, then 
$(\Psi^N,\Phi^N) = (U^N, \varepsilon\partial_t U^N) +(Z, \varepsilon \partial_tZ)$ satisfies 
$$ \|(\Psi^N, \Phi^N)\|_{L^{\infty}_{T_0} \mathcal{H}_x^{-\delta}} \le R,$$
provided 
$\|(\psi,\phi)\|_{\mathcal{H}^{-\delta}}\le R/2$ and
\begin{equation} \label{eq:1}
\sum_{k=1}^{2n+1} \|\Xi_k\|_{L^\infty_T W_x^{-\delta,\infty}} \le R/2.
\end{equation}
($\{\Xi_k\}$ contains $Z$ as an element.)
%and  
%\begin{equation} \label{eq:2}
%\sup_{t\in [0,T]}  \|Z\|_{W^{-\varepsilon_0,\infty}} \le R/2.
%\end{equation}
Note that we may solve equation on any time interval of length $T_0$, so that for $\ell=0,1,2...[\frac{T}{T_0}]$, 
$$ \|(\Psi^N, \Phi^N)\|_{C([\ell T_0, (\ell+1)T_0]; \mathcal{H}^{-\delta})} \le R$$
provided (\ref{eq:1}) holds true and 
$ \|(\Psi^N, \Phi^N)(\ell T_0)\|_{\mathcal{H}^{-\delta}} \le R/2.$
We infer that for any $R>0$ 
\begin{eqnarray*}
&& \prob_{\rho_N} (\sup_{t\in [0,T]} \|(\Psi^N,\Phi^N)(t)\|_{\mathcal{H}^{-\delta}} >R) \\
&& \le 
\sum_{\ell=0}^{[\frac{T}{T_0}]} \prob_{\rho_N} (\sup_{[\ell T_0, (\ell+1)T_0]}\|(\Psi^N,\Phi^N)(t)\|_{\mathcal{H}^{-\delta}} >R ) \\
%\\&& 
&& \le \sum_{\ell=0}^{[\frac{T}{T_0}]} \prob_{\rho_N} (\|(\Psi^N,\Phi^N)(\ell T_0)\|_{\mathcal{H}^{-\delta}} >R/2 ) +
([T/T_0]+1)\prob(\sum_{k=1}^{2n+1} \|\Xi_k\|_{L^\infty_T W_x^{-\delta,\infty}} > R/2).
 %+\prob(\sup_{t\in [0,T]}  \|Z\|_{W^{-\varepsilon_0,\infty}} > R/2)
\end{eqnarray*}
By Markov inequality and Proposition \ref{prop:wick_N},
$$
\prob(\sum_{k=1}^{2n+1} \|\Xi_k\|_{L^\infty_T W^{-\delta,\infty}} > R/2) \lesssim R^{-m}
%\quad \prob(\sup_{t\in [0,T]}  \|Z\|_{W^{-\varepsilon_0,\infty}} > R/2) \lesssim R^{-m}. 
$$
holds for any $m\in\mathbb{N}$.
By the invariance of $\rho_N,$ for any $\ell=0,1,2,\cdots,\left[\frac{T}{T_0}\right],$
\begin{eqnarray*}
\prob_{\rho_N} (\|(\Psi^N,\Phi^N)(\ell T_0)\|_{\mathcal{H}^{-\delta}} >R/2 )
&=&\rho_N (\|(\psi,\phi)\|_{\mathcal{H}^{-\delta}} >R/2) \\
& \lesssim&  R^{-m} \Gamma_N^{-1} \E_\mu(\|(\psi,\phi)\|^m_{\mathcal{H}^{-\delta}} ) \lesssim R^{-m},
\end{eqnarray*}
where the latter bound comes from Lemma \ref{lem:G_N}. 
Finally 
$$
\prob_{\rho_N} (\sup_{t\in [0,T]} \|(\Psi^N(t),\Phi^N(t))\|_{\mathcal{H}^{-\delta}} >R) 
\lesssim R^{{\eta}-m}
$$
thus, choosing $m$ large enough the statement follows.  
\end{proof}

Proposition \ref{mainprop:GWP} follows from the next statement.

\begin{thm} Fix any $T>0$. Then there exists a constant $C_T>0$ such that  
\begin{align*}
\int_{\mathcal{H}^{-\delta}}
\mathbb{E}\left[\sup_{0\le t< T\wedge T^*}\|(\Psi(t),\Phi(t))\|_{\mathcal{H}^{-\delta}}\right]\rho(d\psi d\phi) \le C_T.
\end{align*}
Hence $T^*=\infty$ a.s. for $\rho$-a.e. $(\psi,\phi)\in\mathcal{H}^{-\delta}$. 
\end{thm}

\begin{proof}
It follows from Lemma \ref{lem:G_N} that taking a subsequence,
$e^{-\frac{1}{2n+2}\int_{\mathbb{T}^2} H_{n+1,n+1}(\Pi_N\psi;C_N) dx}$ converges 
to $e^{-\frac{1}{2n+2} \int_{\mathbb{T}^2} : |\psi|^{2n+2}:(x) dx},$ 
as $N\to \infty$, $\mu$-a.e. Hence, recalling Proposition \ref{prop:approx},  Fatou's lemma implies 
\begin{eqnarray*}
&&\int_{\mathcal{H}^{-\delta}}
\mathbb{E}\left[\sup_{0\le t< T\wedge T^*}\|(\Psi(t),\Phi(t))\|_{\mathcal{H}^{-\delta}}\right]\rho(d\psi d\phi)\\
&=& 
\Gamma^{-1}\int_{\mathcal{H}^{-\delta}}
\mathbb{E}\left[\sup_{0\le t< T\wedge T^*}\|(\Psi(t),\Phi(t))\|_{\mathcal{H}^{-\delta}}\right]  
e^{-\frac{1}{2n+2} \int_{\mathbb{T}^2} : |\psi|^{2n+2}:(x) dx}\mu(d\psi d\phi) \\
& \le & \liminf_{N \to \infty} \Gamma_N^{-1}\int_{\mathcal{H}^{-\delta}}
\mathbb{E}\left[\sup_{0\le t< T\wedge T^*}\|(\Psi^N(t),\Phi^N(t))\|_{\mathcal{H}^{-\delta}}\right]  
e^{-\frac{1}{2n+2}\int_{\mathbb{T}^2} H_{n+1,n+1}(\Pi_N\psi;C_N) dx} \mu(d\psi d\phi)  \\
&\le& C_T. 
\end{eqnarray*}
This estimate implies that there exists a $\rho$-measurable set $\mathcal{M}_T \subset \mathcal{O}$ (of Proposition \ref{prop:approx}) 
such that $\rho(\mathcal{M}_T)=1$ and that for $(\psi, \phi) \in \mathcal{M}_T$, 
$\sup_{0\le t< T\wedge T^*}\|(\Psi(t),\Phi(t))\|_{\mathcal{H}^{-\delta}}$ is finite a.s.. Since this value should blow up if $T^*<T$, it implies that
the solution $\|(\Psi(t),\Phi(t))\|_{\mathcal{H}^{-\delta}}$ exists up to time 
$T$ a.s.. For each $T_n$ ($n\in \N$), such that $T_n \to \infty$ as $n \to \infty$, consider $\mathcal{M}_{T_n}$ and set $\overline{\mathcal{M}} := \cap_n \mathcal{M}_{T_n}.$ 
Then $\rho(\overline{\mathcal{M}})=1$. 
Namely $\|(\Psi(t),\Phi(t))\|_{\mathcal{H}^{-\delta}}$ exists globally, i.e. $T^*=\infty$ a.s. for any  $(\psi, \phi) \in \overline{\mathcal{M}}$.
\end{proof}
\vspace{3mm}

Thus, we can define the transition semigroup $P_t f (\psi,\phi)=\E (f(\Psi, \Phi)(t, (\psi, \phi)))
$ for  $(\psi,\phi) \in \overline{\mathcal{M}}$ and $t \ge 0$, and the invariance of measure $\rho$ follows straightforward: 

\begin{cor}
The measure $\rho$ is invariant for $(P_t)_{t\ge 0}$. 
\end{cor}

%%%%%%%%%%%%%%%%%%%%%%%%%%%%%%%%%%%%%%%%%%%%%%%%%%%%%%
\section{non relativistic limit}\label{section:NRL}
%%%%%%%%%%%%%%%%%%%%%%%%%%%%%%%%%%%%%%%%%%%%%%%%%%%%%%

In this section, we fix $\alpha$ and let $\varepsilon$ go to $0$,
so we omit the label $\alpha$ from underlined objects, e.g. 
we write $\Psi_\varepsilon=\Psi_{\varepsilon,\alpha}$.

\subsection{Deterministic result}

Let $d\ge1$ be arbitrary for the moment.
 We can expect that, the limit of the solution of \eqref{eq:dampedwave_f} as $\varepsilon\to0$ solves the linear heat equation
\begin{align*}%\label{eq:limit heat}
\left\{
\begin{aligned}
&2\alpha\partial_tv+(1-\Delta) v=f,&t>0,~ x\in \mathbb{T}^d,\\
&v|_{t=0}=\phi_0, & x\in\mathbb{T}^d,
\end{aligned}
\right.
\end{align*}
or equivalently,
$$
v(t)=e^{\frac{t}{2\alpha}(\Delta-1)}\phi_0+\frac1{2\alpha}\int_0^te^{\frac{t-t'}{2\alpha}(\Delta-1)}f(t')dt'.
$$

\begin{thm}\label{mainthm:simple}
For any $\sigma\in\mathbb{R}$ and $\theta\in[0,1]$,
\begin{align*}
\|u_\varepsilon-v\|_{L_T^\infty H_x^{\sigma}}
\lesssim_{\alpha} \varepsilon^\theta\big(\|\phi_0\|_{H^{\sigma+\theta}}+\|\phi_1\|_{H^{\sigma-1+\theta}}
+\|f\|_{L_T^2H_x^{\sigma-1+\theta}}\big),
\end{align*}
where $u_\varepsilon=u_{\varepsilon,\alpha}$ is the solution of \eqref{eq:dampedwave_f}, and 
the proportional constants are locally bounded functions of $\alpha \in \mathbb{C}_{+} \setminus (0,\infty).$  
\end{thm}
\vspace{3mm}

\begin{proof}  
%[{\bfseries Proof of Theorem \ref{mainthm:simple}]
We decompose
$$
u_\varepsilon-v
={\bf1}_{\{\varepsilon\langle\nabla\rangle\le|\alpha|/\sqrt{2}\}}(u_\varepsilon-v)
+{\bf1}_{\{\varepsilon\langle\nabla\rangle>|\alpha|/\sqrt{2}\}}(u_\varepsilon-v).
$$
For the latter part, we have the required estimate as a consequence of the $\varepsilon$-uniform estimates (Theorem \ref{mainthm_energy}), since
$$
\|{\bf1}_{\{\varepsilon\langle\nabla\rangle>|\alpha|/\sqrt{2}\}}w\|_{H^\sigma}\lesssim_\alpha\varepsilon^\theta\|w\|_{H^{\sigma+\theta}}
$$
for any $w\in H^{\sigma+\theta}$.

We consider the former part. First let $f=0$ and reorganize the term concerning initial values as follows.
\begin{align*}
u_\varepsilon-v
&=e^{t\lambda_\varepsilon^+(\nabla)}(\phi_\varepsilon^+-\phi_0)
+e^{t\lambda_\varepsilon^-(\nabla)}\phi_\varepsilon^-
+(e^{t\lambda_\varepsilon^+(\nabla)}-e^{\frac{t}{2\alpha}\Delta})\phi_0\\
&=:A_1+A_2+A_3.
\end{align*}
As for $A_1$ and $A_2$, we can ignore $e^{t\lambda_\varepsilon^\pm(\nabla)}$ because of the estimate \eqref{L2EST1}.
We write
\begin{align*}
\phi_\varepsilon^+-\phi_0=\frac{(\alpha-\sqrt{\alpha^2-\varepsilon^2\langle\nabla\rangle^2})\phi_0+\varepsilon\phi_1}{2\sqrt{\alpha^2-\varepsilon^2\langle\nabla\rangle^2}}%,\qquad
=-\phi_\varepsilon^-.
%=\frac{(-\alpha+\sqrt{\alpha^2-\varepsilon^2\langle\nabla\rangle^2})\phi_0-\varepsilon\phi_1}{2\sqrt{\alpha^2-\varepsilon^2\langle\nabla\rangle^2}}.
\end{align*} 
By Proposition \ref{prop:base}, we can estimate
$$
\|{\bf1}_{\{\varepsilon\langle\nabla\rangle\le|\alpha|/\sqrt{2}\}}(\phi_\varepsilon^+-\phi_0)\|_{H^\sigma}
+\|{\bf1}_{\{\varepsilon\langle\nabla\rangle\le|\alpha|/\sqrt{2}\}}\phi_\varepsilon^-\|_{H^\sigma}
\lesssim_{\alpha}
 \varepsilon\|\phi_0\|_{H^{\sigma+1}} +\varepsilon\|\phi_1\|_{H^{\sigma}},
$$
where we have used the estimate $\alpha-\sqrt{\alpha^2-\varepsilon^2\langle k\rangle^2}=O(\varepsilon^2\langle k\rangle^2)=O(\varepsilon \langle k\rangle)$ in the region $\varepsilon\langle k\rangle\le\frac{|\alpha|}{\sqrt2}$
%$\sqrt{\alpha^2-\varepsilon^2\langle k\rangle^2}=\alpha-\frac{\varepsilon^2\langle k\rangle^2}{2\alpha}+O(\varepsilon^4\langle k\rangle^4)$ 
from Proposition \ref{prop:base}-\eqref{base:item4} for the term concerning $\phi_0$, and  $\sqrt{\alpha^2-\varepsilon^2\langle k\rangle^2}=O(1)$ for the term concerning $\phi_1$.
%On the other hand, 
We have already proved another estimate
%we can estimate
$$
\|{\bf1}_{\{\varepsilon\langle\nabla\rangle\le|\alpha|/\sqrt{2}\}}(\phi_\varepsilon^+-\phi_0)\|_{H^\sigma}
+\|{\bf1}_{\{\varepsilon\langle\nabla\rangle\le|\alpha|/\sqrt{2}\}}\phi_\varepsilon^-\|_{H^\sigma}
\lesssim\|\phi_0\|_{H^{\sigma}}+\|\phi_1\|_{H^{\sigma-1}},
$$
in Lemma \ref{lem phipm and fpm}-\eqref{eq1:lem phipm and fpm}.
%by replacing $\sqrt{\alpha^2-\varepsilon^2\langle k\rangle^2}$ by $1$ or $\varepsilon\langle k\rangle$, 
%from Proposition \ref{prop:base}-\eqref{base:item2} and \eqref{base:item2'}, similarly to Lemma \ref{lem phipm and fpm}.
By the interpolation we have the required estimate for $A_1$ and $A_2$.
As for $A_3$, we have
$$
\|{\bf1}_{\{\varepsilon\langle\nabla\rangle\le|\alpha|/\sqrt{2}\}}A_3\|_{H^\sigma}
\lesssim\|\phi_0\|_{H^\sigma}
$$
by the $\varepsilon$-uniform estimate \eqref{L2EST1}.
On the other hand,
since in the region $\varepsilon\langle k\rangle\le|\alpha|/\sqrt{2}$ we have
\begin{align*}
|e^{t\lambda_\varepsilon^+(k)}-e^{-\frac{t}{2\alpha}\langle k\rangle^2}|
&\lesssim t\big|\lambda_\varepsilon^+(k)+\frac1{2\alpha}\langle k\rangle^2\big|\,\Big(|e^{t\lambda_\varepsilon^+(k)}|\vee|e^{-\frac{t}{2\alpha}\langle k\rangle^2}|\Big)\\
&\lesssim t\varepsilon^2\langle k\rangle^4e^{-ct\langle k\rangle^2}\\
&\lesssim\varepsilon^2\langle k\rangle^2e^{-c't\langle k\rangle^2}
\end{align*}
for some constants $c>c'>0$ (we used the fact $xe^{-cx}\lesssim e^{-c'x}$ for $x>0$), we have
\begin{align*}
\|{\bf1}_{\{\varepsilon\langle\nabla\rangle\le|\alpha|/\sqrt{2}\}}A_3\|_{H^\sigma}
\lesssim\varepsilon^2\|\phi_0\|_{H^{\sigma+2}}.
\end{align*}
By the interpolation, we have the required estimate for $A_3$.

Next we consider the case $\phi_0=\phi_1=0$.
We decompose
\begin{align*}
u_\varepsilon-v
&=\int_0^t\frac{\alpha-\sqrt{\alpha^2-\varepsilon^2\bracknabla^2}}{2\alpha\sqrt{\alpha^2-\varepsilon^2\bracknabla^2}}e^{(t-s)\lambda_\varepsilon^+(\nabla)}f(s)ds\\
&\quad-\int_0^t\frac1{2\sqrt{\alpha^2-\varepsilon^2\bracknabla^2}}e^{(t-s)\lambda_\varepsilon^-(\nabla)}f(s)ds+\int_0^t\frac1{2\alpha}\big(e^{\frac{t-s}{2\alpha}(\Delta-1)}-e^{(t-s)\lambda_\varepsilon^+(\nabla)}\big)f(s)ds.
%&=:B_1(t)+B_2(t)+B_3(t).
\end{align*}
Then we can estimate each term by a similar way to above.
\end{proof}

\subsection{Probabilistic result}\label{subsection6.2}   

We consider the solutions $\{\Psi_{\varepsilon(j)}\}_{j\in\mathbb{N}}$ according to $\varepsilon(j)=j^{-1}$.
Theorem \ref{mainthm:NRL} is a consequence of the following proposition and Proposition \ref{mainprop:GWP}.

\begin{prop}\label{prop:probability part of NRL}
 Let $\delta>0$. 
Let $(Z_{\varepsilon(j)},Y_{\varepsilon(j)})$ be the solution of \eqref{eq:YZ} with deterministic initial condition $(Z_{\varepsilon(j)},Y_{\varepsilon(j)})|_{t=0}=(z,y)\in\mathcal{H}^{-\delta}$, and let $Z$ be the solution of
$$
2\alpha\partial_tZ+(1-\Delta)Z=2\sqrt{\re(\alpha)}\partial_tW
$$
with $Z|_{t=0}=z$. %Note that $z$ and $y$ are deteministic functions. 
Then there exists a measurable set $\mathcal{A}\subset\mathcal{H}^{-\delta}$ such that $\mu(\mathcal{A})=1$ and 
for any $(z,y)\in\mathcal{A}$, $m,n\in\mathbb{N}$, $T>0$, and $p\in[1,\infty)$, the sequence of Wick products $:Z_{\varepsilon(j)}^m\overline{Z}_{\varepsilon(j)}^n:$ converges to $:Z^m\overline{Z}^n:$ in $C([0,T];W^{-\delta,\infty})$ almost surely.
\end{prop}

\begin{proof}
We first show the $(N,t,x)$-uniform estimate
$$ 
\int_{\mathcal{H}^{-\delta}}
\E\left[|\langle \nabla \rangle^{-\delta}(H_{m,n}(\Pi_N Z_{\varepsilon}(t,x); C_N)-H_{m,n}(\Pi_N Z(t,x);C_N))|^2\right]
\mu(dzdy)
\lesssim\varepsilon^\theta
$$
for any $\theta\in(0,1]$.
For simplicity, we write $\mathcal{Z}_\varepsilon(x)=H_{m,n}(\Pi_NZ_\varepsilon(t,x);C_N)$ and $\mathcal{Z}(x)=H_{m,n}(\Pi_NZ(t,x);C_N)$.
Following \cite[Proposition 2.1]{GKO}, we decompose
\begin{align*}
&\E\left[|\langle\nabla\rangle^{-\delta}(\mathcal{Z}_\varepsilon-\mathcal{Z})(x)|^2\right]\\
&=\int \langle\nabla\rangle^{-\delta}(x,x_1) \langle \nabla \rangle^{-\delta}(x,x_2)  
\mathbb{E}\left[(\mathcal{Z}_\varepsilon(x_1)-\mathcal{Z}(x_1))(\overline{\mathcal{Z}_\varepsilon(x_2)-\mathcal{Z}(x_2)})\right]dx_1dx_2,
\end{align*}
where $\langle\nabla\rangle^{-\delta}(x,\cdot)$ is the Bessel potential.
We decompose
\begin{align*}
&\mathbb{E}\left[(\mathcal{Z}_\varepsilon(x_1)-\mathcal{Z}(x_1))(\overline{\mathcal{Z}_\varepsilon(x_2)-\mathcal{Z}(x_2)})\right]\\
&=\mathbb{E}\left[\mathcal{Z}_\varepsilon(x_1)\overline{\mathcal{Z}_\varepsilon(x_2)}\right]
-\mathbb{E}\left[\mathcal{Z}(x_1)\overline{\mathcal{Z}_\varepsilon(x_2)}\right]
-\mathbb{E}\left[\mathcal{Z}_\varepsilon(x_1)\overline{\mathcal{Z}(x_2)}\right]
+\mathbb{E}\left[\mathcal{Z}(x_1)\overline{\mathcal{Z}(x_2)}\right]\\
&=:\mathrm{(i)}- \mathrm{(ii)}- \mathrm{(iii)}+ \mathrm{(iv)}.
\end{align*}
By the definition of the complex Hermite polynomials,  
\begin{eqnarray*}
\mathrm{(i)}&=&m! n! \E(\overline{\Pi_N Z_{\varepsilon}(t,x_1)} \Pi_N Z_{\varepsilon}(t,x_2))^m \E(\Pi_N Z_{\varepsilon}(t,x_1)\overline{\Pi_N Z_{\varepsilon}(t,x_2)})^n, \\
\mathrm{(ii)}&=&m! n! \E(\overline{\Pi_N Z(t,x_1)} \Pi_N Z_{\varepsilon}(t,x_2))^m \E(\Pi_N Z (t,x_1) \overline{\Pi_N Z_{\varepsilon}(t,x_2)})^n,\\
\mathrm{(iii)}&=&m! n! \E(\overline{\Pi_N Z_{\varepsilon}(t,x_1)} \Pi_N Z(t,x_2))^m \E(\Pi_N Z_{\varepsilon}(t,x_1)\overline{\Pi_N Z (t,x_2)})^n,\\
\mathrm{(iv)}&=&m! n! \E(\overline{\Pi_N Z(t,x_1)} \Pi_N Z(t,x_2))^m \E(\Pi_N Z(t,x_1)\overline{\Pi_N Z(t,x_2)})^n,
\end{eqnarray*}

which can be reorganized as the finite sum of the form
$$
(\mathrm{(i)}- \mathrm{(ii)})- (\mathrm{(iii)}- \mathrm{(iv)}) 
=\sum_{j} \delta C^{j} (x_1,x_2)\prod_{i=1}^{n+m-1}C_i^{j} (x_1,x_2),
$$
where $\delta C^j (x_1,x_2)$ is of the form $\mathbb{E}[(\overline{\delta A})B]$ or $\mathbb{E}[(\delta A)\overline{B}]$ with
\begin{align*}
[(\delta A)B]&=[(\Pi_NZ_\varepsilon(t,x_1)-\Pi_NZ(t,x_1))(\Pi_NZ_\sigma(t,x_2))],~\mbox{or} \\
&\quad [(\Pi_NZ_\varepsilon(t,x_2)-\Pi_NZ(t,x_2))(\Pi_NZ_\sigma(t,x_1))],
\end{align*}
and $C_i^j (x_1,x_2)$ ($i=1,\dots,n+m-1$) are of the form $\mathbb{E}[A\overline{B}]$ with
\begin{align*}
[AB]&=[(\Pi_NZ_\sigma(t,x_1))(\Pi_NZ_\eta(t,x_2))],~ \mbox{or}~  [(\Pi_NZ_\sigma(t,x_2))(\Pi_NZ_\eta(t,x_1))]
\end{align*}
where $\sigma,\eta$ runs over $\{0,\varepsilon\}$ and set $Z_0:=Z$.
For the terms $C_i^j$, we can decompose
\begin{align*}
\mathbb{E}_\mu[A\overline{B}]
=\sum_{|k|\le N}\mathbb{E}_\mu\left[\hat{Z}_\sigma(t;k)\overline{\hat{Z}_\eta(t;k)}\right]e_k(x_1-x_2).
\end{align*}
%by the invariance and Gaussianity of $\mu$
Indeed, $\mathbb{E}_\mu\left[\hat{Z}_\sigma(t;k)\overline{\hat{Z}_\eta(t;\ell)}\right]=0$ unless $k=\ell$, since
$\hat{Z}_\sigma(t;k)$ is a linear combination of $\hat{z}(k),\hat{y}(k)$, and a Wiener integral with respect to $\hat{W}(k)$,
and for $k\neq\ell$, $\hat{z}(k),\hat{y}(k),\hat{W}(k)$ and $\overline{\hat{z}(\ell)},\overline{\hat{y}(\ell)},\overline{\hat{W}(\ell)}$ are mutually uncorrelated under the probability $\mathbb{P}_\mu$.
By the Cauchy-Schwarz inequality and the invariance of $\mu$,
%and 
we have the $\varepsilon$-independent bound
\begin{align*}
\left|\mathbb{E}_\mu\left[\hat{Z}_\sigma(t;k)\overline{\hat{Z}_\eta(t;k)}\right]\right|
\lesssim
\left[\mathbb{E}_\mu|\hat{Z}_\sigma(t;k)|^2\right]^{\frac12}
\left[\mathbb{E}_\mu|\hat{Z}_\eta(t;k)|^2\right]^{\frac12}
\lesssim\frac1{\langle k\rangle^2}.
\end{align*}
We can repeat a similar computation of $\delta C^j$, except to use the bound
\begin{align*}
\left|\mathbb{E}_\mu\left[(\hat{Z}_\varepsilon-\hat{Z})(t;k)\overline{\hat{Z}_\sigma(t;k)}\right]\right|
\lesssim\varepsilon^\theta\frac1{\langle k\rangle^{2+\theta}},
\end{align*}
where $\theta\in(0,1]$. This bound is similarly obtained by the proof of Theorem \ref{mainthm:simple}.
Consequently, we have the estimate 
\begin{align*}
&\E_\mu\left[|\langle\nabla\rangle^{-\delta}(\mathcal{Z}_\varepsilon-\mathcal{Z})(x)|^2\right]\\
& \lesssim \varepsilon^\theta \sum_{k_1,\dots,k_{n+m}} \int \langle\nabla\rangle^{-\delta}(x,x_1)\langle\nabla\rangle^{-\delta}(x,x_2)
\frac1{\langle k_{n+m} \rangle^{2+\theta}} \prod_{i=1}^{n+m-1} \frac1{\langle k_i\rangle^2}e_{k_1+\cdots+k_{n+m}}(x_1-x_2)
dx_1dx_2\\
&\lesssim
\varepsilon^\theta
\sum_{k_1,\dots,k_{n+m}}
\frac1{\langle k_1+\cdots+k_{n+m}\rangle^{2\delta}}\frac1{\langle k_{n+m} \rangle^{2+\theta}}\prod_{i=1}^{n+m-1}\frac1{\langle k_i\rangle^2}\\
&\lesssim\varepsilon^\theta,
\end{align*}
where we have used the notation such that the frequencies $k_i \in \Z^d$ corresponds to $C_i^j$ ($i=1,..., n+m-1$) 
and $k_{n+m}$ to $\delta C^j$.
Thus, in a similar way to \cite{GKO}, for any $h\in [-1,1]$ and $\alpha_0\in(0,2\delta)$,
%$\alpha_0 \in [0,1]$ such that $\alpha_0<2\delta$, 
\begin{eqnarray*}
\E_\mu\left[|\delta_h \left\{\langle \nabla \rangle^{-\delta} (H_{m,n}(\Pi_N Z_{\varepsilon}(t,x); C_N)-H_{m,n}(\Pi_N Z(t,x);C_N))\right\}|^2\right] 
\lesssim |h|^{\alpha_0} \varepsilon^\theta, \\
\end{eqnarray*}
where $\delta_hf(t):=f(t+h)-f(t)$.
Thus, Nelson estimate, Sobolev embedding, and Kolmogorov criterion imply the convergence 
$$ \E_\mu\left[\|H_{m,n}(\Pi_N Z_{\varepsilon}(t,x); C_N)-H_{m,n}(\Pi_N Z(t,x);C_N))\|_{C_T W_x^{-\delta,p}}^p\right] \lesssim \varepsilon^{\theta p/2}$$
for any $p\in[1,\infty)$.
Moreover, by Proposition \ref{prop:wick_N}, as $N\to \infty$
$$ \E_\mu\Big[ \big\| :Z_{\varepsilon}^m \overline{Z}_{\varepsilon}^n :
- :Z^m\overline{Z}^n: \big\|^p_{C_T W_x^{-\delta,\infty}} \Big]  \lesssim \varepsilon^{\theta p/2}. $$
 As for the almost sure convergence result, we use a similar argument to \cite[Proposition 3.2]{OPT}.
For any $k\in\mathbb{N}$ and $\varepsilon>0$, set
$$
\mathbb{A}_\varepsilon^k
=\Big\{((z,y),\omega)\in\mathcal{H}^{-\delta}\times\Omega\ ;\ \big\| :Z_{\varepsilon}^m \overline{Z}_{\varepsilon}^n :
- :Z^m\overline{Z}^n: \big\|_{C_T W_x^{-\delta,\infty}}<k^{-1}\Big\}.
$$
By Markov inequality,
\begin{align*}
\mu\otimes\mathbb{P}((\mathbb{A}_\varepsilon^k)^c)
&=\mu\otimes\mathbb{P}\Big[ \big\| :Z_{\varepsilon}^m \overline{Z}_{\varepsilon}^n :
- :Z^m\overline{Z}^n: \big\|_{C_T W_x^{-\delta,\infty}}\ge k^{-1}
 \Big]\\
&\lesssim k^p
\E_\mu\Big[ \big\| :Z_{\varepsilon}^m \overline{Z}_{\varepsilon}^n :
- :Z^m\overline{Z}^n: \big\|_{C_T W_x^{-\delta,\infty}}^p
 \Big]\\
&\lesssim k^p\varepsilon^{\theta p/2}.
\end{align*}
We then have, taking a subsequence $\varepsilon(j)=j^{-1}$ and setting $p$ to be $\theta p>2$,
$$
\sum_{j=1}^\infty\mu\otimes\mathbb{P}((\mathbb{A}_{\varepsilon(j)}^k)^c)
\lesssim k^p\sum_{j=1}^\infty j^{-\theta p/2}<\infty.
$$
By Borel-Cantelli lemma, the event $\mathbb{A}^k:=\bigcup_{j=1}^\infty\bigcap_{i\ge j}\mathbb{A}_{\varepsilon(i)}^k$ 
has probability one, and for any $((z,y),\omega)\in\mathbb{A}^k$,
$$
\limsup_{j\to\infty}\big\| :Z_{\varepsilon(j)}^m \overline{Z}_{\varepsilon(j)}^n :
- :Z^m\overline{Z}^n: \big\|_{C_T W_x^{-\delta,\infty}}\le k^{-1}.
$$
Therefore, the event $\mathbb{A}:=\bigcap_{k\in\mathbb{N}}\mathbb{A}^k$ also has probability one, and for any $((z,y),\omega)\in\mathbb{A}$,
$$
\lim_{j\to\infty}\big\| :Z_{\varepsilon(j)}^m \overline{Z}_{\varepsilon(j)}^n :
- :Z^m\overline{Z}^n: \big\|_{C_T W_x^{-\delta,\infty}}=0,
$$
i.e.  $\{:Z_{\varepsilon(j)}^m \overline{Z}_{\varepsilon(j)}^n:\}_j$ converges in $C_TW_x^{-\delta,\infty}$, almost  surely, 
for almost every initial condition $(z,y)$.

%We then have, taking a subsequence of $\varepsilon$, 
%\begin{align*}
%\int_{\mathcal{H}^{-\delta}}
%\mathbb{E}\Big[ \big\| :Z_{\varepsilon(j)}^m \overline{Z}_{\varepsilon(j)}^n :
%- :Z^m\overline{Z}^n: \big\|_{C_T W^{-\delta,\infty}} \Big] \mu(dzdy)
%\lesssim\varepsilon(j)^\theta=2^{-\theta j}.
%\end{align*}
%This implies
%\begin{align*}
%\sum_{j=0}^\infty \mathbb{E}\Big[ \big\| :Z_{\varepsilon(j)}^m \overline{Z}_{\varepsilon(j)}^n :
%- :Z^m\overline{Z}^n: \big\|_{C_T W^{-\delta,\infty}} \Big]
%<\infty
%\end{align*}
%for $\mu$-almost every $(z,y)$. Therefore 
%$$ \sum_{j=0}^\infty \big\| :Z_{\varepsilon(j)}^m \overline{Z}_{\varepsilon(j)}^n :- :Z^m\overline{Z}^n: \big\|_{C_TW^{-\delta,\infty}}<\infty,$$ 
%almost surely, i.e.  $\{:Z_{\varepsilon(j)}^m \overline{Z}_{\varepsilon(j)}^n:\}_j$ is a Cauchy sequence in 
%$C_TW^{-\delta,\infty}$, almost  surely, thus converges.
\end{proof}

\begin{proof}[{\bfseries Proof of Theorem \ref{mainthm:NRL}}]
The solution $\Psi_{\varepsilon}$ of \eqref{eq:dampedwave_Wick} is decomposed into the sum $Z_{\varepsilon}+U_{\varepsilon}$, where $U_\varepsilon$ solves
$$
\left\{
\begin{aligned}
&\varepsilon^2\partial_t^2U_\varepsilon+2\alpha\partial_tU_\varepsilon+(1-\Delta)U_\varepsilon
+\sum_{k=0}^{n+1}\sum_{\ell=0}^n\binom{n+1}{k}\binom{n}{\ell}U_\varepsilon^{n+1-k}\overline{U_\varepsilon}^{n-\ell}
:Z_\varepsilon^k\overline{Z_\varepsilon}^\ell:\ =0,\\
&(U_\varepsilon,\varepsilon\partial_tU_\varepsilon)|_{t=0}=(0,0).
\end{aligned}
\right.
$$
Denote by $S_\varepsilon:\{:Z_\varepsilon^k\overline{Z_\varepsilon}^\ell:\}_{k,\ell}\mapsto U_\varepsilon$ the solution map.
Similarly, the solution $\Psi$ of \eqref{eq:CGL} is decomposed into the sum $Z+U$, where $U$ solves
$$
\left\{
\begin{aligned}
&2\alpha\partial_tU+(1-\Delta)U
+\sum_{k=0}^{n+1}\sum_{\ell=0}^n\binom{n+1}{k}\binom{n}{\ell}U^{n+1-k}\overline{U}^{n-\ell}
:Z^k\overline{Z}^\ell:\ =0,\\
&U|_{t=0}=0.
\end{aligned}
\right.
$$
Denote by $S:\{:Z^k\overline{Z}^\ell:\}_{k,\ell}\mapsto U$ the solution map.
Proposition \ref{prop:probability part of NRL} means the almost sure convergence of $:Z_\varepsilon^k\overline{Z_\varepsilon}^\ell:$ to $:Z^k\overline{Z}^\ell:$ along the subsequence $\varepsilon(j)=j^{-1}$, and Theorem \ref{mainthm:simple} implies the convergence of the solution map $S_\varepsilon$ to $S$ as $\varepsilon\to0$.
These yield the almost sure convergence of $U_\varepsilon$ to $U$ as $\varepsilon\to0$, by a similar argument to 
\cite[Theorem 4.4]{H1} or \cite[Theorems 1.2 and 1.3]{IM}. %We omit the details in this paper.
\end{proof}

%%%%%%%%%%%%%%%%%%%%%%%%%%%%%%%%%%%
\section{ultra-relativistic limit}\label{section:URL}
%%%%%%%%%%%%%%%%%%%%%%%%%%%%%%%%%%%%%

This section is devoted to the proof of Corollary \ref{maincor:URL}.
We can repeat the same argument as in Section \ref{section:NRL}, but the only nontrivial point is in the deterministic part. Since the implicit proportional constants in Theorem \ref{mainthm_energy} are locally bounded in $\alpha\in\mathbb{C}_+\setminus(0,\infty)$, the convergence $\im \alpha\to0$ is out of the reach.
We will show only the modified deterministic estimates in this section.

We fix $\varepsilon=1$, and we write $\alpha_1= \re \alpha$, $\alpha_2=\im \alpha$.
Let any $T>0$ be fixed.
%\subsection{Deterministic result} 
We can set $d\ge1$ to be arbitrary in the deterministic part.
We can expect that, the limit of the solution of \eqref{eq:dampedwave_f} 
\begin{align*}%\label{eq:dampedwave_alpha}
\left\{
\begin{aligned}
&\partial_t^2 u_{\alpha} +2 (\alpha_1+i\alpha_2) \partial_t u_{\alpha} +(1-\Delta) u_{\alpha}=f, &t>0,~ x\in \mathbb{T}^d,\\
&(u_{\alpha}, \partial_t u_{\alpha})|_{t=0}=(\phi_0,\phi_1), &x\in\mathbb{T}^d.
\end{aligned}
\right.
\end{align*}
as $\alpha_2\to0$, solves the damped wave equation with real-valued coefficients, i.e.,    
\begin{align*}%\label{eq:limit ultra}
\left\{
\begin{aligned}
&\partial_t^2 v +2 \alpha_1 \partial_t v +(1-\Delta) v=f, &t>0,~ x\in \mathbb{T}^d,\\
&(v, \partial_t v)|_{t=0}=(\phi_0,\phi_1), &x\in\mathbb{T}^d.
\end{aligned}
\right.
\end{align*}
This convergence is not directly covered by Theorem \ref{mainthm_energy}.
Nevertheless, we have the following modification.

\begin{thm} \label{thm:ultra_uniform}
Let $d\ge 1$ and $\alpha_2\in(-1,1)$. For any $\sigma\in\mathbb{R}$,
\begin{align*}
\|(u_{\alpha},\partial_t u_{\alpha})\|_{L_T^\infty \mathcal{H}_x^\sigma(\mathbb{T}^d)}
\lesssim_{\alpha_1,T} \|\phi_0\|_{H^\sigma(\mathbb{T}^d)}+\|\phi_1\|_{H^{\sigma-1}(\mathbb{T}^d)}+\|f\|_{L_T^2 H_x^{\sigma-1}(\mathbb{T}^d)},
\end{align*}
where the implicit proportional constants are locally bounded function of $\alpha_1>0$.
\end{thm}

\begin{proof}
The proof of Theorem \ref{mainthm_energy} is a combination of Proposition \ref{prop L2EST} and Lemma \ref{lem phipm and fpm}.
Proposition \ref{prop L2EST} is $\alpha_2$-independent, but Lemma \ref{lem phipm and fpm} is $\alpha_2$-dependent.
Note that $\lambda_\alpha^\pm(\nabla)$ is of the form
$$
\lambda_\alpha^\pm(\nabla)=-\alpha\pm\sqrt{\alpha^2-\langle\nabla\rangle^2},
$$
so we can use Proposition \ref{prop:base}-\eqref{base:item4} to replace $\sqrt{\alpha^2-\langle k\rangle^2}$ by $1$ or $\langle k\rangle$, 
module proportional constants depending only on $\alpha_1$, provided 
\begin{itemize}
 \item[(i)] $\langle k \rangle^2 < |\alpha|^2 /9$,
 \item[(ii)] $12 |\alpha|^2 < \langle k \rangle^2$.  
\end{itemize}
Indeed, in the case of (i), making use of Proposition \ref{prop:base}-\eqref{base:item4}, we have 
$$ \left|-\frac{\langle k \rangle^2}{2\alpha} +h(\langle k \rangle^2, \alpha)\right| \le \frac{|\alpha|}{2}.$$
Thus, $$ |\sqrt{\alpha^2-\langle k \rangle^2}| \ge \frac{|\alpha|}{2} \ge \frac{1}{2} \re(\alpha).$$
In the case of (ii), $|g( \langle k \rangle^2, \alpha)| <\frac12 \langle k \rangle$, thus 
again by Proposition \ref{prop:base}-\eqref{base:item4}, 
$$ |\sqrt{\alpha^2-\langle k \rangle^2}| \ge \frac12 \langle k \rangle \ge \frac12.$$

Once, as above, the lower bound $C$ of $\sqrt{\alpha^2-\langle k \rangle^2}$ is shown independent of $\alpha_2$, 
we have 
\begin{align*}
\left|\frac{\langle k \rangle}{\sqrt{\alpha^2-\langle k\rangle^2}}\right|^2
&=\frac{\langle k \rangle^2}{|\alpha^2-\langle k \rangle^2|}
\le 1+\frac{|\alpha^2|}{|\alpha^2-\langle k \rangle^2|} \le 1+\frac{|\alpha^2|}{C^2}. 
\end{align*}
Hence we can obtain the same estimates as in Lemma \ref{lem phipm and fpm}.
%modify Lemma \ref{lem phipm and fpm} and derive
%\begin{align*}
%\|\phi_{\alpha}^\pm\|_{H^\sigma}\lesssim_{\alpha_1}\|\phi_0\|_{H^\sigma}+\|\phi_1\|_{H^{\sigma-1}},\quad
%\|f_{\alpha}^\pm\|_{H^{\sigma}}\lesssim_{\alpha_1}\|f\|_{H^{\sigma-1}}.
%\end{align*}

The case 
\begin{itemize}
\item[(iii)]  $|\alpha|^2 /9 \le \langle k \rangle^2 \le 12 |\alpha|^2$
\end{itemize}
still remains.
In this case, we go back to the mild form and rewrite it as: 
\begin{eqnarray*}
u_{\alpha}(t)&=& e^{-\alpha t}\cosh(\sqrt{\alpha^2-\langle \nabla \rangle^2}t) \phi_0 
+e^{-\alpha t} \frac{\sinh(\sqrt{\alpha^2-\langle \nabla \rangle^2}t)}{\sqrt{\alpha^2-\langle \nabla \rangle^2}}(\alpha \phi_0+\phi_1) \\
&& +\int_0^t e^{-\alpha(t-t')} \frac{\sinh(\sqrt{\alpha^2-\langle \nabla \rangle^2}(t-t'))}{\sqrt{\alpha^2-\langle \nabla \rangle^2}} f(t') dt'.
\end{eqnarray*}
Remark that there exists a $\delta>0$ such that if $|z| \le \delta$ then $\displaystyle{\left|\frac{\sinh z}{z}\right| \le \frac32}$.
For any $g \in H^{\sigma-1}$,  
\begin{eqnarray} \label{L2EST4}
\left\|e^{-\alpha t} \frac{\sinh(\sqrt{\alpha^2-\langle \nabla \rangle^2}t)}{\sqrt{\alpha^2-\langle \nabla \rangle^2}} g \right\|_{H^{\sigma}} 
&=&  
\sum_{k\in \Z^d} \langle k \rangle^{2\sigma} \left|e^{-\alpha t} \frac{\sinh(\sqrt{\alpha^2-\langle k \rangle^2}t)}{\sqrt{\alpha^2-\langle k \rangle^2}} \hat{g}(k) \right|^2. 
\end{eqnarray}
If $|T\sqrt{\alpha^2-\langle k \rangle^2}| \le \delta$ for any $t \le T$, 
\begin{eqnarray*}
\left|e^{-\alpha t} \frac{\sinh(\sqrt{\alpha^2-\langle k \rangle^2}t)}{\sqrt{\alpha^2-\langle k \rangle^2}} \right|^2
&\le& \frac{9}{4}  e^{-2\re(\alpha) t} t^2 \lesssim \frac{1}{(\re(\alpha))^2},  
\end{eqnarray*}
or, since we are in the case where $\langle k \rangle^2 \le 12 |\alpha|^2$,
\begin{eqnarray*}
\left|e^{-\alpha t} \frac{\sinh(\sqrt{\alpha^2-\langle k \rangle^2}t)}{\sqrt{\alpha^2-\langle k \rangle^2}} \right|^2
&\le& \frac{9}{4}  e^{-2\re(\alpha) t} t^2 \frac{12 |\alpha|^2}{\langle k \rangle^2} 
\lesssim \frac{|\alpha|^2}{(\re(\alpha))^2 \langle k \rangle^2}. 
\end{eqnarray*}
If $|T\sqrt{\alpha^2-\langle k \rangle^2}| > \delta$, 
\begin{eqnarray*}
\left|e^{-\alpha t} \frac{\sinh(\sqrt{\alpha^2-\langle k \rangle^2}t)}{\sqrt{\alpha^2-\langle k \rangle^2}} \right|^2
&\le & \frac{T^2}{2\delta^2} (|e^{(t-t')\lambda_{\alpha}^{+}(k)}|^2 + |e^{(t-t')\lambda_{\alpha}^{-}(k)}|^2) \le \frac{2T^2}{\delta^2}, 
\end{eqnarray*}
or, similarly, using the fact $\langle k \rangle^2 \le 12 |\alpha|^2$, is bounded by 
$\lesssim \frac{T^2}{\delta^2} \frac{|\alpha|^2}{\langle k \rangle^2}$,
from which we see that summing up both case of $k$, the sum (\ref{L2EST4}) may be bounded by $\|g\|_{H^{\sigma-1}}$ 
($\|g\|_{H^{\sigma}}$ if $g \in H^{\sigma}$) independent of small $\alpha_2$, but 
dependent of $\alpha_1$, $\delta$ and $T$. For the inhomogeneous term, we can similarly estimate in $H^{\sigma}$ independent of small $\alpha_2$ 
with the aid of (\ref{L2EST2}). 
\end{proof}

Since we have the $\alpha_2$-uniform estimate, we can prove Corollary \ref{maincor:URL} by a similar way to Section \ref{section:NRL}.

\begin{thm} 
For any $\sigma\in\mathbb{R}$,
\begin{align*}
\|(u_\alpha, \partial_t u_{\alpha}) -(v, \partial_t v) \|_{L_T^\infty \mathcal{H}_x^{\sigma}}
\lesssim_{\alpha_1,T} |\alpha_2|\big(\|\phi_0\|_{H^{\sigma}}+\|\phi_1\|_{H^{\sigma-1}}
+\|f\|_{L_T^2H_x^{\sigma-1}}\big).
\end{align*}
\end{thm}

\begin{proof}
Since $\lambda_\alpha^\pm(\nabla)$ is smooth in $\alpha_2$, we can see that the derivative $\bar{u}_\alpha=\frac{\partial}{\partial\alpha_2}u_\alpha$ solves the equation
\begin{align*}
\left\{
\begin{aligned}
&\partial_t^2 \bar{u}_{\alpha} +2 (\alpha_1+i\alpha_2) \partial_t \bar{u}_{\alpha} +(1-\Delta) \bar{u}_{\alpha}=-2i\partial_tu_\alpha,\\
&(\bar{u}_{\alpha}, \partial_t \bar{u}_{\alpha})|_{t=0}=(0,0).
\end{aligned}
\right.
\end{align*}
Then by the uniform estimate in Theorem \ref{thm:ultra_uniform},
\begin{align*}
\|(u_\alpha, \partial_t u_{\alpha}) -(v, \partial_t v) \|_{L_T^\infty \mathcal{H}_x^{\sigma}}
&\le|\alpha_2|\sup_{\alpha_2'\in[0,\alpha_2]}\|(\bar{u}_{\alpha_1+\alpha_2'i},\partial_t\bar{u}_{\alpha_1+\alpha_2'i})\|_{L_T^\infty \mathcal{H}_x^\sigma}\\
&\lesssim_{\alpha_1,T} |\alpha_2|\|\partial_tu_\alpha\|_{L_T^\infty H_x^{\sigma-1}}\\
&\lesssim_{\alpha_1,T} |\alpha_2|\big(\|\phi_0\|_{H^{\sigma}}+\|\phi_1\|_{H^{\sigma-1}}
+\|f\|_{L_T^2H_x^{\sigma-1}}\big).
\end{align*}
\end{proof}

\section{uniform Strichartz estimates} \label{section7} 
%%%%%%%%%%%%%%%%%%%%%%%%%%%%%%%%%%%%%%%%%%%%%%%%%%%%%%%

This section is devoted to the proof of Theorem \ref{mainthm}. 
Similarly to Section \ref{section:NRL}, we sometimes omit the label $\alpha$ from underlined objects.
First we remark that we can replace $1-\Delta$ with $\Delta$.
Let $u_\varepsilon= u_{\varepsilon, \alpha}$ be 
the solution of (\ref{eq:dampedwave_f}). Then, by the transform 
\begin{equation} \label{trans}
v_{\varepsilon}(t)= e^{\beta t} u_{\varepsilon}(t), \quad \beta=\frac{\alpha-\sqrt{\alpha^2-\varepsilon^2}}{\varepsilon^2},
\end{equation}  
we see that $v_{\varepsilon}$ satisfies 
\begin{align}\label{eq:massless}
\left\{
\begin{aligned}
&\varepsilon^2\partial_t^2v_\varepsilon+2\sqrt{\alpha^2 -\varepsilon^2}\partial_t v_\varepsilon -\Delta v_\varepsilon=e^{\beta t}f, &t>0,~ x\in \mathbb{T}^d,\\
&(v_\varepsilon,\varepsilon\partial_t v_\varepsilon)|_{t=0}=(\phi_0, \varepsilon \beta \phi_0+\phi_1), &x\in\mathbb{T}^d.
\end{aligned}
\right.
\end{align}
Note that by Proposition \ref{prop:base}, setting $\gamma:=\sqrt{\alpha^2-\varepsilon^2}$, we have  
$\re \gamma >0$, $\im \gamma \ne 0$, $|\varepsilon \beta| \le \frac{\varepsilon}{\re \alpha}$,
and
$$
\re\beta\ge \frac{C_{\alpha,\varepsilon}\varepsilon^2}{\varepsilon^2}=C_{\alpha,\varepsilon}\ge C_{\alpha,1},
$$
where the constants $C_{\alpha,\varepsilon}$, 
$C_{\alpha,1}$ are defined in (\ref{base:item3}) of Proposition \ref{prop:base}. 
%$\re \beta \ge 0$ and 
%$|\varepsilon \beta| \le \frac{\varepsilon}{\re \alpha}$. 
Remark that this transform 
has been used already in several studies (for example, see \cite{BRS}).
\vspace{3mm}
   
We can write the solution of (\ref{eq:massless}) in the mild form.
\begin{align} \label{v:mild}
v_{\varepsilon}(t)=e^{t\lambda_{\varepsilon}^+(\nabla)}\phi_
{\varepsilon}^+
+e^{t\lambda_{\varepsilon}^-(\nabla)}\phi_{\varepsilon}^-
+\int_0^t\left(e^{(t-t')\lambda_{\varepsilon}^+(\nabla)}f_{\varepsilon}^+(t')
+e^{(t-t')\lambda_{\varepsilon}^-(\nabla)}f_{\varepsilon}^-(t')\right)dt',
\end{align}
where
%where,
$$
\lambda_{\varepsilon}^\pm(\nabla)=\frac{-\gamma\pm\sqrt{{\gamma}^2-\varepsilon^2 |\nabla|^2}}{\varepsilon^2},
$$
and 
$$
\phi_{\varepsilon}^\pm
=\frac{\mp\varepsilon^2\lambda_{\varepsilon}^\mp(\nabla)\phi_0 \pm \varepsilon(\varepsilon \beta \phi_0+\phi_1)}{2\sqrt{\gamma^2-\varepsilon^2 |\nabla|^2}},\quad
f_{\varepsilon}^\pm(t)=\pm\frac{e^{\beta t}f(t)}{2\sqrt{\gamma^2-\varepsilon^2 |\nabla|^2}}. 
$$  
\vspace{3mm}  

In the following Sections \ref{sec:L} and \ref{sec:H}, 
we decompose the operator $e^{t\lambda_\varepsilon^\pm(\nabla)}$ into
$$
L_\varepsilon^\pm(t)=e^{t\lambda_\varepsilon^\pm(\nabla)}I(\varepsilon\nabla),\quad
H_\varepsilon^\pm(t)=e^{t\lambda_\varepsilon^\pm(\nabla)}(1-I(\varepsilon\nabla)),
$$
where  $I$ is a radial smooth cut-off function $I:\mathbb{R}^d\to[0,1]$ defined by (\ref{def:cutoff}).
We will establish the uniform Schauder/Strichartz estimate for $L_\varepsilon^\pm$ and $H_\varepsilon^\pm$, i.e. %for $v_{\varepsilon}$, i.e. 
Theorems \ref{thm:L} and \ref{thm:H} below. We will then obtain   
the Strichartz estimate for $u_{\varepsilon}$ in Theorem \ref{mainthm} via the transform (\ref{trans}). 
Note that we consider in what follows $d \ge 1$ if nothing is mentioned.

\subsection{Low frequency part}\label{sec:L}%%%%%%%%%%%%%%%%%%%%%%%%%%%%%%%%%%%%%%%%

\begin{lem}\label{kernel low}
There exist a constant $c>0$ such that, for any $k\in\mathbb{N}^d$,
\begin{align}
\label{kernel low +}&|\partial_\xi^k (e^{t\lambda_\varepsilon^+(\xi)}I(\varepsilon\xi))|
\lesssim e^{-ct|\xi|^2}|\xi|^{-|k|},\\
\label{kernel low -}&|\partial_\xi^k (e^{t\lambda_\varepsilon^-(\xi)}I(\varepsilon\xi))|
\lesssim \varepsilon^{|k|}e^{-ct/\varepsilon^2}.
\end{align}
\end{lem}

Note that the right hand side of \eqref{kernel low -} is bounded by that of \eqref{kernel low +} in the region $|\varepsilon\xi|\le2$.

\begin{proof}
Since $e^{t\lambda_\varepsilon^\pm(\xi)}=e^{(t/\varepsilon^2)\lambda_1^\pm(\varepsilon\xi)}$, it is sufficient to consider the case $\varepsilon=1$.
Since $\re\lambda_1^-(\xi)\le -\re(\gamma)$ and since all derivatives $\partial^k\lambda_1^-$ are bounded on $|\xi|\le2$,
 by Lemma \ref{Faadi Bruno},
$$
|\partial^ke^{t\lambda_1^-(\xi)}|
\lesssim e^{t\re \lambda_1^-(\xi)}\sum_{\sum_\ell m_\ell\ell=k}
\prod_\ell|t\partial^\ell \lambda_1^-(\xi)|^{m_\ell}
\lesssim e^{-ct}
$$
for some constant $c<\re(\gamma)$. Hence we have \eqref{kernel low -}.
On the other hand, by Proposition \ref{prop:base} (\ref{base:item3}) 
there exists a constant $C_{\gamma}>0$ such that
$$
\re \lambda_1^+(\xi)=-\re\frac{|\xi|^2}{\gamma+\sqrt{\gamma^2-|\xi|^2}}\le -C_{\gamma}|\xi|^2,
$$
and moreover,
$$
|\partial_\xi^k\lambda_1^+(\xi)|\lesssim|\xi|^{2-|k|}.
$$
Hence by Lemma \ref{Faadi Bruno},
\begin{align*}
|\partial^ke^{t\lambda_1^+(\xi)}|
&\lesssim e^{t\re \lambda_1^+(\xi)}\sum_{\sum_\ell m_\ell\ell=k}
\prod_\ell|t\partial^\ell \lambda_1^+(\xi)|^{m_\ell}\\
&\lesssim 
e^{-C_{\gamma} t|\xi|^2}\sum (t|\xi|^2)^{\sum_\ell m_\ell}|\xi|^{-|k|}
\lesssim e^{-ct|\xi|^2}|\xi|^{-|k|}
\end{align*}
for some constant $c<C_{\gamma}$. Hence we have \eqref{kernel low +}.
\end{proof}

By Lemma \ref{kernel low}, Lemma \ref{HMmultiplier} and Lemma \ref{lem:Bernstein}, we obtain

\begin{thm}\label{thm:L}
For any $s\ge0$, $1 \le r \le \infty$, 
\begin{align*}
\|L_\varepsilon^\pm(t)\Delta_jf\|_{L^r} \lesssim_{\gamma} t^{-s/2} 2^{-js} \|\Delta_j f\|_{L^r}.
\end{align*}
\end{thm}

\subsection{High frequency part}\label{sec:H}%%%%%%%%%%%%%%%%%%%%%%%%%%%%%

In this section, we prove the following theorem.

\begin{thm}\label{thm:H}
Let $\frac{d-1}2<m\le\infty$ and
let $(q_k,r_k)\in[2,\infty]^2$ ($k=1,2$) be $m$-admissible pairs, that is,
\begin{align}\label{eq:admissible}
\frac1{mq_k}+\frac1{r_k}=\frac12.
\end{align}
Define
$$
s_k=s(r_k)=\frac{d+1}2\left(\frac12-\frac1{r_k}\right),\quad
\delta_k=\delta(q_k,r_k)=\frac2{q_k}-\frac{d-1}2\left(\frac12-\frac1{r_k}\right).
$$
Then for any $j\ge-1$, we have
\begin{align}
\label{Strichartz 1}
\|H_\varepsilon^\pm(t)\Delta_ju\|_{L_T^{q_1}L_x^{r_1}}
&\lesssim_T \varepsilon^{\delta_1}2^{js_1}\|\Delta_ju\|_{L^2},\\
\label{Strichartz 2}
\left\|\int_0^tH_\varepsilon^\pm(t-t')\Delta_jf(t')dt'\right\|_{L_T^{q_1}L_x^{r_1}}
&\lesssim_T \varepsilon^{\delta_1+\delta_2}2^{j(s_1+s_2)}\|\Delta_jf\|_{L_T^{q_2'}L_x^{r_2'}}.
\end{align}
%If $(q_1,r_1)=(\infty,2)$, we can replace $L_T^\infty L_x^2$ on the right hand sides by $C_TL_x^2$.
\end{thm}

%Instead of computing $H_\varepsilon^\pm$ directly, we first consider a simpler operator
%$$
%W_\varepsilon^\pm(t)=e^{-ct/\varepsilon^2}e^{\pm it|\nabla|/\varepsilon}
%$$
%for any fixed $c>0$. 
%As proved in Appendix \ref{app:W}, $W_\varepsilon^\pm$ satisfies the same estimates as Theorem \ref{thm:H}.

Before proving Theorem \ref{thm:H}, remark
%Note 
here that there exists a constant $c>0$ such that
\begin{align*}
\sup_{\xi\notin \text{supp}(I)}\re\sqrt{\gamma^2-|\xi|^2}<\re(\gamma)-c.
\end{align*}
Instead of computing $H_\varepsilon^\pm$ directly, we consider a simpler operator
$$
W_\varepsilon^\pm(t)=e^{-ct/2\varepsilon^2}e^{\pm it|\nabla|/\varepsilon}.
$$
Then we can decompose $H_\varepsilon^\pm(t)=\tilde{H}_\varepsilon^\pm(t)W_\varepsilon^\pm(t)$, where
\begin{align*}
\tilde{H}_\varepsilon^\pm (t) &=e^{-ct/2\varepsilon^2} e^{-(\gamma-c)t/\varepsilon^2}
e^{\pm t(\sqrt{\gamma^2-\varepsilon^2|\nabla|^2}-i\varepsilon|\nabla|)/\varepsilon^2}(1-I(\varepsilon\nabla)).
%W_\varepsilon^\pm&=e^{-ct/\varepsilon^2}e^{\pm it|\nabla|/\varepsilon}.
\end{align*}
Then Theorem \ref{thm:H} follows from Theorem \ref{thm:W} in Appendix \ref{app:W}
once we prove

\begin{lem}  \label{lem:Htilde}
For any $p\in[1,\infty]$, there exists a constant $C$ such that we have
$$
\|\tilde H_\varepsilon^\pm(t)\Delta_jf\|_{L^p}\le C e^{-ct/2\varepsilon^2} \|\Delta_jf\|_{L^p}
$$
for any $j\ge0$, $t\ge0$, and $\varepsilon\in(0,1]$.
\end{lem}

\begin{proof}
Set $\varphi(\xi)=\sqrt{\gamma^2-|\xi|^2}-i|\xi|$. Since $|\partial^k\varphi(\xi)|\lesssim|\xi|^{-|k|}$ ($k \in \N^d$), we have
$$
|\partial^ke^{-(\gamma-c) t\pm \varphi(\xi)t}|\lesssim e^{-\re(\gamma-c) t\pm \re\varphi(\xi)t}t^{|k|}|\xi|^{-|k|}
\lesssim |\xi|^{-|k|}
$$
on the complement of the support of $I$. Consequently,
$$
|\partial^ke^{-(\gamma-c)t/\varepsilon^2\pm\varphi(\varepsilon\xi)t/\varepsilon^2}
(1-I(\varepsilon\xi))|
\lesssim|\xi|^{-|k|}.
$$
By Lemma \ref{HMmultiplier}, we obtain the assertion.
\end{proof}

\subsection{Proof of Theorem \ref{mainthm}}%%%%%%%%%%%%%%%%%%%%%%%%%

Since $u_\varepsilon(t)=e^{-\beta t}v_\varepsilon(t)$ with $\re\beta\ge0$, it is sufficient to show \eqref{thm:u<} and \eqref{thm:u>} for $v_\varepsilon^<=I(\varepsilon\nabla)v_\varepsilon$ and $v_\varepsilon^>=v_\varepsilon-v_\varepsilon^<$, respectively, if $T$ is fixed. However, we go back to $u_\varepsilon^<$ and $u_\varepsilon^>$ to clarify whether the proportional constants depend on $T$. 
We will see that the estimate \eqref{thm:u<} is independent of $T$.

Since $v_{\varepsilon, \alpha}$ satisfies (\ref{v:mild}), we have 
\begin{align*}
u_{\varepsilon, \alpha}^{<}(t)&= e^{-\beta t}L_\varepsilon^+(t)\phi_\varepsilon^+
+ e^{-\beta t}L_\varepsilon^-(t)\phi_\varepsilon^-
+\int_0^t e^{-\beta t} \left(L_\varepsilon^+(t-t')f_\varepsilon^+(t')
+L_\varepsilon^-(t-t')f_\varepsilon^-(t')\right)dt',\\
u_{\varepsilon, \alpha}^{>}(t)&= e^{-\beta t} H_\varepsilon^+(t)\phi_\varepsilon^+
+e^{-\beta t} H_\varepsilon^-(t)\phi_\varepsilon^-
+\int_0^t e^{-\beta t} \left(H_\varepsilon^+(t-t')f_\varepsilon^+(t')
+H_\varepsilon^-(t-t')f_\varepsilon^-(t')\right)dt',
\end{align*}
%where 
%$$
%\phi_\varepsilon^\pm
%=\frac{\mp\varepsilon^2\lambda_\varepsilon^\mp(\nabla)\phi_0 \pm\varepsilon\phi_1}{2\sqrt{\alpha^2-\varepsilon^2|\nabla|^2}},\quad
%f_\varepsilon^\pm(t)=\pm\frac1{2\sqrt{\alpha^2-\varepsilon^2|\nabla|^2}}f(t).
%$$
Then Lemma \ref{HMmultiplier} implies
\begin{lem}
For any $r\in[1,\infty]$, $j \ge -1$ and $s\in[0,1]$,
\begin{align*}
\|\Delta_j\phi_\varepsilon^\pm\|_{L^r}&\lesssim_{\gamma} \|\Delta_j\phi_0\|_{L^r}+2^{-j}\|\Delta_j (\varepsilon \beta \phi_0 +\phi_1)\|_{L^r}
\lesssim_\alpha \|\Delta_j\phi_0\|_{L^r}+2^{-j}\|\Delta_j\phi_1\|_{L^r},\\
\|\Delta_jf_\varepsilon^\pm\|_{L^r}&\lesssim_{\gamma} \varepsilon^{-s}2^{-js} e^{(\re\beta) t} \|\Delta_jf\|_{L^r}.
\end{align*}
\end{lem}

Combining this with Theorems \ref{thm:L} and \ref{thm:H}, we can prove Theorem \ref{mainthm}.

\begin{proof}[{\bfseries Proof of Theorem \ref{mainthm}}]
First we consider the low frequency part. 
By Theorem \ref{thm:L}, and since $\mathrm{Re} \beta \ge 0$, for any $t\ge 0$
\begin{align*}
\|\Delta_ju_{\varepsilon, \alpha}^{<}(t)\|_{L^r}
&\lesssim\sum_{\circ=+,-}  e^{-(\re \beta) t} \|\Delta_j\phi_\varepsilon^\circ\|_{L^r} 
+\sum_{\circ=+,-}e^{-(\re\beta)t}2^{-js}\int_0^t (t-t')^{-s/2} 
\|\Delta_jf_\varepsilon^\circ(t')\|_{L^r}dt'  \\
%+\sum_{\circ=+,-}\int_0^t \|e^{-\beta t }  L^{\circ}_{\varepsilon}(t-t') \Delta_jf_\varepsilon^\circ(t')\|_{L^r}dt'\\
%
&\lesssim e^{-(\re\beta)t}\|\Delta_j\phi_0\|_{L^r}+e^{-(\re\beta)t}2^{-j}\|\Delta_j\phi_1\|_{L^r} 
%\|\Delta_j\phi_0\|_{L^r}+2^{-j}\|\Delta_j(\phi_1+\varepsilon \beta \phi_0)\|_{L^r} 
\\
&\hspace{+5mm}
+2^{-js}\int_0^t(t-t')^{-s/2} 
e^{-(\re\beta)(t-t')}\|\Delta_jf(t')\|_{L^r}
%e^{-(\mathrm{Re}\beta) (t-t')}
%\left\|\frac{\Delta_j f(t')}{2\sqrt{\gamma^2-\varepsilon^2|\nabla|^2}} \right\|_{L^r}
dt',%\\
%
%&\lesssim \|\Delta_j\phi_0\|_{L^r}+2^{-j}\|\Delta_j\phi_1\|_{L^r}
%+2^{-js}\int_0^t(t-t')^{-s/2} \|\Delta_j f(t')\|_{L^r}dt',
\end{align*}
thus by $\re\beta\ge C_{\alpha,1}$ and by Minkowski's inequality,
\begin{align*}
\|u_{\varepsilon, \alpha}^{<}(t)\|_{B_{r,2}^{\sigma+s}}
&\lesssim e^{-C_{\alpha,1}t} \|\phi_0\|_{B_{r,2}^{\sigma+s}} + e^{-C_{\alpha,1}t} \|\phi_1\|_{B_{r,2}^{\sigma+s-1}} \\
&\hspace{+5mm} + \int_0^t (t-t')^{-s/2} e^{-C_{\alpha,1}(t-t')}  \|f(t')\|_{B_{r,2}^\sigma}dt'.
\end{align*}
Since the functions $e^{-C_{\alpha,1}t}$ and $t^{-s/2}e^{-C_{\alpha,1}t}$ are integrable over $[0,\infty)$,
\eqref{thm:u<} is obtained by Young's inequality if $s<2(1+\frac1{q_1}-\frac1{q_2})$, and by Hardy-Littlewood-Sobolev's inequality if $s=2(1+\frac1{q_1}-\frac1{q_2})$ and $s<2$.

Next we consider the high frequency part. By Theorem \ref{thm:H}, and since $\mathrm{Re}\beta \ge 0,$ we have,  
\begin{align*}
\|\Delta_ju_{\varepsilon, \alpha}^{>}(t)\|_{L_T^{q_1}L_x^{r_1}}
&\lesssim_T \varepsilon^{\delta_1} 2^{js_1}\sum_{\circ=+,-}\|\Delta_j\phi_\varepsilon^\circ\|_{L^2} \\
& \hspace{5mm}+\sum_{\circ=+,-} \left\|\int_0^t e^{-\beta(t-t')} H^{\circ}_{\varepsilon}(t-t') 
\frac{\Delta_j f(t')}{2\sqrt{\gamma^2-\varepsilon^2|\nabla|^2}} dt'\right\|_{L^{q_1}_T L^{r_1}_x} \\
&\lesssim_T \varepsilon^{\delta_1} 2^{js_1} (\|\Delta_j\phi_0\|_{L^2}+2^{-j}\|\Delta_j(\phi_1+\varepsilon \beta \phi_0)\|_{L^2}) \\
&\hspace{5mm}+\varepsilon^{\delta_1+\delta_2} 2^{j(s_1+s_2)} \Big\|\frac{\Delta_jf}{2\sqrt{\gamma^2-\varepsilon^2|\nabla|^2}}\Big\|_{L_T^{q_2'}L_x^{r_2'}}\\
&\lesssim \varepsilon^{\delta_1} 2^{js_1}\|\Delta_j\phi_0\|_{L^2}
+\varepsilon^{\delta_1}2^{j(s_1-1)}\|\Delta_j\phi_1\|_{L^2}\\
&\quad+\varepsilon^{\delta_1+\delta_1-s} 2^{j(s_1+s_2-s)} \|\Delta_jf\|_{L_T^{q_2'}L_x^{r_2'}}.
\end{align*}
Hence by Minkowski's inequality we obtain \eqref{thm:u>}.
\end{proof}

\appendix   
   
%%%%%%%%%%%%%%%%%%%%%%%%%%%%%%%%%%%%%%%%%%%%%%%%%%%%%%
\section{Elementary results}\label{app:A}
%%%%%%%%%%%%%%%%%%%%%%%%%%%%%%%%%%%%%%%%%%%%%%%%%%%%%%

\begin{lem}\label{Faadi Bruno}
Let $n\in\mathbb{N}$ and $k\in \N^d$.
For a smooth function $h:\mathbb{R}^d\to\mathbb{C}$, one has
$$
|\partial^ke^{h(x)}|
\lesssim
e^{\re h(x)}
\sum_{m=(m_\ell)_{1\le|\ell|\le|k|}\subset\mathbb{N},\, \sum_\ell m_\ell \ell=k}
\prod_{\ell}|\partial^\ell h(x)|^{m_\ell}.
$$
\end{lem}

\begin{proof}
By the one-component Fa\`{a}di Bruno's formula,
\begin{align*}
\partial^k(F\circ h)(x)
=\sum_{m=(m_\ell)_{1\le|\ell|\le|k|}\subset\mathbb{N},\, \sum_\ell m_\ell \ell=k}
C_{\ell}(\partial^{\sum_\ell m_\ell}F)(h(x))
\prod_\ell(\partial^\ell h(x))^{m_\ell}
\end{align*}
for any holomorphic function $F$ and absolute constants $C_\ell$.
\end{proof}

We regard a smooth function $f\in\mathcal{D}(\mathbb{T}^d)$ as a periodic function on $\mathbb{R}^d$.
For a periodic smooth function $f$ on $\mathbb{R}^d$ and a Schwartz function $g$ on $\mathbb{R}^d$, we define the convolution by
$$
(g*f)(x):=\int_{\mathbb{R}^d} g(x-y)f(y)dy=\int_{\mathbb{R}^d}g(y)f(x-y)dy,\quad x\in\mathbb{R}^d.
$$

\begin{lem}\label{lem:poisson sum}
Let $\varphi$ be a Schwartz function on $\mathbb{R}^d$ and let $\eta$ be its Fourier inverse transform.
For any smooth function $f\in\mathcal{D}(\mathbb{T}^d)$,
$$
\varphi(\nabla)f=\eta*f.
$$
\end{lem}

\begin{lem}\label{lem:minkowski}
For any $p\in[1,\infty]$, Schwartz function $\eta$, and periodic smooth function $f\in\mathcal{D}(\mathbb{T}^d)$,
$$
\|\eta*f\|_{L^p(\mathbb{T}^d)}\le\|\eta\|_{L^1(\mathbb{R}^d)}\|f\|_{L^p(\mathbb{T}^d)}.
$$
\end{lem}

\begin{lem}\label{lem:young}
Let $a>0$ and $N_0>0$, and let $\eta:\mathbb{R}^d\to\mathbb{R}$ be a function such that $|\eta(x)|\lesssim_N1/(a+|x|)^N$ for any $N\ge N_0$. Then
$$
\|\eta*f\|_{L^\infty(\mathbb{T}^d)}
\lesssim(1+a^{-N_0})\|f\|_{L^1(\mathbb{T}^d)}.
$$
\end{lem}

\begin{proof}
Denote by $\tilde\eta=\sum_{k\in\mathbb{Z}^d}\eta(\cdot+2\pi k)$. Then it is sufficient to show that
$$
\|\tilde\eta\|_{L^\infty(\mathbb{T}^d)}\lesssim1+a^{-N_0}.
$$
If $|x|\le\pi$, then by taking some $N>d$,
\begin{align*}
|\tilde\eta(x)|
&\le|\eta(x)|+\sum_{k\neq0}|\eta(x+2\pi k)|\\
&\lesssim a^{-N_0}+\sum_{k\neq0}|k|^{-N}\lesssim1+a^{-N_0}.
\end{align*}
\end{proof}

\begin{lem}[{\cite[Lemma 2.2]{BCD}}]\label{HMmultiplier} 
Let $s\in\mathbb{R}$, $p\in[1,\infty]$, and let $N=2[1+\frac{d}2]$. Let $\sigma:\mathbb{R}^d\to\mathbb{R}$ be a function such that
$$
C_\sigma:=\sup_{k\in\mathbb{N}^d,\,|k|\le N}\sup_{\xi\neq0}\left(|\xi|^{|k|-s}|\partial^k\sigma(\xi)|\right)
<\infty.
$$
Then for any $u\in\mathcal{D}'(\mathbb{T}^d)$ and $j\ge0$,
\begin{align}\label{HM>0}
\|\sigma(\nabla)\Delta_ju\|_{L^p}\lesssim C_\sigma2^{js}\|\Delta_ju\|_{L^p}.
\end{align}
Moreover, if $\sigma$ is bounded,
\begin{align}\label{HM-1}
\|\sigma(\nabla)\Delta_{-1}u\|_{L^p}\lesssim \|\sigma\|_{L^\infty(\text{\rm supp}(\chi_{-1}))}\|\Delta_{-1}u\|_{L^p}.
\end{align}
\end{lem}

\begin{proof}
We fix a radial and smooth function $\tilde\chi$ supported on an annulus of $\mathbb{R}^d$, 
and such that $\tilde\chi(2^{-j}\cdot)\chi_j=\chi_j$ for any $j\ge0$. Set $\tilde\chi_j:=\tilde\chi(2^{-j}\cdot)$.
By Lemma \ref{lem:poisson sum},
$$
\sigma(\nabla)\Delta_ju=(\sigma\tilde\chi_j)(\nabla)\Delta_ju=K_j*\Delta_ju,
$$
where $K_j$ is the Fourier inverse transform (on $\mathbb{R}^d$) of $\sigma\tilde\chi_j$.
By the change of variable, $K_j=2^{jd}\tilde{K}_j(2^j\cdot)$ with $\tilde{K}_j$ the Fourier inverse transform (on $\mathbb{R}^d$) of $\sigma(2^j\cdot)\tilde{\chi}(\cdot)$. Then
\begin{align*}
\big|(1+|x|^2)^{N/2}\tilde{K}_j(x)\big|
&\le\int_{\mathbb{R}^d}\big|(1-\Delta_\xi)^{N/2}\left(\sigma(2^j\xi)\tilde{\chi}(\xi)\right)\big| d\xi\\
&\lesssim \sum_{|k|\le N}\int_{\text{\rm supp}(\tilde\chi)}2^{j|k|}\big|(\partial^k\sigma)(2^j\xi)\big|d\xi\\
&\lesssim C_\sigma2^{js}\sum_{|k|\le N}\int_{\text{\rm supp}(\tilde\chi)}|\xi|^{s-|k|}d\xi\lesssim C_\sigma2^{js}.
\end{align*}
Hence $\|K_j\|_{L^1}=\|\tilde{K}_j\|_{L^1}\lesssim C_\sigma2^{js}$. Then Lemma \ref{lem:minkowski} leads us to the estimate \eqref{HM>0}.

The estimate \eqref{HM-1} for $p=2$ is an immediate consequence of Plancherel's formula.
The general case follows from the equivalence of norms
$$
\|\Delta_{-1}f\|_{L^p}\asymp\|\Delta_{-1}f\|_{L^2}.
$$
Indeed, by taking a smooth and compactly supported function $\tilde{\chi}_{-1}$ such that $\tilde{\chi}_{-1}=1$ on $\text{supp}(\chi_{-1})$, for any $p\ge q$ we have
$$
\|\Delta_{-1}f\|_{L^p}=\|\tilde{\chi}_{-1}(\nabla)\Delta_{-1}f\|_{L^p}
\lesssim\|\Delta_{-1}f\|_{L^q}
$$
since $\mathcal{F}^{-1}\tilde{\chi}_{-1}$ belongs to $L^r$ for any $r$. The reverse inequality is obvious because $\mathbb{T}^d$ is a finite measure space.
\end{proof}

\begin{lem}[{\cite[Lemma 2.1]{BCD}}]\label{lem:Bernstein}
For any $(p,q)\in[1,\infty]^2$ with $1\le p\le q$, we have
$$
\|\Delta_ju\|_{L^q}\lesssim2^{jd(1/p-1/q)}\|\Delta_ju\|_{L^p}.
$$
\end{lem}

%%%%%%%%%%%%%%%%%%%%%%%%%%%%%%%%%%%%%%%%%%%%%%%%%%%%%%
\section{Strichartz estimate for the homogeneous case}\label{app:W}
%%%%%%%%%%%%%%%%%%%%%%%%%%%%%%%%%%%%%%%%%%%%%%%%%%%%%%

In this section, we prove Strichartz estimate for the homogeneous operator
$$
W_\varepsilon^\pm(t)=e^{-ct/\varepsilon^2}e^{\pm it|\nabla|/\varepsilon}.
$$
Since there is no significant difference between $W_\varepsilon^+$ and $W_\varepsilon^-$, we consider only the former one. We write $W_\varepsilon=W_\varepsilon^+$.

\begin{thm}\label{thm:W}
The same estimates as Theorem \ref{thm:H} holds with $H_\varepsilon^\pm$ replaced by $W_\varepsilon^{\pm}$ holds true.
\end{thm}

\subsection{Oscillatory integral}%%%%%%%%%%%%%%%%%%%%%%%%%%%%%%%%%

For any radial smooth function $\varphi\in C_0^\infty(\mathbb{R}^d)$ supported on an annulus, consider the integral
$$
K_\varphi(t,x)
=\mathcal{F}^{-1}(e^{it|\cdot|}\varphi)
=\frac1{(2\pi)^d}\int_{\mathbb{R}^d}e^{it|\xi|+ix\cdot\xi}\varphi(\xi)d\xi
$$
for $(t,x)\in(0,\infty)\times\mathbb{R}^d$.

\begin{lem}[{\cite[Theorem 8.8]{BCD}}]\label{lem:osc int t^-N}
Consider a smooth function $\Phi$ on $\mathbb{R}^d$ such that, for some compact set $K$ we have
$$
c:=\inf_{\xi\in K}|\nabla\Phi(\xi)|>0.
$$
Then, for any $N>0$ and any smooth function $\psi$ supported in $K$, there exists a constant $C=C(N,c,\Phi)$ such that
$$
\left|\int_{\mathbb{R}^d}e^{i\tau\Phi(\xi)}\psi(\xi)d\xi\right|\le C\|\psi\|_{C^N(K)}\tau^{-N}
$$
for any $\tau>0$.
\end{lem}

Let $\sigma_{S^{d-1}}$ be the surface measure of the unit sphere in $\mathbb{R}^d$ and let
$$
\widehat{\sigma_{S^{d-1}}}(x)=\int_{S^{d-1}}e^{ix\cdot\eta}\sigma_{S^{d-1}}(d\eta).
$$

\begin{lem}[{\cite[Corollary 2.37]{NS}}] Let $d \ge 2.$ 
One has the representation
$$
\widehat{\sigma_{S^{d-1}}}(x)=e^{i|x|}\omega_+(|x|)+e^{-i|x|}\omega_-(|x|),
$$
where $\omega_\pm$ are smooth and satisfy
$$
|\omega_\pm(r)|\lesssim r^{-(d-1)/2}.
$$
\end{lem}

We show the dispersive estimate of the function $K_\varphi(t,x)$.

\begin{lem}\label{lem:dispersive estimate of K^pm} Let $d\ge 2$. 
There exists a constant $C>0$ such that
\begin{align*}
|K_\varphi(t,x)|\le Ct^{-(d-1)/2}
\end{align*}
if $t\simeq|x|$.
For any $N\ge\frac{d-1}2$, there exists a constant $C_N$ such that
\begin{align*}
|K_\varphi(t,x)|\le C_N(t+|x|)^{-N}.
\end{align*}
if $t\not\simeq|x|$.
\end{lem}

\begin{proof}
Introducing polar coordinates $\xi=r\eta$ ($r\in(0,\infty)$, $\eta\in S^{d-1}$), we have
\begin{align*}
K_\varphi(t,x)
&=\int_{(0,\infty)\times S^{d-1}}r^{d-1}dr\sigma_{S^{d-1}}(d\eta) e^{irx\cdot\eta+irt}\varphi(r\eta) \\
&=\int_0^\infty \widehat{\sigma_{S^{d-1}}}(rx)e^{irt}\bar\varphi(r)dr\\
&=\int_0^\infty e^{ir(t+|x|)}\omega_+(r|x|)\bar\varphi(r)dr
+\int_0^\infty e^{ir(t-|x|)}\omega_-(r|x|)\bar\varphi(r)dr,
\end{align*}
where $\bar\varphi$ is a smooth function on $\mathbb{R}$ such that $\bar\varphi(r)=\varphi(r\eta)r^{d-1}$.

By Lemma \ref{lem:osc int t^-N},
$$
\left|\int_0^\infty e^{ir(t+|x|)}\omega_+(r|x|)\bar\varphi(r)dr\right|\lesssim_N(t+|x|)^{-N}
$$
for any $N>0$. For the second integral, if $t\not\simeq|x|$, by Lemma \ref{lem:osc int t^-N},
$$
\left|\int_0^\infty e^{ir(t-|x|)}\omega_-(r|x|)\bar\varphi(r)dr\right|\lesssim_N(t+|x|)^{-N}
$$
for any $N>0$. If $t\simeq|x|$,
\begin{align*}
\left|\int_0^\infty e^{ir(t-|x|)}\omega_-(r|x|)\bar\varphi(r)dr\right|
\lesssim\int_0^\infty(r|x|)^{-(d-1)/2}\bar\varphi(r)dr
\lesssim|x|^{-(d-1)/2}\lesssim t^{-(d-1)/2}.
\end{align*}
\end{proof}

\subsection{Proof of Theorem \ref{thm:W}}

By using above theorem, we prove the Strichartz estimate for $W_\varepsilon$.

\begin{lem}\label{lem:dispersive estimate of e^{it|nabla|}}
Let $d \geq 1$. 
For any radial smooth function $\varphi\in C_0^\infty(\mathbb{R}^d)$ supported in an annulus, there exists a constant $c'>0$ such that for any $j\ge0$,
$$
\|W_\varepsilon(t)\varphi(2^{-j}\nabla)f\|_{L^\infty}
\lesssim e^{-c't/\varepsilon^2}2^{j(d+1)/2}(t/\varepsilon)^{-(d-1)/2}\|f\|_{L^1}.
$$
\end{lem}

\begin{proof}
First, we consider the case $d=1$. By Lemma \ref{lem:Bernstein}, we obtain
\begin{align*}
	\|W_\varepsilon(t)\varphi(2^{-j}\nabla)f\|_{L^\infty}
	&\lesssim 2^{j/2}\|W_\varepsilon(t)\varphi(2^{-j}\nabla)f\|_{L^2}
	\\&
	\lesssim  2^{j/2}e^{-ct/\varepsilon^2}\|\varphi(2^{-j}\nabla)f\|_{L^2}
	\lesssim 2^{j}e^{-ct/\varepsilon^2}\|\varphi(2^{-j}\nabla)f\|_{L^1},
\end{align*}
which is the desired estimate.
Next, we consider the higher dimensional cases.
By Lemma \ref{lem:poisson sum},
$$
W_\varepsilon(t)\varphi(2^{-j}\nabla)f=e^{-ct/\varepsilon^2+t/\varepsilon}K_{\varphi,\varepsilon}^{t,j}*f,
$$
where $K_{\varphi,\varepsilon}^{t,j}=e^{-t/\varepsilon}\mathcal{F}^{-1}\left(e^{it|\cdot|/\varepsilon}\varphi(2^{-j}\cdot)\right)$. By the change of variables and by Lemma \ref{lem:dispersive estimate of K^pm},
we have that
\begin{align*}
|K_{\varphi,\varepsilon}^{t,j}(x)|
=e^{-t/\varepsilon}2^{jd}|K_{\varphi}(2^jt/\varepsilon,2^jx)|
\lesssim2^{j(d-N)}\left(t/\varepsilon+|x|\right)^{-N}.
\end{align*}
for any $N\ge\frac{d-1}2$, if $t/\varepsilon\not\simeq|x|$.
On the other hand, if $t/\varepsilon\simeq|x|$,
\begin{align*}
|K_{\varphi,\varepsilon}^{t,j}(x)|
=e^{-t/\varepsilon}2^{jd}|K_{\varphi}(2^jt/\varepsilon,2^jx)|
&\lesssim2^{j(d+1)/2}e^{-t/\varepsilon}(t/\varepsilon)^{-(d-1)/2}\\
&\lesssim2^{j(d+1)/2}(t/\varepsilon)^{-N}
\end{align*}
for any $N\ge\frac{d-1}2$.
By Lemma \ref{lem:young},
\begin{align*}
e^{ct/2\varepsilon^2}\|W_\varepsilon(t)\varphi(2^{-j}\nabla)f\|_{L^\infty}
&\lesssim e^{-ct/2\varepsilon^2+t/\varepsilon}2^{j(d+1)/2}(1+(t/\varepsilon)^{-(d-1)/2})\|f\|_{L^1}\\
&\lesssim 2^{j(d+1)/2}e^{-t/\varepsilon}\|f\|_{L^1}+2^{j(d+1)/2}(t/\varepsilon)^{-(d-1)/2}\|f\|_{L^1}\\
&\lesssim 2^{j(d+1)/2}(t/\varepsilon)^{-(d-1)/2}\|f\|_{L^1}.
\end{align*}
\end{proof}

%\subsection{$TT^*$ argument}

%We say that $(q,r)\in[2,\infty]\times[2,\infty)$ is \emph{$\sigma$-admissible} if
%$$
%\frac1q+\frac{\sigma}r=\frac{\sigma}2.
%We say that $(q,r)\in[2,\infty]^2$ is \emph{$\sigma$-admissible} if
%$$
%\frac1q+\frac{\sigma}r=\frac{\sigma}2.
%$$

\begin{lem}\label{B lem TT* bilinear}
Let $\frac{d-1}2<m\le\infty$ and
let $(q_k,r_k)\in[2,\infty]^2$ ($k=1,2$) be $\sigma$-admissible pairs, see \eqref{eq:admissible}.
Define
$$
s_k=s(r_k)=\frac{d+1}2\left(\frac12-\frac1{r_k}\right),\quad
\delta_k=\delta(q_k,r_k)=\frac2{q_k}-\frac{d-1}2\left(\frac12-\frac1{r_k}\right).
$$
Let $\varphi$ be a smooth function as in Lemma \ref{lem:dispersive estimate of e^{it|nabla|}}.
For any $f_k\in L_T^{q_k}L_x^{r_k}$, we write
$$
T_j(f_1,f_2):=\int_{[0,T]^2}\mathbf{1}_{\{t'<t\}}\langle W_\varepsilon(t-t')\varphi(2^{-j}\nabla)f_1(t'),f_2(t)\rangle_{L_x^2}dtdt'.
$$
Then we have
\begin{align}\label{key estimate of Strichartz}
|T_j(f_1,f_2)|\lesssim \varepsilon^{\delta_1+\delta_2} 2^{j(s_1+s_2)}\|f_1\|_{L_T^{q_1'}L_x^{r_1'}}\|f_2\|_{L_T^{q_2'}L_x^{r_2'}}.
\end{align}
\end{lem}

\begin{proof}
If $(q_1,r_1)=(q_2,r_2)=(q,r)$, then \eqref{key estimate of Strichartz} follows from Lemma \ref{lem:dispersive estimate of e^{it|nabla|}} and Young inequality. Indeed, we have
$$
\|W_\varepsilon(t)\varphi(2^{-j}\nabla)f\|_{L^r}\lesssim
e^{-c't/\varepsilon^2}(2^{j(d+1)/2}(t/\varepsilon)^{-(d-1)/2})^{1-2/r}
\|f\|_{L^{r'}}
$$
for any $0\le t$ and $r\in[2,\infty]$ by an interpolation between the trivial $L^2$-$L^2$ estimate of $W_\varepsilon(t)$ and Lemma \ref{lem:dispersive estimate of e^{it|nabla|}}. Hence
\begin{align*}
&|T_j(f_1,f_2)|\\
&\lesssim
\varepsilon^{\frac{d-1}2(1-\frac2r)}2^{j\frac{d+1}2(1-\frac2r)}
\int_{[0,T]^2}\mathbf{1}_{\{t'<t\}}
e^{-c'\frac{t-t'}{\varepsilon^2}}
(t-t')^{-\frac{d-1}2(1-\frac2r)}
\|f_1(t')\|_{L^{r'}}\|f_2(t)\|_{L^{r'}}dtdt'\\
&\lesssim 
\varepsilon^{\frac{d-1}2(1-\frac2r)}2^{2js(r)}
\|e^{-\frac{c't}{\varepsilon^2}}t^{-\frac{d-1}2(1-\frac2r)}\|_{L^{q/2}}
\|f_1\|_{L_T^{q'}L_x^{r'}}\|f_2\|_{L_T^{q'}L_x^{r'}}.
%
%&\lesssim 
%\varepsilon^{\frac{d-1}2(1-\frac2r)}2^{2js(r)}
%\|e^{-\frac{ct}{\varepsilon^2}}\|_{L^{p_1}}\|t^{-\frac{d-1}2(1-\frac2r)}\|_{L^{p_2}}
%\|f_1\|_{L_T^{q'}L_x^{r'}}\|f_2\|_{L_T^{q'}L_x^{r'}}\\
%&\lesssim
%\varepsilon^{\frac{d-1}2(1-\frac2r)+\frac2{p_1}}2^{2js(r)}
%\|f_1\|_{L_T^{q'}L_x^{r'}}\|f_2\|_{L_T^{q'}L_x^{r'}}.
\end{align*}
Since
\begin{align*}
\|e^{-\frac{c't}{\varepsilon^2}}t^{-\frac{d-1}2(1-\frac2r)}\|_{L^{q/2}}
=\varepsilon^{\frac4q-(d-1)(1-\frac2r)}
\|e^{-c't}t^{-\frac{d-1}2(1-\frac2r)}\|_{L^{q/2}}
\end{align*}
and $\frac{d-1}2(1-\frac2r)\frac{q}2<m(1-\frac2r)\frac{q}2=1$, we have
$$
|T_j(f_1,f_2)|
\lesssim
\varepsilon^{2\delta(q,r)}2^{2js(r)}
\|f_1\|_{L_T^{q'}L_x^{r'}}\|f_2\|_{L_T^{q'}L_x^{r'}}.
$$

Next we consider the case $(q_1,r_1)\neq(q_2,r_2)$.
Note that
\begin{align*}
W_\varepsilon^*(t-s)W_\varepsilon(t-s')
&=e^{-c(t-s)/\varepsilon^2}e^{-c(t-s')/\varepsilon^2}e^{-i(t-s)|\nabla|/\varepsilon}e^{i(t-s')|\nabla|/\varepsilon}\\
&=e^{-2c(t-s)/\varepsilon^2}e^{-c(s-s')/\varepsilon^2}e^{i(s-s')|\nabla|/\varepsilon}\\
&=e^{-2c(t-s)/\varepsilon^2}W_\varepsilon(s-s').
\end{align*}
Then the above estimate yields that, for any $m$-admissible $(q,r)$,
\begin{align}\label{dual of Strichartz 1}
\begin{aligned}
&\left\|\int_0^tW_\varepsilon(t-s)\varphi(2^{-j}\nabla)f(s)ds\right\|_{L^2}^2\\
&=\int_{[0,t]^2}\langle W_\varepsilon(t-s)\varphi(2^{-j}\nabla)f(s), W_\varepsilon(t-s')\varphi(2^{-j}\nabla)f(s') \rangle dsds'\\
&=\int_{[0,t]^2, s<s'}
e^{-2c(t-s')/\varepsilon^2}\langle W_\varepsilon(s'-s)\varphi^2(2^{-j}\nabla)f(s), f(s')\rangle dsds'\\
&\quad+\int_{[0,t]^2, s'\le s}
e^{-2c(t-s)/\varepsilon^2}\langle f(s), W_\varepsilon(s-s')\varphi^2(2^{-j}\nabla)f(s')\rangle dtdt'\\
&\lesssim \varepsilon^{2\delta(q,r)}2^{2js(r)}\|f\|_{L_T^{q'}L_x^{r'}}^2.
\end{aligned}
\end{align}
Let $q_1\le q_2$. Then we have
\begin{align*}
|T_j(f_1,f_2)|
&=\left|\int_0^T
\left\langle
\int_0^t
W_\varepsilon(t-t')\varphi(2^{-j}\nabla)f_1(t')dt',f_2(t)
\right\rangle_{L_x^2}dt
\right|\\
&\lesssim\sup_{t\in[0,T]}
\left\|\int_0^t W_\varepsilon(t-t')\varphi(2^{-j}\nabla)f_1(t')dt'\right\|_{L_x^2}\|f_2\|_{L_T^1L_x^2}\\
&\lesssim 
\varepsilon^{\delta_1}2^{js_1}
\|f_1\|_{L_T^{q_1'}L_x^{r_1'}}\|f_2\|_{L_T^1L_x^2}.
\end{align*}
This means that \eqref{key estimate of Strichartz} holds for $(q_2,r_2)=(\infty,2)$.
Since we already have \eqref{key estimate of Strichartz} with $(q_2,r_2)=(q_1,r_1)$,
by an interpolation we obtain \eqref{key estimate of Strichartz} since
\begin{align*}
\frac1{r_2}=\frac{1-\theta}2+\frac{\theta}{r_1},\quad \theta\in[0,1]
\quad\Rightarrow\quad
\frac1{q_2}=\frac{1-\theta}\infty+\frac{\theta}{q_1}.
\end{align*}
On the other hand, if $q_1\ge q_2$, then by a similar argument as \eqref{dual of Strichartz 1} we have
\begin{align*}
|T_j(f_1,f_2)|
&=\left|\int_0^T
\left\langle
f_1(t'),
\int_{t'}^T
W_\varepsilon^*(t-t')\varphi(2^{-j}\nabla)f_2(t)dt
\right\rangle_{L_x^2}dt'
\right|\\
&\lesssim\|f_1\|_{L_T^1L_x^2}
\sup_{t'\in[0,T]}
\left\|\int_{t'}^T W_\varepsilon^*(t-t')\varphi(2^{-j}\nabla)f_2(t)dt\right\|_{L_x^2}\\
&\lesssim 
\varepsilon^{\delta_2}2^{js_2}\|f_1\|_{L_T^1L_x^2}\|f_2\|_{L_T^{q_2'}L_x^{r_2'}},
\end{align*}
Hence we have \eqref{key estimate of Strichartz} with $(q_1,r_1)=(\infty,2)$.
Thus by a similar interpolation argument, we obtain \eqref{key estimate of Strichartz} with $q_1\ge q_2$.
\end{proof}

\begin{lem}\label{B lem TT*int}
Let $(q,r)\in[2,\infty]^2$ be an $m$-admissible pair.
Let $\varphi$ be a smooth function as in Lemma \ref{lem:dispersive estimate of e^{it|nabla|}}.
Then we have
$$
\left\|\int_0^tW_\varepsilon(s)\varphi(2^{-j}\nabla)f(s)ds\right\|_{L_T^\infty L_x^2}
\lesssim\varepsilon^{\delta(q,r)}2^{js(r)}\|f\|_{L_T^{q'}L_x^{r'}}.
$$
\end{lem}

\begin{proof}
Note that
\begin{align*}
W_\varepsilon^*(s)W_\varepsilon(s')
&=e^{-cs/\varepsilon^2}e^{-cs'/\varepsilon^2}e^{-is|\nabla|/\varepsilon}e^{is'|\nabla|/\varepsilon}\\
&=e^{-2cs'/\varepsilon^2}e^{-c(s-s')/\varepsilon^2}e^{-i(s-s')|\nabla|/\varepsilon}\\
&=e^{-2cs'/\varepsilon^2}W_\varepsilon^*(s-s').
\end{align*}
Hence by a similar argument to \eqref{dual of Strichartz 1},
\begin{align*}
\begin{aligned}
&\left\|\int_0^tW_\varepsilon(s)\varphi(2^{-j}\nabla)f(s)ds\right\|_{L^2}^2\\
&=\int_{[0,t]^2}\langle W_\varepsilon(s)\varphi(2^{-j}\nabla)f(s), W_\varepsilon(s')\varphi(2^{-j}\nabla)f(s') \rangle dsds'\\
&=\int_{[0,t]^2, s<s'}
e^{-2cs/\varepsilon^2}\langle W_\varepsilon^*(s'-s)\varphi^2(2^{-j}\nabla)f(s), f(s')\rangle dsds'\\
&\quad+\int_{[0,t]^2, s'\le s}
e^{-2cs'/\varepsilon^2}\langle f(s), W_\varepsilon(s-s')\varphi^2(2^{-j}\nabla)f(s')\rangle dtdt'\\
&\lesssim\varepsilon^{2\delta(q,r)}2^{2js(r)}\|f\|_{L_T^{q'}L_x^{r'}}^2.
\end{aligned}
\end{align*}
\end{proof}

\begin{proof}[{\bfseries Proof of Theorem \ref{thm:W}}]
We prove \eqref{Strichartz 1} and \eqref{Strichartz 2} for $W_\varepsilon$.
Taking the same function $\tilde\chi$ as in the proof of Lemma \ref{HMmultiplier}, we have
%We fix a radial and smooth function $\tilde\chi$ supported on an annulus of $\mathbb{R}^d$, 
%and such that $\tilde\chi(2^{-j}\cdot)\chi_j=\chi_j$ for any $j\ge0$. Set $\tilde\chi_j:=\tilde\chi(2^{-j}\cdot)$.
%Then by definition,}
%Remark that, by definition of $\tilde\chi$,
\begin{align*}
W_\varepsilon(t)\Delta_ju&=W_\varepsilon(t)\tilde\chi(2^{-j}\nabla)\Delta_ju.
\end{align*}
As for \eqref{Strichartz 1}, by the duality argument and Lemma \ref{B lem TT*int},
\begin{align*}
\|W_\varepsilon(t) \Delta_ju\|_{L_T^{q_1}L_x^{r_1}}
&=\sup_{\varphi;\ \|\varphi\|_{L_T^{q_1'}L_x^{r_1'}}\le1}
\left|\int_0^T\langle W_\varepsilon(t)\tilde{\chi}(2^{-j}\nabla)\Delta_ju,\varphi(t) \rangle_{L_x^2} dt\right|\\
&=\sup_{\varphi;\ \|\varphi\|_{L_T^{q_1'}L_x^{r_1'}}\le1}
\left|\int_0^T\langle \Delta_ju,W_\varepsilon^*(t)\tilde{\chi}(2^{-j}\nabla)\varphi(t) \rangle_{L_x^2} dt\right|\\
&\le \|\Delta_ju\|_{L^2}
\sup_{\varphi;\ \|\varphi\|_{L_T^{q_1'}L_x^{r_1'}}\le1}
\left\|\int_0^TW_\varepsilon^*(t)\tilde{\chi}(2^{-j}\nabla)\varphi(t)dt\right\|_{L_x^2}\\
&\lesssim \varepsilon^{\delta_1}2^{js_1}\|\Delta_ju\|_{L^2}.
\end{align*}
As for \eqref{Strichartz 2}, by the duality argument and Lemma \ref{B lem TT* bilinear},
\begin{align*}
&\left\|\int_0^t W_\varepsilon(t-t')\Delta_jf(t') dt'\right\|_{L_T^{q_1}L_x^{r_1}}\\
&=\sup_{\varphi;\ \|\varphi\|_{L_T^{q_1'}L_x^{r_1'}}\le1}
\left|\int_{[0,T]^2}\mathbf{1}_{\{t'<t\}}\langle W_\varepsilon(t-t')\tilde{\chi}(2^{-j}\nabla) \Delta_jf(t'), \varphi(t)\rangle_{L_x^2}dtdt'\right|\\
&\lesssim \varepsilon^{\delta_1+\delta_2}2^{j(s_1+s_2)}\|f\|_{L_T^{q_2'}L_x^{r_2'}}.
\end{align*}
\end{proof}

%%%%%%%%%%%%%%%%%%%%%%%%%%%%%%%%%%%%%%%%%%%%%%%%%%%%%%
\section{Review on the complex Wick products}\label{app:Wick}
%%%%%%%%%%%%%%%%%%%%%%%%%%%%%%%%%%%%%%%%%%%%%%%%%%%%%%

We recall from \cite{Ito} a few results about the complex Hermite polynomials, 
and our definition of complex Wick products will be based on these results.
Remark that other constructions of the complex Wick products can be found in \cite{OT, M,Tr}. 
For $m,n\in\mathbb{N}$, $z\in\mathbb{C}$, and $\sigma>0$, the polynomial $H_{m,n}(z;\sigma)$ are determined by the generating function
\begin{align}\label{eq:gf of CHermite}
\exp(t\bar{z}+\bar{t}z-\sigma t\bar{t}\,)=\sum_{m,n=0}^\infty \frac{\bar{t}^mt^n}{m!n!}H_{m,n}(z;\sigma),\qquad t\in\mathbb{C}.
\end{align}
It follows from the definition that $\overline{H_{m,n}(z;\sigma)}=H_{n,m}(z;\sigma)$. Here are some examples:
\begin{align*}
&H_{0,0}(z;\sigma)=1,\qquad H_{1,0}(z;\sigma)=z,\\
&H_{2,0}(z;\sigma)=z^2,\qquad H_{1,1}(z;\sigma)=z\bar{z}-\sigma,\\
&H_{3,0}(z;\sigma)=z^3,\qquad H_{2,1}(z;\sigma)=z^2\bar{z}-2\sigma z.
\end{align*}
We identify an element $z=x+iy\in\mathbb{C}$ with $(x,y)\in\mathbb{R}^2$ and consider the differential operators
$$
\partial_z:=\frac12(\partial_x-i\partial_y),\qquad
\partial_{\bar z}:=\frac12(\partial_x+i\partial_y).
$$ 

\begin{prop}\label{app eq:Hermite}
For any $m,n\in\mathbb{N}$, $z\in\mathbb{C}$, and $\sigma>0$, we have
\begin{align}
\label{eq:Hermite1}
H_{m,n}(z+w;\sigma)
&=\sum_{m'=0}^m\sum_{n'=0}^n
\binom{m}{m'}\binom{n}{n'}w^{m-m'}\bar{w}^{n-n'}H_{m',n'}(z;\sigma),\qquad
(w\in\mathbb{C})\\
\label{eq:Hermite2}
\partial_{\bar z}H_{m,n}(z;\sigma)&=H_{m,n-1}(z;\sigma).
\end{align}
\end{prop}

\begin{proof}
\eqref{eq:Hermite1} follows from the comparison between
$$
e^{t(\bar{z}+\bar{w})+\bar{t}(z+w)-\sigma t\bar{t}}
=\sum_{m,n=0}^\infty \frac{\bar{t}^mt^n}{m!n!}H_{m,n}(z+w;\sigma)
$$
and
$$
e^{t\bar{w}}e^{\bar{t}w}e^{t\bar{z}+\bar{t}z-\sigma t\bar{t}}
=\sum_{k,\ell,m,n=0}^\infty \frac{\bar{t}^{k+m}t^{\ell+n}}{k!m!\ell!n!}w^k\bar{w}^\ell H_{m,n}(z+w;\sigma).
$$
\eqref{eq:Hermite2} follows from the $\bar{z}$-differentials of both sides of \eqref{eq:gf of CHermite}.
\end{proof}

\begin{prop}
Let $X$ and $Y$ be complex Gaussian random variables with zero mean and such that
$$
\mathbb{E}[X^2]=\mathbb{E}[Y^2]=\mathbb{E}[XY]=0,\qquad
\mathbb{E}[X\overline{X}]=\sigma_X,\qquad \mathbb{E}[Y\overline{Y}]=\sigma_Y.
$$
Then we have
$$
\mathbb{E}[H_{m,n}(X;\sigma_X)H_{k,\ell}(Y;\sigma_Y)]
=\mathbf{1}_{m=\ell,\, n=k}m!n!(\mathbb{E}[\overline{X}Y])^m(\mathbb{E}[X\overline{Y}])^n.
$$
\end{prop}

\begin{proof}
For any $t,s\in\mathbb{C}$, we have
\begin{align*}
\mathbb{E}[e^{t\overline{X}+\bar{t}X-\sigma_Xt\bar{t}}
e^{s\overline{Y}+\bar{s}Y-\sigma_Ys\bar{s}}]
&=e^{\frac12\mathbb{E}[(t\overline{X}+\bar{t}X+s\overline{Y}+\bar{s}Y)^2]-\sigma_Xt\bar{t}-\sigma_Ys\bar{s}}\\
&=e^{t\bar{s}\overline{X}Y+\bar{t}sX\overline{Y}}.
\end{align*}
The required identity follows from the comparison of $\bar{t}^mt^n\bar{s}^ks^\ell$-coefficients.
\end{proof}

We define the Gaussian measure on $\mathcal{S}'(\mathbb{T}^2)$ by
$$
\mu_0(dz)=\frac1{\Gamma'}\exp\left(-\frac12\int_{\mathbb{T}^2}(|z(x)|^2+|\nabla z(x)|^2)dx\right)dz,
$$
with a normalizing constant $\Gamma'$. Precisely, by identifying an element $z\in\mathcal{S}'(\mathbb{T}^2)$ 
with a sequence $(\hat{z}(k))_{k\in\mathbb{Z}^2}$, the measure $\mu_0$ is defined by the product
$$
\mu_0=\bigotimes_{k\in\mathbb{Z}^2}\mathcal{N}_c(0,2(1+|k|^2)^{-1}),
$$
where $\mathcal{N}_c(0,r)$ denotes the probability law of $X+iY$, 
where $X$ and $Y$ are real-valued independent Gaussian random variables with mean zero and covariance $r/2$. The density function of $\mathcal{N}_c(0,r)$ is given by
$$
\frac1{\pi r}e^{-\frac1r|z|^2},\qquad z\in\mathbb{C}.
$$
It is straightforward to see that $\mu_0$ is supported in $H^{-\kappa}$ for any $\kappa>0$.
Denoting by $\Pi_N:z\mapsto \sum_{|k|\le N}\hat{z}(k)e_k$ the standard projection, we define the Wick product by the limit (if exists)
$$
:z^m\bar{z}^n:\ :=\lim_{N\to\infty}H_{m,n}(\Pi_Nz;C_N),
$$
where
$$
C_N:=\int_{H^{-\kappa}}|\Pi_Nz(x)|^2\mu_0(dz)=\sum_{|k|\le N}\frac2{1+|k|^2}.
$$
We have the following convergence result by a similar way to \cite[Proposition 2.1]{GKO}.

\begin{prop}
For any $m,n\in\mathbb{N}$, the sequence $\{H_{m,n}(\Pi_Nz;C_N)\}_{N\in\mathbb{N}}$ is a Cauchy sequence in $L^p(\mu_0;W^{-\delta,\infty}(\mathbb{T}^2))$ 
for any $p\in[1,\infty)$ and $\delta>0$.
\end{prop}

The following complex valued It\^o formula is easily obtained from the real-valued version.
For a complex-valued continuous semimartingale $Z=X+iY$, we define
\begin{align*}
dZdZ&=dXdX-dYdY+2idXdY,\\
dZd\overline{Z}&=dXdX+dYdY,\\
d\overline{Z}d\overline{Z}&=dXdX-dYdY-2idXdY.
\end{align*}

\begin{prop}\label{app:complexIto}
Let $Z$ be a complex-valued continuous semimartingale. For any $C_b^2$ function $f$ on $\mathbb{R}^2$ identified with $\mathbb{C}$, we have
\begin{align*}
f(Z_t)
&=f(Z_0)+\int_0^t\partial_zf(Z_s)dZ_s+\int_0^t\partial_{\bar z}f(Z_s)d\overline{Z}_s\\
&\quad+\frac12\int_0^t\partial_{z}^2f(Z_s)dZ_sdZ_s+\int_0^t\partial_{z\bar{z}}f(Z_s)dZ_sd\overline{Z}_s
+\frac12\int_0^t\partial_{\bar{z}}^2f(Z_s)d\overline{Z}_sd\overline{Z}_s.
\end{align*}
\end{prop}

\noindent
{\bf Acknowledgement} The authors are grateful for the discussion 
with Michikazu Kobayashi and Toshiyuki Sugawa. 
%The authors also thank the referee for the useful comments to improve the manuscript.
This work was supported by JSPS KAKENHI 19KK0066, 19K14556, 18K13444. 
This work was partly supported by Osaka City University Advanced Mathematical Institute
(MEXT Joint Usage/Research Center on Mathematics and Theoretical Physics JPMXP0619217849).
%%%%%%%%%%%%%%%%%%%%%%%% References %%%%%%%%%%%%%%%%%%%%%%%%%%%%%%%%%%%%%%%%%%

\articleend

\title[Erratum]{Corrigendum:``Non relativistic and ultra relativistic limits in 2d stochastic nonlinear damped Klein-Gordon equation" 
(2022 Nonlinearity \bf{35} 2878)}
\author[Reika Fukuizumi]{Reika Fukuizumi$^{\scriptsize 1}$}
\author[Masato Hoshino]{Masato Hoshino$^{\scriptsize 2}$}
\author[Takahisa Inui]{Takahisa Inui$^{\scriptsize 3}$}

\maketitle

\begin{center} 
$^1$ Research Center for Pure and Applied Mathematics, \\
Graduate School of Information Sciences, Tohoku University,\\
Sendai 980-8579, Japan; \\
\email{fukuizumi@math.is.tohoku.ac.jp}
\end{center}

\begin{center} 
$^2$ Graduate School of Engineering Science, Osaka University,\\
Toyonaka, Osaka 560-8531,Japan;\\
\email{hoshino@sigmath.es.osaka-u.ac.jp}
\end{center}

\begin{center}
$^3$ Department of Mathematics, Graduate School of Science, \\
Osaka University, Toyonaka, Osaka 560-0043, Japan; \\
\email{inui@math.sci.osaka-u.ac.jp}
\end{center}

\bigskip

%%%%%%%%%%%%%%%%%%%%%%% Introduction %%%%%%%%%%%%%%%%%%%%%%%%%%%%%%%%%%%

This is a corrigendum for the paper ``Non relativistic and ultra relativistic limits in 2d stochastic nonlinear damped Klein-Gordon equation" \cite{FHI}. 
We proved the global existence of the solution $\Psi$ in Section 5 of \cite{FHI}, however, the proof of Theorem 4 of \cite{FHI} contains an error. We used the statement that
$$
T^*<\infty\quad \Rightarrow\quad \lim_{t\to T^*}\sup_{0\le s< t}\|(\Psi(s),\Phi(s))\|_{\mathcal{H}^{-\delta}}=\infty
$$
in the proof, but this is false because $(\Psi(t),\Phi(t))$ is not locally well-posed for \emph{all} initial values $(\psi,\phi)\in\mathcal{H}^{-\delta}$.
We give a correct proof in this erratum. All notations come from Section 5 of \cite{FHI}.

First we replace Lemma 5.2 of \cite{FHI} with the following one.

\begin{lem}
Let any $T>0$ be fixed, and let $\sigma<1$ be sufficiently close to $1$ and $\delta=1-\sigma$ as in Corollary 2.1 of \cite{FHI}. There exists a constant $C_T$ independent of $N$ such that 
\begin{align*} 
\sup_{N\in\mathbb{N}}\int_{\mathcal{H}^{-\delta}}
\mathbb{E}\left[\sup_{0\le t\le T}\|(U^N(t),\varepsilon\partial_tU^N(t))\|_{\mathcal{H}^\sigma}\right]\rho_N(\dd\psi \dd\phi) \le C_T.
\end{align*}
\end{lem}

\begin{proof}
By definition, $U^N$ satisfies the equation
$$
\left\{
\begin{aligned}
&\varepsilon^2\partial_t^2U^N+2\alpha\partial_tU^N+(1-\Delta)U^N+\Pi_NH_{n+1,n}(\Pi_N\Psi^N;C_N)=0,\\
&(U^N,\varepsilon\partial_tU^N)|_{t=0}=(0,0).
\end{aligned}
\right.
$$
Hence, by Theorem 1 of \cite{FHI}, it is sufficient to show the bound
$$
\sup_{N\in\mathbb{N}}\int_{\mathcal{H}^{-\delta}}
\mathbb{E}\left[\|\Pi_NH_{n+1,n}(\Pi_N\Psi^N;C_N)\|_{L_T^2H^{\sigma-1}}^2\right]\rho_N(\dd\psi \dd\phi) \le C_T.
$$
By using Proposition 7 of \cite{FHI} (the invariance of the law of $\Psi^N(t)$ under $\rho_N$) and Lemma 5.1 of \cite{FHI} 
($N$-uniform $L^{2}(\dd\mu)$-bounds of $\dd\rho_N/\dd\mu$), we have
\begin{align*}
&\sup_{N\in\mathbb{N}}\int_{\mathcal{H}^{-\delta}}
\mathbb{E}\left[\|\Pi_NH_{n+1,n}(\Pi_N\Psi^N;C_N)\|_{L_T^2H^{\sigma-1}}^2\right]\rho_N(\dd\psi \dd\phi)\\
&=T\sup_{N\in\mathbb{N}}\int_{\mathcal{H}^{-\delta}}
\mathbb{E}\left[\|\Pi_NH_{n+1,n}(\Pi_N\psi;C_N)\|_{H^{\sigma-1}}^2\right]\rho_N(\dd\psi \dd\phi)\\
&\le T\sup_{N\in\mathbb{N}}
\sqrt{\int_{\mathcal{H}^{-\delta}}\mathbb{E}\left[\|\Pi_NH_{n+1,n}(\Pi_N\psi;C_N)\|_{H^{\sigma-1}}^2\right]^2\mu(\dd\psi \dd\phi)}
\sqrt{\int_{\mathcal{H}^{-\delta}}\frac{\dd\rho_N}{\dd\mu}(\psi)^2\mu(\dd\psi \dd\phi)}\\
&\lesssim T\sup_{N\in\N}
\sqrt{\int_{H^{-\delta}}\mathbb{E}\left[\|H_{n+1,n}(\Pi_N\psi;C_N)\|_{H^{-\delta}}^4\right]\mu_0(\dd\psi)}.
\end{align*}
The last quantity is finite because of Proposition 11 of \cite{FHI}.
\end{proof}

We replace Theorem 4 of \cite{FHI} with the following. 

\begin{thm} Fix any $T>0$. Then there exists a constant $C_T>0$ such that  
\begin{align*}
\int_{\mathcal{H}^{-\delta}}
\mathbb{E}\left[\sup_{0\le t< T\wedge T^*}\|(U(t),\varepsilon \partial_tU(t))\|_{\mathcal{H}^\sigma}\right]\rho(\dd\psi \dd\phi) \le C_T.
\end{align*}
Hence $T^*=\infty$ a.s. for $\rho$-a.e. $(\psi,\phi)\in\mathcal{H}^{-\delta}$. 
\end{thm}

\begin{proof}
The above estimate follows from the similar argument to the former part in the proof of Theorem 4 of \cite{FHI} -- the convergence $(U^N,\varepsilon\partial_tU^N)\to(U,\varepsilon\partial_tU)$ in $C([0,T^*);\mathcal{H}^{\sigma})$ in probability and Fatou's lemma.
This estimate implies that there exists a $\rho$-measurable set $\mathcal{M}_T$ such that $\rho(\mathcal{M}_T)=1$ and that for $(\psi, \phi) \in \mathcal{M}_T$, the solution $U$ exists up to time $T\wedge T^*$ and
$$
\sup_{0\le t< T\wedge T^*}\|(U(t),\varepsilon\partial_tU(t))\|_{\mathcal{H}^\sigma}<\infty
$$
almost surely.
Since $(U,\varepsilon\partial_tU)$ is locally well-posed for \emph{all} initial values $(u,v)\in\mathcal{H}^\sigma$ (Corollary 2.1 of \cite{FHI}), we have
$$
T^*<T\quad \Rightarrow\quad \lim_{t\to T^*}\sup_{0\le s< t}\|(U(s),\varepsilon\partial_sU(s))\|_{\mathcal{H}^\sigma}=\infty.
$$
If not, $(U,\varepsilon\partial_tU)$ would be extended beyond $T^*$. Hence $T^*\ge T$ a.s. for $\rho$-a.e. $(\psi,\phi)\in\mathcal{H}^{-\delta}$.
\end{proof}

Corollary 5.3 of \cite{FHI} (the invariance of the law of $(\Psi(t),\Phi(t))$ under $\rho$) still holds.

\bigskip

%%%%%%%%%%%%%%%%%%%%%  Acknowledgement %%%%%%%%%%%%%%%%%%%%%%%%%%%%%%%%%%%%%%%%%%
\noindent
{\bf Acknowledgement}. The authors thank Hirotatsu Nagoji for pointing out the mistake on the original paper.
%%%%%%%%%%%%%%%%%%%%%%%% References %%%%%%%%%%%%%%%%%%%%%%%%%%%%%%%%%%%%%%%%%%

\end{document}